\documentclass[12pt]{amsart}
\usepackage{SSdefn}

\newcommand{\gen}{{\rm gen}}
\newcommand{\alg}{{\rm alg}}
\newcommand{\fgen}{{\rm fg}}
\newcommand{\dfg}{{\rm dfg}}

\newcommand{\sExt}{\mathcal{E}\mathrm{xt}}
\newcommand{\sHom}{\mathcal{H}\mathrm{om}}
\newcommand{\cdim}{\mathrm{cdim}}

\newcommand{\CoMod}{\mathrm{CoMod}}
\newcommand{\FI}{\mathbf{FI}}
\newcommand{\Fl}{\mathbf{Fl}}

\DeclareMathOperator{\injdim}{inj.dim.}
\DeclareMathOperator{\II}{Ind.Inj.}

\DeclareMathOperator{\usHom}{\underline{\cH{}om}}

\author{Steven V Sam}
\address{Department of Mathematics, University of California, San Diego, CA}
\email{\href{mailto:ssam@ucsd.edu}{ssam@ucsd.edu}}
\urladdr{\url{http://math.ucsd.edu/~ssam/}}
\thanks{SS was supported by a Miller research fellowship and NSF grant DMS-1500069.}

\author{Andrew Snowden}
\address{Department of Mathematics, University of Michigan, Ann Arbor, MI}
\email{\href{mailto:asnowden@umich.edu}{asnowden@umich.edu}}
\urladdr{\url{http://www-personal.umich.edu/~asnowden/}}
\thanks{AS was supported by NSF grants DMS-1303082 and DMS-1453893 and a Sloan Fellowship.}

\title[GL-equivariant modules II]{GL-equivariant modules over polynomial rings\\ in infinitely many variables. II}

\subjclass[2010]{
13A50, %Actions of groups on commutative rings; invariant theory
13C05, %Structure, classification theorems
13D02, %Syzygies, resolutions, complexes
14M15, %Grassmannians, Schubert varieties, flag manifolds
05E05%Symmetric functions and generalizations
}

\date{October 18, 2018}

\begin{document}

\begin{abstract}
Twisted commutative algebras (tca's) have played an important role in the nascent field of representation stability. Let $A_d$ be the tca freely generated by $d$ indeterminates of degree~1. In a previous paper, we determined the structure of the category of $A_1$-modules (which is equivalent to the category of $\FI$-modules). In this paper, we establish analogous results for the category of $A_d$-modules, for any $d$. Modules over $A_d$ are closely related to the structures used by the authors in previous works studying syzygies of Segre and Veronese embeddings, and we hope the results of this paper will eventually lead to improvements on those works. Our results also have implications in asymptotic commutative algebra.
\end{abstract}

\maketitle
\tableofcontents

\section{Introduction}

In recent years, twisted commutative algebras (tca's) have played an important role in the nascent field of representation stability. The best known example is the twisted commutative algebra $\Sym(\bC\langle 1 \rangle)$ freely generated by a single indeterminate of degree one. Modules over this tca are equivalent to the $\FI$-modules of Church--Ellenberg--Farb \cite{fimodule}, and have received a great deal of attention. In \cite{symc1}, we studied the module theory of this tca, and established a number of fundamental structural results. The purpose of this paper is to extend these results to tca's freely generated by any number of degree one generators. This is, we believe, an important step in the development of tca theory, and connects to a number of concrete applications.

\subsection{The spectrum}

Let $A$ be the tca $\Sym(E\langle 1 \rangle)$, where $E=\bC^d$; this is the tca freely generated by $d$ elements of degree~1. We identify $A$ with the polynomial ring $\Sym(E \otimes \bC^{\infty})$ in variables $\{x_{i,j}\}_{\substack{1 \le i \le d\\ 1 \le j}}$, equipped with its natural $\GL_{\infty}$-action; $A$-modules are required to admit a compatible polynomial $\GL_{\infty}$-action. (See \S\ref{s:tca} for complete definitions.) The goal of this paper is to understand the structure of the module category $\Mod_A$ as best we can.

As a first step, we introduce the prime spectrum of a tca. This is defined similarly to the spectrum of a commutative ring, but (as far as we are concerned in this paper) is just a topological space. The spectrum of a tca gives a coarse view of its module category, so determining the spectrum is a good first step in analyzing the structure of modules.

We explicitly determine the spectrum of $A$. To state our result, we must make a definition. The {\bf total Grassmannian} of $E$, denoted $\Gr(E)$, is the following topological space. As a set, it is the disjoint union of the topological spaces $\Gr_r(E)$ for $0 \le r \le d$. (Here we are using the topological space underlying the scheme $\Gr_r(E)$ parametrizing rank $r$ quotients of $E$.) A set $Z \subset \Gr(E)$ is closed if each $Z \cap \Gr_r(E)$ is closed and moreover $Z$ is downwards closed in the sense that if a quotient $E \to U$ belongs to $Z$ (meaning the closed point of $\Gr_r(E)$ it corresponds to belongs to $Z$) then any quotient of $U$ also belongs to $Z$. We prove:

\begin{theorem} \label{thm:spec}
The spectrum of $A$ is canonically homeomorphic to $\Gr(E)$.
\end{theorem}

In the course of proving this theorem, we classify the irreducible closed subsets of $\Gr(E)$: each is the closure of a unique irreducible closed subset of $\Gr_r(E)$, for some $r$. This provides a wealth of interesting prime ideals in $A$: for example, when $d=3$ the space $\Gr_1(E)$ is $\bP^2$, and so each irreducible planar curve gives a prime ideal of $A$. This shows that for $d>1$ there is interesting geometry contained in $A$, contrary to the more rigid structure when $d=1$.

In joint work with Rohit Nagpal (which he kindly allowed us to include in this paper), we show:

\begin{theorem}[with R.~Nagpal]
The space $\Gr(E)$ has Krull dimension $\binom{d+1}{2}$.
\end{theorem}

From this, we deduce:

\begin{corollary}
The category $\Mod_A$ has Krull--Gabriel dimension $\binom{d+1}{2}$.
\end{corollary}

\subsection{Structure theory}

Let $\fa_r \subset A$ be the $r$th determinantal ideal. If we think of the variables $\{x_{i,j}\}$ as the entries of a $d \times \infty$ matrix, then $\fa_r$ is generated by the $(r+1) \times (r+1)$ minors of this matrix. Alternatively, in terms of representation theory, $\fa_r$ is generated by the representation $\lw^{r+1}(E) \otimes \lw^{r+1}(\bC^{\infty})$ occurring in the Cauchy decomposition of $A$. Let $\Mod_{A,\le r}$ be the full subcategory of $\Mod_A$ spanned by modules supported on $\fa_r$ (i.e., locally annihilated by a power of $\fa_r$). Equivalently, $\Mod_{A, \le r}$ is the category of modules whose support in $\Gr(E)$ is contained in $\bigcup_{s \le r} \Gr_s(E)$. These categories give a filtration of $\Mod_A$:
\begin{displaymath}
\Mod_{A,\le 0} \subset \Mod_{A,\le 1} \subset \cdots \subset \Mod_{A, \le d}=\Mod_A.
\end{displaymath}
We call this the {\bf rank stratification}. Let
\begin{displaymath}
\Mod_{A,r} = \frac{\Mod_{A,\le r}}{\Mod_{A, \le r-1}}
\end{displaymath}
be the Serre quotient category. Intuitively, $\Mod_{A,r}$ is the piece of $\Mod_A$ corresponding to $\Gr_r(E) \subset \Gr(E)$. Our approach to studying $\Mod_A$ is to first understand the structure of the pieces $\Mod_{A,r}$, and then understand how these pieces fit together to build $\Mod_A$.

Every object of $\Mod_{A,r}$ is locally annihilated by a power of $\fa_r$. We concentrate on the subcategory $\Mod_{A,r}[\fa_r]$ consisting of objects annihilated by $\fa_r$. The following theorem completely describes this category:

\begin{theorem}
Let $\cQ$ be the tautological bundle on $\Gr_r(E)$ and let $B$ be the tca $\Sym(\cQ\langle 1 \rangle)$ on $\Gr_r(E)$. Then $\Mod_{A,r}[\fa_r]$ is equivalent to $\Mod_{B,0}$, the category of $B$-modules locally annihilated by a power of $\cQ\langle 1 \rangle \subset B$.
\end{theorem}

Every finitely generated object of $\Mod_{A,r}$ admits a finite length filtration with graded pieces in $\Mod_{A,r}[\fa_r]$. Thus, for many purposes, the above theorem is sufficient for understanding $\Mod_{A,r}$. For example, it immediately implies:

\begin{corollary}
The Grothendieck group of $\Mod_{A,r}$ is canonically isomorphic to $\Lambda \otimes \rK(\Gr_r(E))$, where $\Lambda$ is the ring of symmetric functions, and thus is free of rank $\binom{d}{r}$ over $\Lambda$.
\end{corollary}

We now describe how $\Mod_A$ is built from its graded pieces. For this we introduce two functors. Let $M$ be an $A$-module. We define $\Gamma_{\le r}(M)$ to be the maximal submodule of $M$ supported on $\fa_r$, and we define $\Sigma_{>r}(M)$ to be the universal module to which $M$ maps that has no non-zero submodule supported on $\fa_r$. We call $\Sigma_{>r}(M)$ the {\bf saturation} of $M$ with respect to $\fa_r$. The functor $\Sigma_{>r}$ can be identified with the composition
\begin{displaymath}
\Mod_A \to \Mod_A/\Mod_{A, \le r} \to \Mod_A,
\end{displaymath}
where the first functor is the localization functor and the second is the section functor (i.e., the right adjoint to localization). The functors $\Gamma_{\le r}$ and $\Sigma_{>r}$ are left-exact, and we consider their right derived functors. We refer to $\rR^i \Gamma_{\le r}$ as {\bf local cohomology} with respect to the ideal $\fa_r$. The most important result in this paper is the following finiteness theorem:

\begin{theorem} \label{thm:fin}
Let $M$ be a finitely generated $A$-module. Then $\rR^i \Gamma_{\le r}(M)$ and $\rR^i \Sigma_{>r}(M)$ are finitely generated for all $i$ and vanish for $i \gg 0$.
\end{theorem}

This result has a number of important corollaries. Write $\rD^b_{\fgen}$ for the bounded derived category with finitely generated cohomology groups. Write $\cD = \langle \cT_1, \ldots, \cT_n \rangle$ to indicate that the triangulated category $\cD$ admits a semi-orthogonal decomposition into subcategories $\cT_i$.

\begin{corollary}
We have a semi-orthogonal decomposition:
\begin{displaymath}
\rD^b_{\fgen}(\Mod_A) = \langle \rD^b_{\fgen}(\Mod_{A,0}), \ldots, \rD^b_{\fgen}(\Mod_{A,d}) \rangle.
\end{displaymath}
\end{corollary}

Here $\rD^b_{\fgen}(\Mod_{A,r})$ is identified with a subcategory of $\rD^b_{\fgen}(\Mod_A)$ via the functor $\Sigma_{\ge r}$. We note that without finiteness conditions, such a decomposition follows almost formally; to get the decomposition with finiteness conditions imposed requires the theorem. The functor $\rR \Gamma_{\le r}$ is essentially the projection onto the subcategory $\langle \rD(\Mod_{A,0}), \ldots, \rD(\Mod_{A,r}) \rangle$, while the functor $\rR \Sigma_{>r}$ is the projection onto $\langle \rD(\Mod_{A,r+1}), \ldots, \rD(\Mod_{A,d}) \rangle$. (This point of view explains the subscripts on these functors.) We introduce the functor $\rR \Pi_r = \rR \Sigma_{\ge r} \circ \rR \Gamma_{\le r}$, which projects onto $\rD(\Mod_{A,r})$.

\begin{corollary}
We have a canonical isomorphism
\begin{displaymath}
\rK(\Mod_A) = \bigoplus_{r=0}^d \rK(\Mod_{A,r})
\end{displaymath}
The projection onto the $r$th factor is given by $\rR \Pi_r$. In particular, $\rK(\Mod_A)$ is free of rank $2^d$ as a $\Lambda$-module.
\end{corollary}

Finally, we prove a structure theorem for $\rD^b_{\fgen}(\Mod_A)$ that refines the above corollary. For an integer $0 \le r \le d$, let $P(r)$ denote the set of partitions $\lambda$ contained in the $r \times (d-r)$ rectangle (i.e., $\lambda_1 \le d-r$ and $\ell(\lambda) \le r$). For $\lambda \in P(r)$, put
\begin{displaymath}
K_{r,\lambda} = \rH^0(\Gr_r(E), \bS_{\lambda}(\cQ) \otimes \Sym(\cQ\langle 1 \rangle)),
\end{displaymath}
where $\cQ$ is the tautological bundle on $\Gr_r(E)$. Alternatively, $K_{r,\lambda}$ is the quotient of $\bS_{\lambda}(E) \otimes A$ by the ideal spanned by those copies of $\bS_{\mu}(E)$ where $\mu$ has more than $r$ parts. The classes $[\bS_{\lambda}(\cQ)]$ with $\lambda \in P(r)$ form a $\bZ$-basis for $\rK(\Gr_r(E))$, while the classes $[K_{r,\lambda}]$ form a $\Lambda$-basis for $\rK(\Mod_A)$. Our structure theorem is:

\begin{theorem} \label{thm:struc}
The objects $\bS_{\mu}(\bC^{\infty}) \otimes K_{r,\lambda}$, with $\mu$ arbitrary and $\lambda \in P(r)$, generate $\rD^b_{\fgen}(\Mod_A)$, in the sense of triangulated categories. Thus every object of $\rD^b_{\fgen}(\Mod_A)$ admits a finite filtration where the graded pieces are shifts of modules of this form.
\end{theorem}

In fact, our results are more precise than this: for instance, we show that the $K$'s all appear in a certain order, with the $K_{0,\ast}$'s first, then the $K_{1,\ast}$'s, and so on. See Remark~\ref{rmk:proof-thm:struc} for a proof.

When $d=1$, we showed in \cite{symc1} that every object $M$ of $\rD^b_{\fgen}(A)$ fits into an exact triangle
\begin{displaymath}
T \to M \to P \to
\end{displaymath}
where $T$ is a finite length complex of finitely generated torsion modules and $P$ is a finite length complex of finitely generated projective modules. A finitely generated torsion module admits a finite filtration where the graded pieces have the form $\bS_{\mu}(\bC^{\infty}) \otimes K_{0,\emptyset}$, while a finitely generated projective module admits a finite filtration where the graded pieces have the form $\bS_{\mu}(\bC^{\infty}) \otimes K_{1,\emptyset}$. (We note that $K_{0,\emptyset}=\bC$ and $K_{1,\emptyset}=A$.) Thus Theorem~\ref{thm:struc} is essentially a generalization of the structure theorem from \cite{symc1}.

We prove several other results about the structure of $\Mod_A$. We mention a few here:
\begin{itemize}
\item The extremal pieces of the rank stratification $\Mod_{A,0}$ and $\Mod_{A,d}$ are equivalent.
\item Projective $A$-modules are injective.
\item Finitely generated $A$-modules have finite injective dimension.
\end{itemize}
Lest the reader extrapolate too far, we offer two warnings: (a) For $d=1$, every finitely generated $A$-module injects into a finitely generated injective $A$-module. This is no longer true for $d>1$. (b) We believe that $\Mod_{A,r}$ and $\Mod_{A,d-r}$ are inequivalent for $r \ne 0,d$, though we do not have a rigorous proof of this.

\subsection{Duality}

Koszul duality gives an equivalence between the derived category of $A=\Sym(E \otimes \bC^{\infty})$ modules and the derived category of $\lw(E^* \otimes \bC^{\infty})$ modules (assuming some finiteness). The category of polynomial representations of $\GL_{\infty}$ has a transpose functor, which induces an equivalence between $\lw(E^* \otimes \bC^{\infty})$ modules and $A^*=\Sym(E^* \otimes \bC^{\infty})$ modules. We call the resulting equivalence
\begin{displaymath}
\sF \colon \rD_{\dfg}(\Mod_A)^{\op} \to \rD_{\dfg}(\Mod_{A^*})
\end{displaymath}
the {\bf Fourier transform}. (Here the ``dfg'' subscript means the $\GL_{\infty}$-multiplicity space of each cohomology sheaf is coherent.) Our main result on it is:

\begin{theorem} \label{thm:ft}
The Fourier transform induces an equivalence between $\rD^b_{\fgen}(\Mod_A)$ and $\rD^b_{\fgen}(\Mod_{A^*})$.
\end{theorem}

This theorem can be unpackaged into a much more concrete statement. Let $M$ be an $A$-module, and let $P_{\bullet} \to M$ be its minimal projective resolution. Write $P_i=A \otimes V_i$, where $V_i$ is a representation of $\GL_{\infty}$, and let $V_{i,n}$ be the degree $n$ piece of $V_i$. Then $L_n = \bigoplus_{i \ge 0} V_{i,n+i}$ is called the $n$th {\bf linear strand} of the resolution. Up to a duality and transpose, $L_n$ is $\rH^n(\sF(M))$. Thus the above theorem implies that if $M$ is a finitely generated $A$-module then its resolution has only finitely many non-zero linear strands, and each linear strand (after applying duality and transpose) admits the structure of a finitely generated $A^*$-module. Thus Theorem~\ref{thm:ft} is a strong statement about the structure of projective resolutions of $A$-modules. In particular, it implies:

\begin{corollary}
A finitely generated $A$-module has finite regularity.
\end{corollary}

We also prove a duality theorem for local cohomology and saturation with respect to the Fourier transform. We just mention the following version of this result here:

\begin{theorem}
We have $\sF \circ \rR \Pi_r = \rR \Pi_{d-r} \circ \sF$.
\end{theorem}

In other words, the Fourier transform reverses the rank stratification of $\Mod_A$.

\subsection{Additional results}

The results of this paper are of a foundational nature. We have additional, more concrete results, that build on this foundation; for reasons of length, we have deferred them to companion papers \cite{tcahilb,tcadepth,tcareg}. We summarize the main results here.

The first group of results concerns Hilbert series. In \cite{delta}, the second author introduced a notion of Hilbert series for twisted commutative algebras and their modules, and proved a rationality result for the tca's considered in this paper. In \cite{symc1}, we introduced an ``enhanced'' Hilbert series that records much more information, and proved a rationality result in the $d=1$ case. Using the tools of this paper, we have greatly extended this theory. We can now prove a rationality result for the enhanced Hilbert series for arbitrary $d$. Moreover, we understand how the pieces of the Hilbert series match up with the structure of the category $\Mod_A$. We have similar results on the far more subtle Poincar\'e series as well.

The second group of results concerns depth and local cohomology. Suppose that $M$ is an $A$-module. We can then consider the local cohomology group $\rR^i \Gamma_{\le r}(M)$ defined in this paper and, treating it as a polynomial functor, evaluate on $\bC^n$. Alternatively, we can take the local cohomology of the $A(\bC^n)$-module $M(\bC^n)$ with respect to the ideal $\fa_r(\bC^n)$. We show that these two constructions are canonically isomorphic for $n \gg 0$ when $M$ is finitely generated. In particular, this shows that the local cohomology of $M(\bC^n)$ with respect to determinantal ideals is finitely generated for $n \gg 0$, and exhibits representation stability in the sense of Church--Farb \cite{churchfarb}. We also study the depth of $M(\bC^n)$ with respect to $\fa_r(\bC^n)$ and show that, for $n \gg 0$, it has the form $an+b$ for integers $a \ge 0$ and $b$. Moreover, we show that if $a > 0$ then $\rR \Gamma_{\le r}(M)=0$, and if $a=0$ then the first non-zero local cohomology $\rR^i \Gamma_{\le r}(M)$ occurs for $i=b$.

The third group of results concerns regularity. In \cite{ce}, Church and Ellenberg show that the regularity of an $\FI$-module can be controlled in terms of its presentation. We generalize this result to arbitrary $d$. Our theorem states that the regularity of a finitely generated $A$-module $M$ can be controlled in terms of $\Tor^A_i(M, \bC)$ for $0 \le i \le \tfrac{1}{4} d^2+1$. As a corollary, we find that the regularity of the $A(\bC^n)$-module $M(\bC^n)$, for any $n$, can be controlled by the regularity of $M(\bC^{n_0})$, where $n_0$ depends only on $d$ and the degrees of generators of $M$.

\subsection{Relation to previous work}

The second author used the tca's appearing in this paper to study $\Delta$-modules, which served as the primary tool in his study of syzygies of the Segre embeddings and related varieties \cite{delta}. In \cite{catgb}, we showed that the category of $A$-modules is equivalent to the category of $\FI_d$-modules, where $\FI_d$ is the category whose objects are finite sets and whose morphisms are injections together with a $d$-coloring on the complement of the image. Ramos further studied $\FI_d$-modules in \cite{ramos}, and recently used them to study configuration spaces of graphs in \cite{ramos2}. $\FI_d$-modules are also used in the first author's study of equations and syzygies of secant varieties of Veronese embeddings \cite{secver, secver-syzygy}, where they play a crucial role. We hope that the results of this paper will lead to additional insight related to the applications mentioned here.

The equivariant structure of the ring $A$ has been intensively studied in the literature from combinatorial and algebraic perspectives, and we refer the reader to \cite{dcep} for some background and additional references. The homological aspects of this ring were shown to be closely related to the representation theory of the general linear Lie superalgebra in \cite{akin-weyman}, and this motivates the study of resolutions of its equivariant ideals. We refer the reader to \cite{raicu-weyman1, raicu-weyman2, raicu} for further information and calculations. Our results imply that one can expect certain patterns and universal bounds to appear as the size of the matrix increases.

\subsection{Outline}

In \S \ref{s:tca}, we recall the requisite background on the representation theory of $\GL_{\infty}$ and tca's, and prove some general results about tca's. In \S \ref{sec:spec}, we introduce the spectrum of a tca and study the spectrum of $A$. In \S \ref{s:formalism}, we develop a formalism of local cohomology and saturation functors with respect to a filtration of an abelian category. These results are mostly well-known; we include this material simply to recall salient facts and set notation. In \S \ref{s:at0}, we study the two extremal pieces of the filtration of $\Mod_A$, namely the category $\Mod_{A,0}$ of modules supported at~0, and what we call the ``generic category'' $\Mod_A^{\gen}$, which is just another name for $\Mod_{A,d}$. These are important special cases since the other pieces of the category will be described using these pieces. In \S \ref{s:rank}, we study the full rank stratification of $\Mod_A$, and prove the primary theorems of the paper. In \S \ref{s:koszul}, we treat Koszul duality and develop the theory of the Fourier transform. We also include two appendices: Appendix~\ref{s:grass} proves some well-known results about Grassmannians for which we could not find a suitable reference, and Appendix~\ref{ss:oldkoszul} gives a different, more direct, proof of the finiteness properties of Koszul duality.

\subsection{Notation and terminology}

\begin{itemize}
\item All schemes in this paper are noetherian, of finite Krull dimension, and separated over $\bC$. For a scheme $X$, we use the term ``$\cO_X$-module'' in place of ``quasi-coherent $\cO_X$-module,'' and we use the term ``finitely generated $\cO_X$-module'' in place of ``coherent $\cO_X$-module.'' We write $\Mod_X$ for the category of $\cO_X$-modules. $|X|$ denotes the underlying topological space of $X$.
\item For a vector bundle $\cE$ over a scheme $X$, we write $\Gr_r(\cE)$ for the relative Grassmannian parametrizing rank $r$ quotients of $\cE$. We often write $Y$ for $\Gr_r(\cE)$. We write $\cQ$ for the tautological quotient bundle on $\Gr_r(\cE)$ and $\cR$ for the subbundle.
\item We let $\bV=\bC^{\infty}$ be the standard representation of $\GL_{\infty}$. We write $\bS_{\lambda}$ for the Schur functor associated to the partition $\lambda$.
\item For a vector bundle $\cE$ on a scheme $X$, we let $\bA(\cE)$ be the tca $\Sym(\cE\langle 1 \rangle)=\Sym(\cE \otimes \bV)$. We let $\fa_r \subset \bA(\cE)$ be the $r$th determinantal ideal.
\item For an abelian category $\cA$ (typically Grothendieck), we write $\cA^{\fgen}$ for the category of finitely generated objects in $\cA$. We write $\rD(\cA)$ for the derived category, $\rD^b(\cA)$ for the bounded derived category, $\rD^+(\cA)$ for the bounded below derived category, and $\rD_{\fgen}(\cA)$ for the subcategory of the derived category on objects with finitely generated cohomologies. We always use cochain complexes and cohomological indexing.
\item If $\ast$ is an object for which $\Mod_{\ast}$ is defined (and locally noetherian), we write $\rK(\ast)$ for the Grothendieck group of the category $\Mod_{\ast}^{\fgen}$. In particular, if $X$ is a noetherian scheme then $\rK(X)$ is the Grothendieck group of coherent sheaves, and if $A=\bA(\cE)$ then $\rK(A)$ is the Grothendieck group of the category of finitely generated $A$-modules.
\end{itemize}

\subsection*{Acknowledgements}

We thank Rohit Nagpal for helpful discussions, and for allowing us to include the joint material appearing in \S \ref{ss:rohit}.

\section{Preliminaries on tca's} \label{s:tca}

\subsection{Polynomial representations}

A representation of $\GL_{\infty}=\bigcup_{n \ge 1} \GL_n(\bC)$ is {\bf polynomial} if it occurs as a subquotient of a direct sum of tensor powers of the standard representation $\bV=\bC^{\infty}=\bigcup_{n \ge 1} \bC^n$. Let $\cV$ be the category of such representations. Equivalently, $\cV$ can be described as the category of polynomial functors, and this will be a perspective we often employ (see \cite{expos} for details). The category $\cV$ is semi-simple abelian, and the simple objects are the representations $\bS_{\lambda}(\bV)$ indexed by partitions $\lambda$. From the perspective of polynomial functors, the simple objects are just the Schur functors $\bS_{\lambda}$. The category $\cV$ is closed under tensor product. The tensor product of simple objects is computed using the Littlewood--Richardson rule.

Every object $V$ of $\cV$ admits a decomposition $V=\bigoplus_{\lambda} V_{\lambda} \otimes \bS_{\lambda}(\bV)$ where $V_{\lambda}$ is a vector space. We refer to $\bS_{\lambda}(\bV) \otimes V_{\lambda}$ as the {\bf $\lambda$-isotypic piece} of $V$, and to $V_{\lambda}$ as the {\bf $\lambda$-multiplicity space} of $V$. We let $V_n=\bigoplus_{\vert \lambda \vert=n} V_{\lambda} \otimes \bS_{\lambda}(\bV)$, and call this the degree $n$ piece of $V$; in this way, every object of $\cV$ is canonically graded. We say that $\lambda$ {\bf occurs} in $V$ if $V_{\lambda} \ne 0$. For a partition $\lambda$, we let $\ell(\lambda)$ be the number of non-zero parts in $\lambda$. We let $\ell(V)$ be the supremum of the $\ell(\lambda)$ over those $\lambda$ that occur in $V$, and we say that $V$ is {\bf bounded} if $\ell(V)<\infty$. We have $\ell(V \otimes W)=\ell(V)+\ell(W)$ by the Littlewood--Richardson rule; in particular, a tensor product of bounded representations is bounded.

Let $\cV_{\le n}$ be the full subcategory of $\cV$ on objects $V$ with $\ell(V) \le n$. The functor $\cV_{\le n} \to \Rep(\GL_n)$ given by $V \mapsto V(\bC^n)$ is fully faithful, and its image consists of all polynomial representations of $\GL_n$. This is an extremely important fact, since it implies that in $\cV_{\le n}$ one can evaluate on $\bC^n$---and thus reduce to a familiar finite dimensional setting---without losing information.

Let $\Rep(S_{\ast})$ be the category whose objects are sequences $(M_n)_{n \ge 0}$, where $M_n$ is a representation of the symmetric group $S_n$. Schur--Weyl duality provides an equivalence between $\Rep(S_{\ast})$ and $\cV$; see \cite{expos} for details. This perspective will appear in a few places in this paper.

Suppose that $X$ is a scheme over $\bC$. We then let $\cV_X$ be the category of polynomial representations of $\GL_{\infty}$ on $\cO_X$-modules. Every object of this category $V$ admits a decomposition $V=\bigoplus_{\lambda} \bS_{\lambda}(\bV) \otimes V_{\lambda}$ where $V_{\lambda}$ is an $\cO_X$-module. If $f \colon Y \to X$ is a map of schemes then there are induced functors $f_* \colon \cV_Y \to \cV_X$ and $f^* \colon \cV_X \to \cV_Y$ computed by applying $f_*$ and $f^*$ to the multiplicity spaces. We also have the derived functors $\rR^i f_* \colon \cV_Y \to \cV_X$, computed by applying $\rR^i f_*$ to the multiplicity spaces.

\subsection{Twisted commutative algebras}

For the purposes of this paper, a {\bf twisted commutative algebra} (tca) is a commutative algebra object in the category $\cV$, or more generally, in $\cV_X$ for some scheme $X$. Explicitly, a tca is a commutative associative unital $\bC$-algebra equipped with an action of $\GL_{\infty}$ by algebra automorphisms, under which it forms a polynomial representation. A {\bf module} over a tca $A$ is a module object in the category $\cV$ (or $\cV_X$), that is, an $A$-module equipped with a compatible $\GL_{\infty}$ action under which it forms a polynomial representation. We write $\Mod_A$ for the category of $A$-modules. This is a Grothendieck abelian category. An {\bf ideal} of $A$ is an $A$-submodule of $A$. If $M$ is an $A$-module then, treating $M$ and $A$ as Schur functors, $M(\bC^n)$ is an $A(\bC^n)$-module with a compatible action of $\GL_n$.

Let $\cE$ be a vector bundle of rank $d$ on $X$. We define $A=\bA(\cE)$ to be the tca $\Sym(\bV \otimes \cE)$ on $X$. As a Schur functor, we have $A(\bC^n)=\Sym(\bC^n \otimes \cE)$. In particular, if $X$ is a point then $A(\bC^n)$ is just a polynomial ring in $nd$ variables over $\bC$. The Cauchy formula gives a decomposition
\begin{displaymath}
A = \bigoplus_{\lambda} \bS_{\lambda}(\cE) \otimes \bS_{\lambda}(\bV).
\end{displaymath}
Since $\bS_{\lambda}(\cE)=0$ if $\lambda$ has more than $d$ rows, we see that $\ell(A)=d$. Thus $A$ is bounded. It follows that any finitely generated $A$-module is bounded, as such a module is a quotient of $A \otimes V$ for some finite length (and thus bounded) object $V$ of $\cV_X$. We recall the following well-known result (first proved in \cite{delta}):

\begin{theorem}
The tca $A$ is noetherian, that is, any submodule of a finitely generated module is finitely generated.
\end{theorem}

\begin{proof}
Suppose $M$ is a finitely generated $A$-module, and consider an ascending chain $N_{\bullet}$ of $A$-submodules of $M$. Let $n=\ell(M)$, which is finite by the above remarks; of course, $\ell(N_i) \le n$ for all $i$ as well. Since $M(\bC^n)$ is a finitely generated $A(\bC^n)$-module, it is noetherian, as $A(\bC^n)$ is a finitely generated over $\cO_X$. Thus the chain $N_{\bullet}(\bC^n)$ stabilizes, which implies that $N_{\bullet}$ stabilizes.
\end{proof}

We let $\fa_r \subset A$ be the $r$th determinantal ideal of $A$; it is generated by $\lw^{r+1}(\cE) \otimes \lw^{r+1}(\bV) \subset A$. The tca $A$ and its ideals $\fa_r$ are the main focus of this paper.

\subsection{Internal Hom} \label{ss:internalhom}

We let $\sHom$ be the internal Hom in the category of $\cO_X$-modules. For $V,W \in \cV_X$, we let $\sHom(V,W)$ be the sheaf of $\GL$-equivariant homomorphisms. Explicitly,
\begin{displaymath}
\sHom(V,W) = \prod_{\lambda} \sHom(V_{\lambda}, W_{\lambda}).
\end{displaymath}
We define the internal $\Hom$ on $\cV_X$ by
\begin{displaymath}
\usHom(V,W) = \bigoplus_{\lambda} \sHom(V \otimes \bS_{\lambda}(\bV), W) \otimes \bS_{\lambda}(\bV).
\end{displaymath}
This is again an object of $\cV_X$. We have the adjunction
\begin{displaymath}
\sHom(U, \usHom(V, W))=\sHom(U \otimes V, W).
\end{displaymath}
The trivial multiplicity space in $\usHom(V,W)$ is $\sHom(V,W)$. When $X$ is affine, we write $\uHom$ in place of $\usHom$.

\begin{proposition}
Suppose $X$ is a point. Let $V \in \cV$ and let $\wt{V} \in \Rep(S_{\ast})$ be its Schur--Weyl dual. Then
\begin{displaymath}
\uHom(V, \bV^{\otimes n}) = \bigoplus_{i+j=n} \Ind_{S_i \times S_j}^{S_n} (\wt{V}_i^* \otimes \bV^{\otimes j}).
\end{displaymath}
\end{proposition}

\begin{proof}
The coefficient of $\bS_{\lambda}(\bV)$ in $\uHom(V, \bV^{\otimes n})$ is $\Hom(V \otimes \bS_{\lambda}(\bV), \bV^{\otimes n})$. We can compute this $\Hom$ space after applying Schur--Weyl duality. Schur--Weyl converts $\bV^{\otimes n}$ to $\bC[S_n]$, the regular representation in degree $n$, and converts $V \otimes \bS_{\lambda}(\bV)$ to $\bigoplus_i \Ind_{S_i \times S_j}^{S_{i+j}} (\wt{V}_i \otimes \bM_{\lambda})$, where $j=\vert \lambda \vert$. We thus find
\begin{displaymath}
\uHom(V, \bV^{\otimes n}) = \bigoplus_{i,j} \bigoplus_{\vert \lambda \vert=j} \Hom(\Ind_{S_i \times S_j}^{S_{i+j}}(\wt{V}_i \otimes \bM_{\lambda}), \bC[S_n]) \otimes \bS_{\lambda}(\bV).
\end{displaymath}
Since $\bC[S_n]$ is concentrated in degree $n$, only the terms with $i+j=n$ contribute. For an $S_n$-representation $W$, we have $\Hom_{S_n}(W, \bC[S_n])=W^*$. Thus, via the canonical auto-duality of $\bM_{\lambda}$, the above becomes
\begin{displaymath}
\bigoplus_{i+j=n} \bigoplus_{\vert \lambda \vert=j} \Ind_{S_i \times S_j}^{S_n}(\wt{V}_i^* \otimes \bM_{\lambda}) \otimes \bS_{\lambda}(\bV).
\end{displaymath}
Using the formula $\bV^{\otimes j} = \bigoplus_{\vert \lambda \vert=j}(\bM_{\lambda} \otimes \bS_{\lambda}(\bV))$, the result follows.
\end{proof}

Let $A=\bA(\cE)$. Suppose that $M$ and $N$ are $A$-modules. Then $M \otimes_{\cO_X} N$ is naturally an $A^{\otimes 2}$-module, and thus an $A$-module via the comultiplication map $A \to A^{\otimes 2}$. We denote this $A$-module by $M \odot N$. The operation $\odot$ endows $\Mod_A$ with a new symmetric tensor product. In general, this operation does not preserve finiteness properties of $M$ and $N$. However, if $M$ and $N$ are finitely generated and annihilated by a power of the maximal ideal of $A$ then then $M \odot N$ is again finitely generated and annihilated by a power of the maximal ideal.

\begin{proposition} \label{prop:uhom-ten}
The functor $\cV_X \to \Mod_A$ given by $V \mapsto \usHom(A, V)$ is a tensor functor, using the $\odot$ tensor product on $\Mod_A$.
\end{proposition}

\addtocounter{equation}{-1}
\begin{subequations}
\begin{proof}
Let $V,W \in \cV_X$. We have canonical maps
\begin{equation} \label{eq:uhom-ten}
\usHom(A, V) \odot \usHom(A, W) \to \usHom(A \otimes A, V \otimes W) \to \usHom(A, V \otimes W),
\end{equation}
We show that this map is an isomorphism. It suffices to work Zariski locally on $X$, so we may assume $\cE$ is trivial. Since both sides are bi-additive in $V$ and $W$, it suffices to treat the case where each has the form $\cF \otimes \bS_{\lambda}(\bV)$, where $\cF$ is an $\cO_X$-module. But $\cO_X$-modules pull out of these $\Hom$'s, and so we may as well assume $\cF=\cO_X$. The map in question is then pulled back from a point. It thus suffices to treat the case where $X$ is a point and $V$ and $W$ are irreducible; we write $E$ in place of $\cE$. We make one more reduction: instead of taking $V$ and $W$ irreducible, we can assume each is a tensor power of the standard representation, since every irreducible is a summand of such a tensor power.

Let $\wt{A}$ be the Schur--Weyl dual of $A$. We have $\wt{A}_n=E^{\otimes n}$. Thus
\begin{displaymath}
\uHom(A, \bV^{\otimes n})=\bigoplus_{i+j=n} \Ind_{S_i \times S_j}^{S_n} ((E^*)^{\otimes i} \otimes \bV^{\otimes j}) = (E^* \oplus \bV)^{\otimes n}.
\end{displaymath}
Thus with $V=\bV^{\otimes n}$ and $W=\bV^{\otimes m}$, the map \eqref{eq:uhom-ten} takes the form
\begin{displaymath}
(E^* \oplus \bV)^{\otimes n} \otimes (E^* \oplus \bV)^{\otimes m} \to (E^* \oplus \bV)^{\otimes (n+m)}.
\end{displaymath}
We can thus regard it as an endomorphism of the target. By adjunction, to give a map of $A$-modules $M \to \uHom(A, \bV^{\otimes (n+m)})$ is the same as to give a map $M_{n+m} \to \bV^{\otimes (n+m)}$; in particular, an endomorphism of $\uHom(A, \bV^{\otimes (n+m)})$ is an isomorphism if and only if it is so in degree $n+m$. Thus, to prove that the above map is an isomorphism, it suffices to prove that it is an isomorphism in degree $n+m$. Now, the degree $n+m$ piece of each side is obtained by replacing $A$ with $\bC$ in \eqref{eq:uhom-ten}. This map is clearly an isomorphism, and so the proof is complete.
\end{proof}
\end{subequations}

\subsection{Some remarks on injective objects}

We let $\sExt$ be the sheaf version of $\Ext$ for $\cO_X$-modules. 

\begin{proposition} \label{prop:ext-specialize}
Let $A$ be a tca over $X$ and let $M$ and $N$ be $A$-modules with $n=\ell(M)<\infty$. Then there is a natural isomorphism
\begin{displaymath}
\sExt^i_A(N, M) \to \sExt^i_{|A(\bC^n)|}(N(\bC^n), M(\bC^n))^{\GL_n}.
\end{displaymath}
\end{proposition}

Let us clarify one point here: $|A(\bC^n)|$ is the underlying algebra of $A(\bC^n)$ without any equivariance issues. Hence, $\sExt_A$ deals with $\GL$-equivariant extensions of the algebra $A$, while $\sExt_{|A(\bC^n)|}$ deals with extensions of the underlying algebra $A(\bC^n)$. The latter space carries an action of $\GL_n$.

\begin{proof}
Evaluation gives a map
\[
\sExt^i_A(N,M) \to \sExt^i_{|A(\bC^n)|}(N(\bC^n), M(\bC^n))^{\GL_n}
\]
and it suffices to prove that it is an isomorphism over some affine cover, so we now assume that $X$ is affine. Let $P_{\bullet} \to N$ be a locally free resolution. Then $\sHom_A(P_{\bullet}, M)$ computes $\sExt_A^{\bullet}(N, M)$. Since $\ell(M)=n$, the natural map
\begin{displaymath}
\sHom_A(P_{\bullet}, M) \to \sHom_{|A(\bC^n)|,\GL_n}(P_{\bullet}(\bC^n), M(\bC^n))
\end{displaymath}
is an isomorphism. Note that $\sHom_{|A(\bC^n)|}(P_{\bullet}(\bC^n), M(\bC^n))$ is an algebraic representation of $\GL_n$, and thus is semi-simple as a $\GL_n$-representation, and so formation of $\GL_n$ invariants commutes with formation of cohomology. Thus the target complex computes
\begin{displaymath}
\sExt_{|A(\bC^n)|}^{\bullet}(N(\bC^n), M(\bC^n))^{\GL_n},
\end{displaymath}
and so the result follows.
\end{proof}

We write $\injdim(M)$ for the injective dimension of an object $M$ in an abelian category. For a scheme $X$, we write $\cdim(X)$ for the cohomological dimension of $X$: this is the maximum $i$ for which $\rH^i(X, -)$ is non-zero on quasi-coherent sheaves. We note that $\cdim(X) \le \dim(X)$ (Grothendieck vanishing) and $\cdim(X)=0$ if $X$ is affine (Serre vanishing).

\begin{proposition}
Suppose $X$ is regular. Let $A=\Sym(V)$, where $V \in \cV_X$ has all multiplicity spaces locally free of finite rank. If $M$ is an $A$-module with $n=\ell(M)$ then
\begin{displaymath}
\injdim(M) \le \dim_\bC(V(\bC^n))+\dim(X)+\cdim(X).
\end{displaymath}
In particular, every bounded $A$-module has finite injective dimension. $($Recall our standing assumption that $\dim(X)<\infty$.$)$
\end{proposition}

\begin{proof}
Since $\Spec(A(\bC^n))$ is regular, being an affine bundle over the smooth scheme $X$, we have $\sExt^i_{A(\bC^n)}(-,-)=0$ identically for $i> \dim(A(\bC^n))$. From the previous proposition, we thus have $\sExt^i_A(N,M)=0$ for $i>\dim(A(\bC^n))$. Next, we have a local-to-global spectral sequence
\[
\rE^{p,q}_2 = \rH^p(X; \sExt^q_A(N,M)) \Longrightarrow \Ext^{p+q}_A(N,M)
\]
which implies that $\Ext^i_A(N,M) = 0$ whenever $i>\dim(A(\bC^n)) + \cdim(X)$. As $\dim(A(\bC^n))=\dim(X)+\dim(V(\bC^n))$, the result follows.
\end{proof}

\begin{corollary} \label{cor:fin-inj-dim}
Suppose $X$ is regular. Let $A=\bA(\cE)$ for a vector bundle $\cE$ on $X$ of rank $d$, and let $M$ be a bounded $A$-module. Then 
\begin{displaymath}
\injdim(M) \le \ell(M) \cdot \rank(\cE) + \dim(X) + \cdim(X).
\end{displaymath}
In particular, all finitely generated $A$-modules have finite injective dimension.
\end{corollary}

Let $A$ be a tca over $X$. Let $\cC$ be the category of $A$-modules and let $\cC_{\le n}$ be the full subcategory on $A$-modules $M$ with $\ell(A) \le M$. We have an inclusion functor $\cC_{\le n} \to \cC$ and a truncation functor $\cC \to \cC_{\le n}$ mapping $M$ to $M_{\le n}$. These functors are both exact, and the inclusion functor is the right adjoint to the truncation functor. It follows that the inclusion functor preserves injectives, that is, if $I$ is an injective object of $\cC_{\le n}$ then it is also an injective object of $\cC$. From this, we deduce the following useful result:

\begin{proposition} \label{prop:bdinj}
Let $A$ be a tca and let $M$ be an $A$-module with $\ell(M) \le n$. Then there exists an injection $M \to I$ where $I$ is an injective $A$-module with $\ell(I) \le n$. 
\end{proposition}

\begin{proof}
The category $\cC_{\le n}$ is Grothendieck and therefore has enough injectives. We can thus find an injection $M \to I$ where $I$ is an injective object of $\cC_{\le n}$. By the above observation, $I$ is injective in the category of all $A$-modules.
\end{proof}

\begin{remark}
In particular, we get an injective resolution $I^{\bullet}$ of $M$ with $\ell(I^i) \le \ell(M)$ for all $i$, so we also get bounds on $\ell$ of some derived functors, like local cohomology (see \S \ref{s:rank}).
\end{remark}

\begin{proposition} \label{prop:injOX}
Let $A=\bA(\cE)$ and let $I$ be an injective $A$-module. Then the $\bS_\lambda$-isotypic component $I_{\lambda}$ of $I$ is an injective $\cO_X$-module for all $\lambda$. 
\end{proposition}

\begin{proof}
The forgetful functor $\Mod_A \to \cV_X$ is right adjoint to the exact functor $- \otimes A \colon \cV_X \to \Mod_A$, and therefore takes injectives to injectives. Thus $I$ is injective as an object of $\cV_X$. As an abelian category, $\cV_X$ is simply the product of the categories $\Mod_X$ indexed by partitions, and so the injectivity of $I$ in $\cV_X$ implies the injectivity of $I_{\lambda}$ in $\Mod_X$. 
\end{proof}

\subsection{Pushforwards}

Let $f \colon Y \to X$ be a proper map of schemes, let $\cE_X$ be a finite rank vector bundle on $X$, and let $\cE_Y=f^*(\cE_X)$ be its pullback to $Y$. Let $A_X=\bA(\cE_X)$ and $A_Y=\bA(\cE_Y)$. If $M$ is an $A_Y$-module then $\rR^i f_*(M)$ is naturally an $A_X$-module.

\begin{proposition} \label{prop:finpushfwd}
Suppose $M$ is a finitely generated $A_Y$-module. Then $\rR^i f_*(M)$ is a finitely generated $A_X$-module, for all $i \ge 0$.
\end{proposition}

\begin{proof}
We prove the result by descending induction on $i$. To begin, note that $\rR^i f_*=0$ for $i>\dim(Y)$. (Recall our assumption that $\dim(Y)$ is finite.) Now suppose the result has been proved for $i+1$ and let us prove it for $i$. Since $M$ is finitely generated, we can pick a short exact sequence
\begin{displaymath}
0 \to N \to A_Y \otimes V \to M \to 0
\end{displaymath}
where $V$ is a finitely generated object of $\cV_Y$. Note that $N$ is a finitely generated $A_Y$-module by noetherianity. From the above, we obtain an exact sequence
\begin{displaymath}
\rR^i f_*(A_Y \otimes V) \to \rR^i f_*(M) \to \rR^{i+1} f_*(N).
\end{displaymath}
Now, $\rR^{i+1} f_*(N)$ is finitely generated by induction, and so the image of $\rR^i f_*(M)$ in it is finitely generated by noetherianity. Since $A_Y=f^*(A_X)$, the projection formula gives $\rR^i f_*(A_Y \otimes V)=A_X \otimes \rR^i f_*(V)$. Since $f$ is proper, $\rR^i f_*(V)$ is a finitely generated object of $\cV_X$, and so the result follows.
\end{proof}

\begin{corollary} \label{cor:finpushfwd}
The functor $\rR f_*$ carries $\rD^b_{\fgen}(A_Y)$ into $\rD^b_{\fgen}(A_X)$.
\end{corollary}

\section{The prime spectrum} \label{sec:spec}

\subsection{The spectrum}

Let $A$ be a tca. An ideal $I \subset A$ is {\bf prime} if for any ideals $\fa,\fb$, we have that $\fa \fb \subset I$ implies $\fa \subset I$ or $\fb \subset I$. We define $\Spec(A)$ to be the set of prime ideals of $A$, and equip it with the Zariski topology. If $\ell(A) \le n$, then $I$ is prime if and only if $\vert I(\bC^n) \vert$ is a prime ideal in $\vert A(\bC^n) \vert$ (see \cite[\S\S 8.5, 8.6]{expos}, and note that ``domain'' and ``weak domain'' coincide in the bounded case). In particular, $\Spec(A)$ coincides with the set of $\GL_n$ fixed points in $\Spec(\vert A(\bC^n) \vert)$, given the subspace topology.

The spectrum of $A$ is a useful tool for obtaining a coarse picture of the module theory of $A$. In this section, we will determine the spectrum of $\bA(\cE)$, and deduce some consequences for modules.

\subsection{The total Grassmannian} \label{ss:total-grass}

Let $X$ be a scheme and let $\cE$ be a vector bundle of rank $d$ over $X$. For each $0 \le r \le d$ we have the Grassmannian $\Gr_r(\cE)$ of $r$-dimensional quotients of $\cE$, which is a scheme over $X$. To be precise, given an $X$-scheme $T \to X$, a morphism $f \colon T \to \Gr_r(\cE)$ is given by the datum of a short exact sequence $0 \to \cE_1 \to f^*(\cE) \to \cE_2 \to 0$ of locally free sheaves on $T$ such that the rank of $\cE_2$ is $r$. 

Given $0 \le r \le s \le d$, let $\Fl_{r,s}(\cE)$ be the partial flag variety parametrizing surjections $\cE \to \cQ_s \to \cQ_r$ where $\cQ_r$ has rank $r$ and $\cQ_s$ has rank $s$ (we mean this in the functor of points language as above). There are projection maps $\pi_s \colon \Fl_{r,s}(\cE) \to \Gr_s(\cE)$ and $\pi_r \colon \Fl_{r,s}(\cE) \to \Gr_r(\cE)$. Given a subset $Z$ of $\Gr_s(\cE)$, we let $Z^{(r)} \subset \Gr_r(\cE)$ be $\pi_r(\pi_s^{-1}(Z))$. If $Z$ is closed then so is $Z^{(r)}$ (since $\pi_r$ is proper), and if $Z$ is irreducible then so is $Z^{(r)}$ (since $\pi^{-1}_s(Z)$ is irreducible, as $\pi_s$ is a fiber bundle with irreducible fibers). Explicitly, $Z^{(r)}$ consists of all $r$-dimensional quotients of a space in $Z$.

We now define a topological space $\Gr(\cE)$ called the {\bf total Grassmannian} of $\cE$. As a set, $\Gr(\cE)$ is the disjoint union of the $\vert \Gr_r(\cE) \vert$ for $0 \le r \le d$, including the non-closed points. A subset $Z$ of $\Gr(\cE)$ is closed if each set $Z_r=Z \cap \Gr_r(\cE)$ is Zariski closed in $\Gr_r(\cE)$ and $Z$ is downwards closed in the sense that $Z_s^{(r)} \subset Z_r$ for all $r \le s$. There is a natural continuous map $\Gr(\cE) \to \vert X \vert$.

The discussion above gives:

\begin{lemma} \label{lem:closure}
Let $Z \subset \Gr_r(\cE)$ be closed. Then the closure of $Z$ in $\Gr(\cE)$ is the set of all quotients of a space in $Z$.
\end{lemma}

\begin{proposition} \label{prop:grirred}
Suppose $Z \subset \Gr_r(\cE)$ is Zariski closed and irreducible. Then the closure $\ol{Z}$ of $Z$ in $\Gr(\cE)$ is irreducible, and all irreducible closed subsets of $\Gr(\cE)$ are obtained in this way.
\end{proposition}

\begin{proof}
The closure of an irreducible set is irreducible, so $\ol{Z}$ is irreducible. Now suppose that $Y \subset \Gr(\cE)$ is a given irreducible set. Let $r$ be maximal so that $Y$ meets $\Gr_r(\cE)$. Then $Z=Y \cap \Gr_r(\cE)$ is a non-empty open subset of $Y$. Thus $Z$ is irreducible, and $Y$ is the closure of $Z$. Of course, $Z$ is also a closed subset of $\Gr_r(\cE)$, which completes the proof.
\end{proof}

\subsection{The spectrum of $\bA(\cE)$}

Fix a scheme $X$ and a vector bundle $\cE$ of rank $d$ on $X$. Our goal is to prove the following theorem:

\begin{theorem}
We have a canonical identification $\Spec(\bA(\cE))=\Gr(\cE)$.
\end{theorem}

In what follows, let $A=\bA(\cE)$ and $Y=\Spec(A)$. Recall that we have determinantal ideals $\fa_r \subset A$. We let $Y_{\le r}$ be the closed subset $V(\fa_r)$ of $Y$, and we let $Y_r=Y_{\le r} \setminus Y_{\le r-1}$.

\begin{lemma}
Suppose that a connected algebraic group $G$ acts freely on a scheme $X$ and that the quotient scheme $X/G$ exists. Then the natural map $\pi_0 \colon \vert X \vert^G \to \vert X/G \vert$ is a homeomorphism.
\end{lemma}

\begin{proof}
Let $\pi \colon X \to X/G$ be the quotient map. Given a point $z \in X/G$, let $Z$ be its closure, an irreducible closed subscheme of $X/G$. Then $\pi^{-1}(Z)$ is an irreducible closed subscheme of $X$ that is $G$-stable (it is irreducible because all fibers are irreducible). We define a map $\rho \colon \vert X/G \vert \to \vert X \vert^G$ by sending $z$ to the generic point of $\pi^{-1}(Z)$.

Suppose $y \in \vert X \vert^G$. Let $Y \subset X$ be the closure of $y$, an irreducible closed $G$-stable subset. Then $\pi(Y)$ is an irreducible closed subset of $X/G$ and $Y=\pi^{-1}(\pi(Y))$. Thus if $z=\pi_0(y)$ is the generic point of $\pi(Y)$ then $\rho(z)=y$. Thus $\pi_0 \circ \rho = \id$. Similarly, if $z \in X/G$ with closure $Z$ and $y=\rho(z)$ is the generic point of $Y=\pi^{-1}(Z)$ then $Z=\pi(Y)$, and so $z=\pi_0(y)$. Thus $\rho \circ \pi_0 = \id$. We therefore see that $\rho$ and $\pi_0$ are mutually inverse bijections.

Suppose now that $Y \subset \vert X \vert^G$ is a closed subset. Then $Y=\ol{Y} \cap \vert X \vert^G$, where $\ol{Y}$ is the closure of $Y$ in $\vert X \vert$. The set $Y$ is $G$-stable, and so $\pi_0(Y)=\pi(\ol{Y})$ is closed in $\vert X/G \vert$. We thus see that $\pi_0$ is a closed mapping. As $\pi_0$ is also a bijection, it is thus a homeomorphism.
% It is immediate from the definition of $\rho$ that specializations lift along $\pi_0$: if $z'$ belongs to the closure of $z$ in $\vert X/G \vert$ then $\rho(z')$ belongs to the closure of $\rho(Z)$ in $\vert X \vert^G$. It is clear that $\vert X \vert^G$ is a noetherian spectral space, and so $\pi_0$ is a homeomorphism by \cite[Tag 09XU]{stacks}.
\end{proof}

\begin{lemma}
We have a canonical homeomorphism $Y_r = \vert \Gr_r(\cE) \vert$.
\end{lemma}

\begin{proof}
The space $Y_{\le r}$ is the spectrum of the tca $A/\fa_r$, which has $\le r$ rows. Thus $Y_{\le r}$ is identified with the $\GL_r$ fixed space of $\Spec(Y_{\le r}(\bC^r))$, whose closed points are identified with the space of maps $\cE \to (\bC^r)^*$. The complement of $V(\fa_{r-1})$ is the locus where the map is surjective. The group $\GL_r$ acts freely on this locus, and the quotient is the scheme $\Gr_r(\cE)$. The lemma thus follows from the previous lemma.
\end{proof}

\begin{lemma} \label{lem:irred-Y}
Suppose $Z \subset Y_r$ is Zariski closed and irreducible. Then the closure $\ol{Z}$ of $Z$ in $Y$ is irreducible, and all irreducible closed subsets $Z'$ of $Y$ are obtained in this way, and $r$ can be recovered as the largest index such that $Z' \cap \Gr_r(\cE) \ne \emptyset$.
\end{lemma}

\begin{proof}
Same as Proposition~\ref{prop:grirred}.
\end{proof}

The homeomorphisms $Y_r \to \Gr_r(\cE)$ yield a bijective function $f \colon Y \to \Gr(\cE)$.

\begin{lemma}
The map $f$ is a homeomorphism.
\end{lemma}

\begin{proof}
Let $Z \subset \Gr_r(\cE)$ be a closed set, and let $\ol{Z} \subset \Gr(\cE)$ be its closure. Let $Z'=f^{-1}(Z)$, a closed subset of $Y_r$, and let $\ol{Z}'$ be its closure. We claim that $\ol{Z}'=f^{-1}(\ol{Z})$. It suffices to check this on each fiber of $\vert X \vert$, so we may as well assume $X$ is a single point and $\cE=E$ is a vector space. It then suffices to check on closed points after intersecting with each $Y_s$. A closed point of $\ol{Z} \cap \Gr_s(E)$ is a rank $s$ quotient $V$ of a rank $r$ quotient $U$ belonging to $Z$. By definition, there is a point in $\Spec(A(\bC^r))$, thought of as a map $f \colon E \to (\bC^r)^*$, with coimage $U$. (Recall that the coimage of $f$ is $E/\ker(f)$.) It is easy to construct a map in the orbit closure of $f$ with coimage $V$. It follows that $\ol{Z}'$ contains $f^{-1}(\ol{Z})$. For the reverse inclusion, the locus $\Spec(A(\bC^r))$ where the image is contained in $Z$ is a closed set, and so the closure of $Z'$ is contained in $f^{-1}(\ol{Z})$.

It follows from the previous paragraph that $f$ and $f^{-1}$ take closed sets to closed sets. Indeed, every closed set is a finite union of irreducible closed sets, and each irreducible closed set is the closure of an irreducible closed set in $Y_r$ or $\Gr_r(\cE)$ (Lemma~\ref{lem:irred-Y}). This completes the proof.
\end{proof}

\subsection{Krull dimension}

The next result compares the Krull dimension of $\Spec(A)$, which is typically easy to calculate, to the Krull--Gabriel dimension of the category $\Mod_A$, which is harder to compute directly. (See \cite[\S IV.1]{gabriel} for the definition of Krull--Gabriel dimension, though the following proof effectively contains a definition as well.)

\begin{proposition} \label{prop:krulldim}
Let $A$ be a noetherian tca. Suppose that the following condition holds:
\[
\label{eqn:prop-P}
\makebox{\parbox{\dimexpr\linewidth-20\fboxsep-2\fboxrule}{For any prime ideal $\fp$ of $A$, let $M$ be a finitely generated $A/\fp$-module, and let $M \supset N_0 \supset N_1 \supset \cdots$ be a descending chain. Then $N_i/N_{i-1}$ has non-zero annihilator in $A/\fp$ for all $i \gg 0$.}} \tag{P}
\]
Then the Krull--Gabriel dimension of the category $\Mod_A$ agrees with the Krull dimension of the topological space $\Spec(A)$.
\end{proposition}

\begin{proof}
Let $\cA$ be the category of finitely generated $A$-modules. Let $\cB_0$ be the category of finite length objects in $\cA$, and having defined $\cB_i$, let $\cB_{i+1}$ be the category consisting of objects in $\cA$ that become finite length in $\cA/\cB_i$. Let $\cC_i$ be the subcategory of $\cA$ on objects whose support locus in $\Spec(A)$ has Krull dimension at most $i$. We claim $\cB_i=\cC_i$. This is clear for $i=0$: a finitely generated $A$-module has finite length if and only if it is supported at the maximal ideal. Suppose now we have shown $\cB_{i-1}=\cC_{i-1}$, and let us prove $\cB_i=\cC_i$. If $M$ is in $\cC_i$ then (P) shows that $M$ has finite length in $\cA/\cC_{i-1}=\cA/\cB_{i-1}$, and so $M$ belongs to $\cB_i$.

Conversely, suppose that $M$ belongs to $\cB_i$, and let us show that $M$ belongs to $\cC_i$. We may as well suppose $M$ is simple in $\cA/\cB_{i-1}$ and contains no non-zero subobject in $\cB_{i-1}=\cC_{i-1}$. Suppose that $\fp$ is a prime ideal such that $V(\fp)$ has dimension $i$ and is contained in the support of $M$. (If no such $\fp$ exists then $M \in \cC_{i-1}$.) We claim $\fp$ annihilates $M$. Suppose not. Then $\fp M$ is a non-zero subobject of $M$ and so does not belong to $\cB_{i-1}$. Since $M$ is simple modulo $\cB_{i-1}$, it follows that $M/\fp M$ belongs to $\cB_{i-1}=\cC_{i-1}$. But this is a contradiction, since the support of $M/\fp M$ is $V(\fp)$, but $M/\fp M$ belongs to $\cC_{i-1}$ and therefore has support of dimension $<i$. We conclude that $\fp M=0$, and so $M$ has support of dimension $\le i$, and thus belongs to $\cC_i$.

To finish, let $d$ be the Krull dimension of $\Spec(A)$. Then $\cC_d=\cA$ but $\cC_{d-1} \ne \cA$. It follows that $\cB_d=\cA$ but $\cB_{d-1} \ne \cA$, and so $d$ is also the Krull--Gabriel dimension of $\Mod_A$.
\end{proof}

\begin{proposition}
Let $A$ be a finitely generated bounded tca. Then the condition \eqref{eqn:prop-P} holds.
\end{proposition}

\begin{proof}
Let $B=A/\fp$ for a prime ideal $\fp$, let $M$ be a finitely generated $B$-module, and let $N_{\bullet}$ be a descending chain. Let $n>\ell(M)+\ell(B)$. So $N_i \ne N_{i+1}$ implies that $N_i(\bC^n) \ne N_{i+1}(\bC^n)$. Then $N_i(\bC^n) \otimes_{B(\bC^n)} \Frac(B(\bC^n))$ is a descending chain of finite dimensional vector spaces, and therefore stabilizes; suppose it is stable for $i>N$. Then $N_i(\bC^n)/N_{i-1}(\bC^n)$ has non-zero annihilator in $B(\bC^n)$ for all $i>N$. It follows that $N_i/N_{i-1}$ has non-zero annihilator in $B$ for $i>N$.
\end{proof}

\begin{remark}
The condition \eqref{eqn:prop-P} can be rephrased as: for every prime ideal $\fp$, the category $\Mod_{\kappa(\fp)}$ has Krull--Gabriel dimension~0, where $\Mod_{\kappa(\fp)}$ is the quotient of $\Mod_{A/\fp}$ by the Serre subcategory of modules with non-zero annihilator. (One thinks of $\Mod_{\kappa(\fp)}$ as modules over a hypothetical residue field $\kappa(\fp)$.) Proposition~\ref{prop:krulldim} is not specific to tca's, and holds for any tensor category satisfying similar conditions.
\end{remark}

\subsection{Krull dimension of $\Sym(\cE\langle 1\rangle)$ (joint with Rohit Nagpal)} \label{ss:rohit}

Fix a vector bundle $\cE$ on $X$ of rank $d$. The goal of this section is to prove the following theorem:

\begin{theorem} \label{thm:gru-dim}
The space $\Gr(\cE)$ has Krull dimension $\dim(X) + \binom{d+1}{2}$. If $X$ is universally catenary, then $\Gr(\cE)$ is catenary.
\end{theorem}

Combined with the other results of this section, we obtain:

\begin{corollary}
Let $A=\bA(\cE)$. The category $\Mod_A$ has Krull--Gabriel dimension $\dim(X) + \binom{d+1}{2}$. 
\end{corollary}

\begin{lemma} \label{lem:incidence}
Let $Y \subset \Gr_{r+1}(\cE)$ and $Z \subset \Gr_r(\cE)$ be irreducible closed sets such that $Z \subset \ol{Y}$ in $\Gr(E)$. Then $\dim(Y)+r \ge \dim(Z)$.
\end{lemma}

\begin{proof}
Recall the definition of $\Fl_{r,r+1}(\cE)$ and $\pi_r$ and $\pi_{r+1}$ from \S\ref{ss:total-grass}. By Lemma~\ref{lem:closure}, $\ol{Y} \cap \Gr_r(\cE)$ is $\pi_r(\pi_{r+1}^{-1}(Y))$. The space $\pi_{r+1}^{-1}(Y)$ is a $\bP^r$-bundle over $Y$, and therefore has dimension $\dim(Y)+r$. Since this space surjects onto a closed set containing $Z$, we obtain the stated inequality.
\end{proof}

\begin{lemma} \label{lem:krull-dim}
The Krull dimension of $\Gr(\cE)$ is at most $\dim(X) + \binom{d+1}{2}$.
\end{lemma}

\begin{proof}
Let $Z_0 \subset \cdots \subset Z_k$ be a maximal strict chain of irreducible closed subsets in $\Gr(\cE)$. Let $S_r$ be the set of indices $i$ for which $Z_i$ meets $\Gr_r(\cE)$ but not $\Gr_{r+1}(\cE)$, so that for $i \in S_r$ we have $Z_i=\ol{Y}_i$ for some irreducible closed set $Y_i \subset \Gr_r(\cE)$ by Lemma~\ref{lem:irred-Y}. We note that each $S_i$ is non-empty by maximality of the chain. Let $\delta_r$ (resp.\ $\delta'_r$) be the maximum (resp.\ minimum) dimension of $Y_i$ with $i \in S_r$. Since $Y_j \subset \ol{Y}_i$ for all $j \in S_r$ and $i \in S_{r+1}$, we have $\delta'_{r+1}+r \ge \delta_r$ by the previous lemma (for $0 \le r <d$). We thus find
\begin{displaymath}
\# S_r \le \delta_{r}-\delta'_{r} +1 \le r + 1 + \delta'_{r+1} - \delta'_r.
\end{displaymath}
(In fact, the first inequality is an equality by the maximality of the chain.) Therefore,
\begin{align*}
k+1 &= \# S_0 + \# S_1 + \cdots + \# S_{d}\\
& \le \left( \sum_{r=0}^{d-1} (r+1) \right) + \delta_d'-\delta_0' + \# S_d
\le \binom{d+1}{2}+\delta'_d+\# S_d.
\end{align*}
Now, we can regard $\{Z_i\}_{i \in S_d}$ as a descending chain of irreducible closed sets in $X$, as $\Gr_d(\cE)=X$. The smallest member of this chain has dimension $\delta'_d$. Thus we have $\delta_d'+\# S_d \le \dim(X)+1$, and the theorem follows.
\end{proof}

\begin{example} \label{eg:krull-dim}
Here is a chain of irreducible closed sets of length $\binom{d+1}{2}$ in $\Gr(E)$ for a vector space $E$. Let
\begin{displaymath}
0 = V_0 \subset V_1 \subset \cdots \subset V_d = E
\end{displaymath}
be a complete flag in $E$. For $0 \le r < d$ and $0 \le i \le r$, let $Y_{r,i}$ be the set of all $r$-dimensional subspaces of $V_{r+1}$ containing $V_{r-i}$. By replacing a subspace by its quotient, these give irreducible closed subsets of $\Gr_{d-r}(E)$, and form a chain
\begin{displaymath}
Y_{r,0} \subset Y_{r,1} \subset \cdots \subset Y_{r,r}.
\end{displaymath}
Let $Z_{r,i}$ be the closure of $Y_{r,i}$ in $\Gr(E)$. This is irreducible by general principles. Furthermore, by Lemma~\ref{lem:closure}, we have a strict chain
\begin{displaymath}
Z_{0,0} \subset Z_{1,0} \subset Z_{1,1} \subset Z_{2,0} \subset Z_{2,1} \subset Z_{2,2} \subset Z_{3,0} \subset \cdots \subset Z_{d-1,d-1} \subset \Gr(E).
\end{displaymath}
There are $\sum_{i=0}^{d-1} (i+1) + 1= \binom{d+1}{2}+1$ sets in this chain.

For a general base $X$, let $Y$ be the reduced subscheme of an irreducible component of largest possible dimension. Over the generic point of $Y$, $\cE$ becomes a vector space and we can build the chain as above. Now take closures of these subvarieties to get a chain starting at $Y$ of length $\binom{d+1}{2}$. Now concatenate this with a maximal chain of irreducible subspaces ending at $Y$ to get a chain of length $\dim(X) + \binom{d+1}{2}$.
\end{example}

\begin{proof}[Proof of Theorem~\ref{thm:gru-dim}]
Combining Lemma~\ref{lem:krull-dim} and Example~\ref{eg:krull-dim} shows that $\dim \Gr_r(\cE) = \dim(X) + \binom{d+1}{2}$. 

So it remains to show $\Gr_r(\cE)$ is catenary when $X$ is universally catenary. Without loss of generality, we may replace $X$ with one of its irreducible components. We need to show that for any irreducible subspaces $Y \subset Y'$, every maximal chain of irreducible closed subsets between $Y$ and $Y'$ has the same length. By extending these to maximal chains in the whole space (which is irreducible), it suffices to consider the case $Y = \emptyset$ and $Y' = \Gr(\cE)$.

Use the notation from the previous proof. Consider a chain $Z_0 \subset \cdots \subset Z_k$ which is maximal. Suppose $Z_i$ meets $\Gr_{r}(\cE)$ but not $\Gr_{r+1}(\cE)$ and that $Z_{i+1}$ meets $\Gr_{r+1}(\cE)$. Write $Z_i = \ol{Y}_i$ and $Z_{i+1} = \ol{Y}_{i+1}$ for irreducible closed subsets $Y_i \subset \Gr_r(\cE)$ and $Y_{i+1} \subset \Gr_{r+1}(\cE)$. By Lemma~\ref{lem:incidence}, $\dim Y_{i+1} +r \ge \dim Y_i$. It suffices to check that this must be an equality, since $\Gr_r(\cE)$ is catenary. Now use the notation from Lemma~\ref{lem:incidence}. So we have a map $\pi_r \colon \pi_{r+1}^{-1}(Y_{i+1}) \to \Gr_r(\cE)$ whose image is irreducible and contains $Y_i$. If the image strictly contains $Y_i$, we can insert its closure in between $Z_i$ and $Z_{i+1}$ and get a longer chain, which is a contradiction. So we conclude that the image is equal to $Y_i$. 

If $\dim Y_{i+1} + r > \dim Y_i$, then the fibers of $\pi_r$ all have positive dimension. In that case, let $F_W$ be the fiber over $W \in Y_i$, i.e., the set of quotients in $Y_{i+1}$ which further quotient to $W$. It is easy to check that $F_W$ is closed and from what we have assumed, $\dim F_W \ge 1$. Now pick an open affine subset of $X$ that trivializes $\cE$ and that has nonempty intersection with the image of $Y_i$ and $Y_{i+1}$ in $X$; let $E$ denote the restriction of $\cE$ to this open affine. Let $H$ be a Schubert divisor in $\Gr_{r+1}(E)$, i.e., the intersection of $\Gr_{r+1}(E)$ with a hyperplane in the Pl\"ucker embedding. If we choose $H$ generically, then it does not contain $Y_{i+1}$ and $H \cap Y_{i+1}$ has codimension $1$ in $Y_{i+1}$. Also, the intersection $H \cap F_W$ is always nonempty (this is easy to see by considering the Pl\"ucker embedding). So $Y_i \subset \pi_r (\pi^{-1}_{r+1}(\ol{H \cap Y_{i+1}}))$. Since $Y_i$ is irreducible and $\pi^{-1}_{r+1}$ is a projective bundle, we can choose an irreducible component $Z' \subset H \cap Y_{i+1}$ such that $Y_i \subset \pi_r (\pi^{-1}_{r+1}(\ol{Z'}))$. So $Z'$ is an irreducible subvariety of one smaller dimension whose closure can be inserted in between $Z_i$ and $Z_{i+1}$, which gives a contradiction.
\end{proof}

\section{The formalism of saturation and local cohomology} \label{s:formalism}

\subsection{Decomposing into two pieces}

Let $\cA$ be a Grothendieck abelian category and let $\cB$ be a localizing subcategory. We assume the following hypothesis holds: 
\[
\label{eqn:prop-H}
\makebox{\parbox{\dimexpr\linewidth-20\fboxsep-2\fboxrule}{Injective objects of $\cB$ remain injective in $\cA$.}} \tag{Inj}
\]
This is not automatic (see Examples~\ref{ex:nonH} and~\ref{ex:nonH2}). Let $T \colon \cA \to \cA/\cB$ be the localization functor and let $S \colon \cA/\cB \to \cA$ be its right adjoint (the section functor). We define the {\bf saturation} of $M \in \cA$ (with respect to $\cB$) to be $\Sigma(M)=S(T(M))$.  We also define the {\bf torsion} of $M$ (with respect to $\cB$), denoted $\Gamma(M)$, to be the maximal subobject of $M$ that belongs to $\cB$. (This exists since $\cB$ is localizing.) We say that $M \in \cA$ is {\bf saturated} if the natural map $M \to \Sigma(M)$ is an isomorphism. We say that $M \in \rD^+(\cA)$ is {\bf derived saturated} if the natural map $M \to \rR \Sigma(M)$ is an isomorphism. We note that $\Sigma(M)$ is always saturated and $\rR \Sigma(M)$ is always derived saturated. We refer to $\rR \Gamma$ as {\bf local cohomology}.

\begin{proposition} \label{prop:satcrit}
Let $M \in \cA$. The following are equivalent:
\begin{enumerate}[\indent \rm (a)]
\item $M$ is saturated.
\item $\Ext^i_{\cA}(N, M)=0$ for $i=0,1$ and all $N \in \cB$.
\item $\Hom_{\cA}(N,M)=\Hom_{\cA/\cB}(T(N), T(M))$ for all $N \in \cA$.
\end{enumerate}
\end{proposition}

\begin{proof}
The equivalence of (a) and (b) is \cite[p.371, Corollaire]{gabriel} (objects satisfying (b) are called $\cB$-ferm\'e). The equivalence of (b) and (c) is \cite[p.370, Lemma]{gabriel}.
\end{proof}

\begin{proposition}
If $M \in \cB$ then $\rR^i \Gamma(M)=0$ for $i>0$.
\end{proposition}

\begin{proof}
Let $M \to I^{\bullet}$ be an injective resolution in $\cB$. By \eqref{eqn:prop-H}, this remains an injective resolution in $\cA$. Thus $\rR \Gamma(M) \to \Gamma(I^{\bullet})$ is an isomorphism. But $\Gamma(I^{\bullet})=I^{\bullet}$ since each $I^n$ belongs to $\cB$, and so the result follows.
\end{proof}

\begin{proposition} \label{prop:injses}
Let $I \in \cA$ be injective. Then we have a short exact sequence
\begin{displaymath}
0 \to \Gamma(I) \to I \to \Sigma(I) \to 0
\end{displaymath}
with $\Gamma(I)$ and $\Sigma(I)$ both injective. Moreover, $T(I) \in \cA/\cB$ is injective.
\end{proposition}

\begin{proof}
Since $\Gamma$ is right adjoint to the inclusion $\cB \to \cA$, it takes injectives to injectives. Thus $\Gamma(I)$ is injective in $\cB$, and so injective in $\cA$ by \eqref{eqn:prop-H}. Let $J=I/\Gamma(I)$. Then $J$ is injective and $\Gamma(J)=0$, and so $J$ is saturated by Proposition~\ref{prop:satcrit}(b). Since the map $I \to J$ has kernel and cokernel in $\cB$, it follows that $J$ is the saturation of $I$, that is, the natural map $\Sigma(I) \to \Sigma(J)=J$ is an isomorphism. Finally, note that $\Hom_{\cA/\cB}(T(-), T(I)) \cong \Hom_{\cA}(-, \Sigma(I))$ is exact, which implies that $\Hom_{\cA/\cB}(-, T(I))$ is exact, and so $T(I)$ is injective.
\end{proof}

\begin{proposition}
The functors $T$ and $S$ give mutually quasi-inverse equivalences between the category of torsion-free injectives in $\cA$ and the category of injectives in $\cA/\cB$.
\end{proposition}

\begin{proof}
Proposition~\ref{prop:injses} shows that $T$ carries injectives to injectives, while this is true for $S$ for general reasons. If $I$ is a torsion-free injective in $\cA$ then Proposition~\ref{prop:injses} shows that the map $I \to S(T(I))$ is an isomorphism. On the other hand, for any object $I$ of $\cA/\cB$ the map $T(S(I)) \to I$ is an isomorphism. Thus $T$ and $S$ are quasi-inverse.
\end{proof}

Let $\II(\cA)$ be the set of isomorphism classes of indecomposable injectives in $\cA$. From the previous proposition, we find:

\begin{proposition}
We have $\II(\cA) = \II(\cB) \amalg \II(\cA/\cB)$.
\end{proposition}

\begin{proposition} \label{prop:triangle}
Let $M \in \rD^+(\cA)$. Then we have an exact triangle
\begin{displaymath}
\rR \Gamma(M) \to M \to \rR \Sigma(M) \to
\end{displaymath}
where the first two maps are the canonical ones.
\end{proposition}

\begin{proof}
Work in the homotopy category of injective complexes. Let $M$ be an object in this category. Then from Proposition~\ref{prop:injses}, we have a short exact sequence of complexes
\begin{displaymath}
0 \to \Gamma(M) \to M \to \Sigma(M) \to 0,
\end{displaymath}
and this gives the requisite triangle.
\end{proof}

\begin{proposition} \label{prop:derived-sat}
An object $M \in \rD^+(\cA)$ is derived saturated if and only if $\rR \Hom_{\cA}(N, M)=0$ for all $N \in \rD^+(\cB)$.
\end{proposition}

\begin{proof}
Suppose $M$ is derived saturated. Choose an injective resolution $T(M) \to I^\bullet$ in $\cA/\cB$. Applying $S$, we get a quasi-isomorphism $M \to S(I^\bullet)$. Given $N \in \rD^+(\cB)$, we have $\rR \Hom_\cA(N,M) = \Hom(N, S(I^\bullet))$, and the latter is $0$ by Proposition~\ref{prop:satcrit} since $S(I^\bullet)$ consists of saturated objects.

Conversely, if $M$ is not derived saturated, Proposition~\ref{prop:triangle} implies that $\rR \Gamma(M) \ne 0$, and so $\rR \Hom(\rR\Gamma(M), M) \ne 0$. Since $\rR \Gamma(M) \in \rD^+(\cB)$, we are done.
\end{proof}

\begin{definition}
Given a triangulated category $\cT$, a collection of full triangulated subcategories $\cT_1, \dots, \cT_n$ is a {\bf semi-orthogonal decomposition} of $\cT$ if 
\begin{enumerate}[\indent (1)]
\item $\hom_\cT(X_i, X_j) = 0$ whenever $X_i \in \cT_i$ and $X_j \in \cT_j$ and $i<j$, and 
\item the smallest triangulated subcategory of $\cT$ containing $\cT_1, \dots, \cT_n$ is $\cT$. 
\end{enumerate}
In that case, we write $\cT = \langle \cT_1, \dots, \cT_n \rangle$.
\end{definition}

Let $\cD_0^+$ (resp.\ $\cD_1^+$) be the full subcategory of $\rD^+(\cA)$ on objects $M$ such that $\rR \Sigma(M)=0$ (resp.\ $\rR \Gamma(M)=0$).

\begin{proposition} 
We have the following:
\begin{enumerate}[\indent \rm (a)]
\item The inclusion $\rD^+(\cB) \to \cD_0^+$ is an equivalence, with quasi-inverse $\rR \Gamma$.
\item The functor $T \colon \cD_1^+ \to \rD^+(\cA/\cB)$ is an equivalence, with quasi-inverse $\rR S$.
\item We have a semi-orthogonal decomposition $\rD^+(\cA)=\langle \cD_0^+, \cD_1^+ \rangle$.
\end{enumerate}
\end{proposition}

\begin{proof}
(a) Pick $M \in \cD_0^+$. By Proposition~\ref{prop:triangle}, we have a naturally given quasi-isomorphism $\rR \Gamma(M) \cong M$. So the composition $\cD_0^+ \to \rD^+(\cB) \to \cD_0^+$ is isomorphic to the identity. For $M \in \rD^+(\cB)$, $\rR \Gamma(M)$ can be computed by applying $\Gamma$ termwise, so it is immediate that the composition $\rD^+(\cB) \to \cD_0^+ \to \rD^+(\cB)$ is also isomorphic to the identity.

(b) For $M \in \cD_1^+$, Proposition~\ref{prop:triangle} gives a natural quasi-isomorphism $M \to \rR \Sigma(M)$, so $\rR S \circ T \cong {\rm id}_{\cD_1^+}$. For $M \in \rD^+(\cA/\cB)$, the homology of $\rR S (M)$ lives in $\cB$, so $T(\rR S(M))$ can be computed by applying $TS$ termwise, and we get $T \circ \rR S \cong {\rm id}_{\rD^+(\cA/\cB)}$.

(c) Since $\cD_1^+$ is the subcategory of derived saturated objects, we have $\hom_{\rD^+(\cA)}(X,Y) = 0$ whenever $X \in \cD_0^+$ and $Y \in \cD_1^+$ by Proposition~\ref{prop:derived-sat}. By Proposition~\ref{prop:triangle}, $\rD^+(\cA)$ is generated by $\cD_0^+$ and $\cD_1^+$.
\end{proof}

\begin{remark}
In many familiar situations, the above functors and decompositions do not preserve finiteness. For example, suppose $\cA$ is the category of $\bC[t]$-modules and $\cB$ is the category of torsion modules. Let $M=\bC[t]$. Then $\rR \Gamma(M)=(\bC(t)/\bC[t])[1]$ and $\rR \Sigma(M)=\bC(t)$. Thus the projection of $M$ to the two pieces of the semi-orthogonal decomposition are not finitely generated objects of $\cA$. Nonetheless, in our eventual application of this formalism to tca's, finiteness \emph{will} be preserved!
\end{remark}

\subsection{Decomposing into many pieces} \label{ss:formalism2}

Let $\cA$ be a Grothendieck abelian category and let
\begin{displaymath}
\cA_{\le 0} \subset \cdots \subset \cA_{\le d}=\cA
\end{displaymath}
be a chain of localizing subcategories (we do not assume $\cA_{\le 0}=0$). We assume that each $\cA_{\le r} \subset \cA$ satisfies \eqref{eqn:prop-H}. We put
\begin{displaymath}
\cA_{>r}=\cA/\cA_{\le r} \qquad \textrm{and} \qquad \cA_r = \cA_{\le r}/\cA_{\le r-1}.
\end{displaymath}
We let $T_{>r} \colon \cA \to \cA_{>r}$ be the localization functor, $S_{>r}$ its right adjoint, and $\Sigma_{>r}=S_{>r} \circ T_{>r}$, and we let $\Gamma_{\le r}$ be the right adjoint to the inclusion $\cA_{\le r} \to \cA$. The subscripts in these functors are meant to indicate that they are truncating in certain ways: intuitively $\Gamma_{\le r}(M)$ keeps the part of $M$ in $\cA_{\le r}$ and discards the rest, while $T_{>r}(M)$ discards the part of $M$ in $\cA_{\le r}$ and keeps the rest. By convention, we put $\Gamma_{\le r}=0$ for $r<0$ and $\Gamma_{\le r}=\id$ for $r>d$, and put $\Sigma_{>r}=\id$ for $r<0$ and $\Sigma_{>r}=0$ for $r>d$. We also put $\cA_{\le r}=0$ for $r<0$ and $\cA_{\le r}=\cA$ for $r>d$. We have the following connections between these functors:

\begin{proposition}
We have the following:
\begin{enumerate}[\indent \rm (a)]
\item Any pair of functors in the set $\{\Gamma_{\le i},\Sigma_{>j}\}_{i,j \in \bZ}$ commutes.
\item We have $\Gamma_{\le i} \Gamma_{\le j} = \Gamma_{\le \min(i,j)}$.
\item We have $\Sigma_{>i} \Sigma_{>j} = \Sigma_{>\max(i,j)}$.
\item We have $\Gamma_{\le i} \Sigma_{>j}=0$ if $i \le j$.
\end{enumerate}
These results hold for the derived versions of the functors as well.
\end{proposition}

\begin{proposition}
Let $i<j$, let $T' \colon \cA_{\le j} \to \cA_{\le j}/\cA_{\le i}$ be the localization functor, and let $S'$ be its right adjoint.
\begin{enumerate}[\indent \rm (a)]
\item $S'$ coincides with the restriction of $S_{>i}$ to $\cA_{\le j}/\cA_{\le i}$.
\item Injectives in $\cA_j/\cA_i$ remain injective in $\cA/\cA_{\le i}$.
\item $\rR S'$ coincides with the restriction of $\rR S_{>i}$ to $\rD^+(\cA_{\le j}/\cA_{\le i})$.
\end{enumerate}
\end{proposition}

\begin{proof}
(a) Suppose $M \in \cA_{\le j}/\cA_{\le i}$ and write $M=T'(N)$ with $N \in \cA_{\le j}$. The natural map $N \to S_{>i}(M)$ has kernel and cokernel in $\cA_{\le i}$. Since $\cA_{\le j}$ is a Serre subcategory, it follows that $S_{>i}(M)$ belongs to $\cA_{\le j}$. Thus $S_{>i}$ maps $\cA_{\le j}/\cA_{\le i}$ into $\cA_{\le j}$, from which it easily follows that it is the adjoint to $T'$.

(b) Let $I$ be injective in $\cA_{\le j}/\cA_{\le i}$. Then $S'(I)$ is injective in $\cA_{\le j}$ since section functors always preserve injectives. By (a), $S'(I)=S_{>i}(I)$. Thus, by \eqref{eqn:prop-H}, we see that $S_{>i}(I)$ is injective in $\cA$. But $T_{>i}$ takes injectives to injectives, and so $I=T_{>i}(S_{>i}(I))$ is injective in $\cA/\cA_{\le i}$.

(c) This follows immediately from (a) and (b) (compute with injective resolutions).
\end{proof}

\begin{proposition} \label{prop:ind-inj}
We have a natural bijection $\II(\cA) = \coprod_{r=0}^d \II(\cA_r)$.
\end{proposition}

Let $\rD^+(\cA)_i$ \Acom{changed notation from $\cD_i^+$ to be more consistent with what's used in following sections} be the full triangulated subcategory of $\rD^+(\cA)$ on objects $M$ such that $\rR \Gamma_{<i}(M)$ and $\rR \Sigma_{>i}(M)$ vanish.

\begin{proposition} \label{prop:derived-decomp}
We have the following:
\begin{enumerate}[\indent \rm (a)]
\item We have an equivalence $\rR S_{\ge i} \colon \rD^+(\cA_i) \to \rD^+(\cA)_i$.
\item We have a semi-orthogonal decomposition $\rD^+(\cA)=\langle \rD^+(\cA)_0, \ldots, \rD^+(\cA)_d \rangle$.
\end{enumerate}
\end{proposition}

We now define a functor $\rR \Pi_i \colon \rD^+(\cA) \to \rD^+(\cA)_i$ by $\rR \Pi_i=\rR \Sigma_{\ge i} \rR \Gamma_{\le i}$. This is just the derived functor of $\Sigma_{\ge i} \Gamma_{\le i}$. The functor $\rR \Pi_i$ is idempotent, and projects onto the $i$th piece of the semi-orthogonal decomposition.

\subsection{Consequences of finiteness} \label{ss:Fin}

Maintain the set-up from the previous section. We now assume the following hypothesis:
\begin{itemize}
\item[(Fin)] If $M \in \cA$ is finitely generated then $\rR \Sigma_{>i}(M)$ and $\rR \Gamma_{\le i}(M)$ are finitely generated for all $i$ and vanish for $i \gg 0$.
\end{itemize}
We deduce a few consequences of this. We let $\rD^b_{\fgen}(\cA)_r = \rD^+(\cA)_r \cap \rD^b_{\fgen}(\cA)$.

\begin{proposition}
\begin{enumerate}[\indent \rm (a)]
\item We have an equivalence $\rR S_{\le r} \colon \rD^b_{\fgen}(\cA_r) \to \rD^b_{\fgen}(\cA)_r$. The inverse is induced by $T_{\le r}$.
\item We have a semi-orthogonal decomposition $\rD^b_{\fgen}(\cA) = \langle \rD^b_{\fgen}(\cA)_0, \ldots, \rD^b_{\fgen}(\cA)_d \rangle$.
\end{enumerate}
\end{proposition}

\begin{corollary} \label{cor:Kisom}
The functors $T_{\ge r}$ and $S_{\ge r}$ induce inverse isomorphisms $\rK(\rD^b_{\fgen}(\cA)_r)=\rK(\cA_r)$.
\end{corollary}

\begin{proposition} \label{prop:Kdecomp}
We have an isomorphism
\begin{displaymath}
\rK(\cA) = \bigoplus_{r=0}^d \rK(\rD^b_{\fgen}(\cA)_r)
\end{displaymath}
The projection onto the $r$th factor is given by $\rR \Pi_r$.
\end{proposition}

\begin{proof}
The map above is surjective: the composition $\rK(\rD^b_\fgen(\cA)_r) \to \rK(\cA) \to \rK(\rD^b_\fgen(\cA)_r)$, where the first map comes from the inclusion, is an isomorphism.

Suppose it is not injective. Pick an object $M \in \rD^b(\cA)$ whose image is $0$ but such that $[M] \in \rK(\cA)$ is nonzero. Let $i$ be minimal so that $[M]$ is in the image of  $\rK(\cA_{\le i}) \to \rK(\cA)$ but not in the image of $\rK(\cA_{<i}) \to \rK(\cA)$. Also, suppose we have chosen $M$ so that this index $i$ is as small as possible. Note that $[\rR \Pi_r M] = [\rR \Sigma_{\ge i} M]$, so by our choice of $M$, $[\rR \Sigma_{\ge i} M] = 0$. Proposition~\ref{prop:triangle} then implies that $[M] = [\rR \Gamma_{<i} M]$, which contradicts our choice of $i$, so no such object $M$ exists.
\end{proof}

\subsection{Examples and non-examples of property \eqref{eqn:prop-H}}

We first give some positive results on property \eqref{eqn:prop-H}. Recall that a Grothendieck abelian category is {\bf locally noetherian} if it has a set of generators which are noetherian objects.

\begin{proposition}
Let $\cA$ be a locally noetherian Grothendieck category and let $\cB$ be a localizing subcategory satisfying the following condition:
\[
\label{eqn:star}
\makebox{\parbox{\dimexpr\linewidth-20\fboxsep-2\fboxrule}{Given $M \in \cA$ finitely generated, there exists $M_0 \subset M$ such that $M/M_0 \in \cB$ and $M_0$ has no subobject in $\cB$.}} {\tag{$\ast$}}
\]
Then $\cB \subset \cA$ satisfies \eqref{eqn:prop-H}.
\end{proposition}

\begin{proof}
Let $I$ be an injective object of $\cB$. We first note that $\Ext^1_{\cA}(M,I)=0$ if $M \in \cB$. Indeed, a class in $\Ext^1_{\cA}(M,I)$ is represented by an extension
\begin{displaymath}
0 \to I \to E \to M \to 0,
\end{displaymath}
and $E$ necessarily belongs to $\cB$ since this is a Serre subcategory. Since $I$ is injective in $\cB$, the above sequence splits in $\cB$, and this gives a splitting in $\cA$. The result follows.

Now suppose that $M$ is a finitely generated object of $\cA$ and we have an extension
\begin{displaymath}
0 \to I \to E \to M \to 0.
\end{displaymath}
Let $N$ be a finitely generated submodule of $E$ that surjects onto $M$; this exists because $E$ is the sum of its finitely generated subobjects. Using \eqref{eqn:star}, pick $N_0 \subset N$ such that $N/N_0 \in \cB$ and $N_0$ has no subobject in $\cB$. Let $M_0$ be the image of $N_0$ in $M$, and let $E_0=I+N_0$, so that we have an extension
\begin{displaymath}
0 \to I \to E_0 \to M_0 \to 0.
\end{displaymath}
Now, $N_0 \cap I$ is a subobject of $N_0$ belonging to $\cB$, and thus is zero. Thus the map $N_0 \to M_0$ is an isomorphism, and so the above extension splits.

From the exact sequence
\begin{displaymath}
0 \to M_0 \to M \to M/M_0 \to 0
\end{displaymath}
we obtain a sequence
\begin{displaymath}
\Ext^1_{\cA}(M/M_0, I) \to \Ext^1_{\cA}(M, I) \stackrel{i}{\to} \Ext^1_{\cA}(M_0, I).
\end{displaymath}
As $M/M_0$ is a quotient of $N/N_0$, it belongs to $\cB$, and so the leftmost group vanishes by the first paragraph. Thus $i$ is injective. We have shown that the image of the class of $E$ under $i$ vanishes, and so the class of $E$ is in fact 0. However, $E$ was an arbitrary extension, so we conclude that $\Ext^1_{\cA}(M, I)=0$.

We have thus shown that $\Ext^1_{\cA}(M,I)=0$ whenever $M \in \cA$ is a finitely generated. A variant of Baer's criterion now shows that $I$ is injective.
\end{proof}

\begin{corollary}
Let $\cA$ be a locally noetherian Grothendieck category with a right-exact symmetric tensor product. Let $\fa$ be an ideal, that is, a subobject of the unit object, and let $\cB \subset \cA$ be the full subcategory spanned by objects that are locally annihilated by a power of $\fa$. Suppose that the Artin--Rees lemma holds, that is:
\[
\label{eqn:artin-rees}
\makebox{\parbox{\dimexpr\linewidth-20\fboxsep-2\fboxrule}{If $M \in \cA$ is finitely generated and $N \subset M$ then there exists $r$ such that $\fa^n M \cap N = \fa^{n-r} (\fa^r M \cap N)$ for all $n \ge r$.}} {\tag{$\ast\ast$}}
\]
Then $\cB$ satisfies \eqref{eqn:star}, and thus \eqref{eqn:prop-H} as well. 
\end{corollary}

\begin{proof}
  Let $M \in \cA$ be finitely generated and let $T$ be the maximal subobject of $M$ in $\cB$. The subobjects $T[\fa^k]$ form an ascending chain in $T$ that union to all of $T$, and so $T=T[\fa^n]$ for some $n$ by noetherianity. Let $r$ be the Artin--Rees constant for $T \subset M$. Put $M_0=\fa^{n+r} M$. Then $M/M_0$ belongs to $\cB$: by right-exactness of tensor products, we have $\fa^{n+r} M \to \fa^{n+r}(M/M_0) \to 0$, but the image of this map is $0$, and hence $\fa^{n+r}(M/M_0)=0$.

  Finally, by \eqref{eqn:artin-rees}, we have $M_0 \cap T=\fa^n (\fa^r M \cap T) \subset \fa^n T=0$, and so $M_0$ has no subobject belonging to $\cB$.
\end{proof}

\begin{corollary} \label{cor:B-propH}
Let $A$ be the tca $\bA(\cE)$ and let $\fa$ be an ideal in $A$. Then the Artin--Rees lemma holds. In particular, $\Mod_A[\fa^{\infty}] \subset \Mod_A$ satisfies \eqref{eqn:prop-H}.
\end{corollary}

\begin{proof}
Let $M$ be a finitely generated $A$-module with submodule $N$. Pick $n \ge \max(\ell(M), \ell(A))$. Given $r$, the equality $\fa^m M \cap N = \fa^{m-r}(\fa^r M \cap N)$ holds for all $m \ge r$ if and only if it holds when we evaluate on $\bC^n$. The existence of $r$ after we evaluate on $\bC^n$ can be deduced from the standard Artin--Rees lemma which guarantees that such an $r$ exists locally; since our space is noetherian, one can find a value of $r$ that works globally.
\end{proof}

We now give two examples where \eqref{eqn:prop-H} does not hold.

\begin{example} \label{ex:nonH}
Let $\cA$ be the category of $R$-modules for a commutative ring $R$. Fix a multiplicative subset $S \subset R$, and let $\cB$ be the category of modules $M$ such that $S^{-1} M=0$. Then $\Sigma(M)=S^{-1} M$. If $R$ is noetherian then it satisfies Artin--Rees, and so \eqref{eqn:prop-H} holds by the above results. Thus Proposition~\ref{prop:injses} implies that the localization of an injective $R$-module with respect to $S$ remains injective. 

However, there are examples of non-noetherian rings $R$ where injectives localize to non-injectives \cite[\S 3]{dade}. In such a case, \eqref{eqn:prop-H} must fail. 
\end{example}

\begin{example} \label{ex:nonH2}
Let $\cA$ be the category of graded $\bC[t]$-modules supported in degrees~0 and~1 (where $t$ has degree~1), and let $\cB$ be the subcategory of modules supported in degree~1. Then $\cB$ is equivalent to the category of vector spaces, and thus is semi-simple. However, except for 0, no object of $\cB$ is injective in $\cA$, and so \eqref{eqn:prop-H} does not hold. In this example, $\cA$ is locally noetherian, but \eqref{eqn:star} does not hold. Furthermore, $\Sigma$ and $T$ carry injectives to injectives, so these properties are weaker than \eqref{eqn:prop-H}.
\end{example}

\section{Modules supported at 0 and generic modules} \label{s:at0}

\subsection{Set-up}

We fix (for all of \S \ref{s:at0}) a scheme $X$ over $\bC$ (noetherian, separated, and of finite Krull dimension, as always) and a vector bundle $\cE$ on $X$ of rank $d$. We let $A$ be the tca $\bA(\cE)$. We let $\Mod_A^0$ denote the category of $A$-modules supported at~0: precisely, this consists of those $A$-modules $M$ that are locally annihilated by a power of $\cE \otimes \bV \subset A$. We say that an $A$-module is {\bf torsion} if every element is annihilated by a non-zero element of $A$ of positive degree. We let $\Mod_A^{\gen}$ (the ``generic category'') be the Serre quotient of $\Mod_A$ by the subcategory of torsion modules. The goal of \S \ref{s:at0} is to analyze the two categories $\Mod_A^0$ and $\Mod_A^{\gen}$.

In general, we denote trivial bundles of the form $V \otimes \cO_X$ by simply $V$.

\subsection{Modules supported at 0}

Recall from \S \ref{ss:internalhom} the tensor product $\odot$ on $\Mod_A$. The object $\cE^* \oplus \bV$ is naturally a finitely generated $A$-module supported at 0. It follows that $(\cE^* \oplus \bV)^{\odot n}$ is again a finitely generated module supported at 0. The $S_n$-action permuting the $\odot$ factors is $A$-linear, and so $J_{\lambda}=\bS_{\lambda}(\cE^* \oplus \bV)$ is a finitely generated torsion $A$-module.

\begin{proposition} \label{prop:tors}
We have the following:
\begin{enumerate}[\indent \rm (a)]
\item We have an isomorphism of $A$-modules $J_{\lambda} = \usHom(A, \bS_{\lambda}(\bV))$.
\item If $M$ is an $A$-module and $\cF$ is an $\cO_X$-module then
\begin{displaymath}
\sHom_A(M, \cF \otimes J_{\lambda}) = \sHom_{\cO_X}(M_{\lambda}, \cF).
\end{displaymath}
\item If $\cF$ is an injective $\cO_X$-module then $\cF \otimes J_{\lambda}$ is an injective $A$-module.
\item Every finitely generated object $M$ in $\Mod_A^0$ admits a resolution $M \to N^{\bullet}$ where each $N^i$ is a finite sum of modules of the form $\cF \otimes J_{\lambda}$ with $\cF \in \Mod_X^{\fgen}$ and $N^i=0$ for $i \gg 0$.
\item Every finitely generated object of $\Mod_A^0$ admits a finite length filtration where the graded pieces have the form $\cF \otimes \bS_{\lambda}(\bV)$, where $\cF \in \Mod_X^{\fgen}$ and $A_+$ acts by zero on $\bS_{\lambda}(\bV)$. Moreover, if $X$ is connected then the module $\bS_{\lambda}(\bV)$ has no non-trivial $\cO_X$-flat quotients.
\end{enumerate}
\end{proposition}

\begin{proof}
(a) This is clear for $\lambda=1$. The general case then follows from Proposition~\ref{prop:uhom-ten}.

(b) This follows immediately from (a), and the fact (tensor-hom adjunction) that
\[
  \usHom(A, \cF \otimes \bS_{\lambda}(\bV))=\cF \otimes \usHom(A, \bS_{\lambda}(\bV))
\]
since the multiplicity spaces of $A$ are locally free of finite rank as $\cO_X$-modules.

(c) This follows immediately from (b).

(d) Suppose $M$ is supported in degrees $\le n$. Let $J=\bigoplus_{\vert \lambda \vert=n} M_{\lambda} \otimes J_{\lambda}$. By part (b), there is a canonical map $M \to J$, the cokernel of which is supported in degrees $\le n-1$. The result follows from induction on $n$.

(e) Given $M \in \Mod_A^0$, first consider the filtration $M \supset A_+ M \supset A_+^2 M \supset \cdots$. By Nakayama's lemma, $A_+^r M = A_+^{r+1} M$ for some $r$, in which case the common value is $0$ since $M$ is torsion. Each quotient $A_+^i M / A_+^{i+1} M$ has a trivial action of $A$, so is a direct sum of modules of the form $\cF \otimes \bS_\lambda(\bV)$. These can be used to refine our filtration to the desired form. The last statement follows from irreducibility of $\bS_\lambda(\bV)$.
\end{proof}

\begin{remark}
If $X=\Spec(\bC)$ then part~(e) says that every finitely generated object of $\Mod_A^0$ has finite length and the $\bS_{\lambda}(\bV)$ are the simple objects. In general, the statement in part~(e) is a relative version of this, taking into account the non-trivial structure of $\cO_X$-modules.
\end{remark}

The category $\Mod_A^{0,\fgen}$ is naturally a module for the tensor category $\cV^{\fgen}$. Thus $\rK(\Mod_A^{0,\fgen})$ is a module for $\Lambda=\rK(\cV^{\fgen})$. The following result describes $\rK(\Mod_A^{0,\fgen})$ as a $\Lambda$-module.

\begin{corollary} \label{cor:Ktors}
The map $\rK(X) \to \rK(\Mod_A^0)$ taking $[V]$ to the class of the trivial $A$-module $V$ induces an isomorphism $\Lambda \otimes \rK(X) \to \rK(\Mod_A^0)$.
\end{corollary}

\begin{proof}
This follows from Proposition~\ref{prop:tors}(e). 
\end{proof}

\subsection{Representations of general affine groups}

Define $\bG(\cE) = \GL(\bV)_X \ltimes (\bV \otimes \cE)$, which is an (infinite dimensional) algebraic group over $X$. A {\bf representation} of $\bG(\cE)$ is a quasi-coherent sheaf on $X$ on which $\bG(\cE)$ acts $\cO_X$-linearly. For example, $\bV \oplus \cE^*$ is naturally a representation of $\bG(\cE)$, which we call the {\bf standard representation}. A representation is {\bf polynomial} if it is a subquotient of a direct sum of representations of the form $\cF \otimes (\bV \oplus \cE^*)^{\otimes k}$ with $\cF$ is a quasi-coherent sheaf on $X$ and $k \ge 0$. We note that if $V$ and $W$ are polynomial representations then $V \otimes W$ (tensor product as $\cO_X$-modules) is again a polynomial representation. We write $\Rep^{\pol}(\bG(\cE))$ for the category of polynomial representations.

\begin{proposition} \label{prop:GAtors}
We have a natural equivalence of categories $\Mod^0_{\bA(\cE)} = \Rep^{\pol}(\bG(\cE))$, under which the tensor product $\odot$ corresponds to the tensor product $\otimes$.
\end{proposition}

\begin{proof}
Let $M$ be an $\bA(\cE)$-module. Since $\bA(\cE)$ is the universal enveloping algebra of the abelian Lie algebra $\bV \otimes \cE$, the $\bA(\cE)$-module structure on $M$ gives a representation of the algebraic group $\bV \otimes \cE$ on $M$. Moreover, the $\GL(\bV)$-action on $M$ interacts with the $\bV \otimes \cE$ action in the appropriate way to define an action of $\bG(\cE)$ on $M$. We show that this construction induces the equivalence of categories.

We first observe that this construction is compatible with tensor products, that is, if $M$ and $N$ are $\bA(\cE)$-modules then the $\bG(\cE)$-representation on $M \odot N$ is just the tensor product of the $\bG(\cE)$-representations on $M$ and $N$. This follows immediately from the definitions.

Next we show that if $M$ is a torsion $\bA(\cE)$-module then the associated $\bG(\cE)$-representation is polynomial. It suffices to treat the case where $M$ is finitely generated, since a direct limit of polynomial representations is still polynomial. In this case, we can embed $M$ into a module of the form $\cF \otimes J_{\lambda}$ by Proposition~\ref{prop:tors}. By the previous paragraph, we see that $J_{\lambda} = \bS_{\lambda}(\cE^* \oplus \bV)$ as a representation of $\bG(\cE)$, and thus is polynomial. Thus $M$ embeds into a polynomial representation, and is therefore polynomial.

It is clear from the construction that one can recover the $\bA(\cE)$-module structure on $M$ from the $\bG(\cE)$-representation, and so the functor in question is fully faithful. Moreover, it follows that if $M$ is an $\bA(\cE)$-module and $N$ is a $\bG(\cE)$-subrepresentation then $N$ is in fact an $\bA(\cE)$-submodule. From this, it follows that the essential image of the functor in question is closed under formation of subquotients. Since the essential image includes all representations of the form $\cF \otimes (\cE^* \oplus \bV)^{\otimes n}$, it follows that our functor is essentially surjective.
\end{proof}

\begin{corollary} \label{cor:GAtors}
We have the following:
\begin{enumerate}[\indent \rm (a)]
\item If $\cF$ is an injective $\cO_X$-module then $\cF \otimes \bS_{\lambda}(\bV \oplus \cE^*)$ is an injective object of $\Rep^{\pol}(\bG(\cE))$.
\item Every finitely generated object $M$ of $\Rep^{\pol}(\bG(\cE))$ admits a resolution $M \to N^{\bullet}$ where each $N^i$ is a finite direct sum of representations of the form $\cF \otimes \bS_{\lambda}(\bV \oplus \cE^*)$ with $\cF \in \Mod_X^{\fgen}$ and $N^i=0$ for $i \gg 0$.
\item Every finitely generated object of $\Rep^{\pol}(\bG(\cE))$ admits a finite length filtration where the graded pieces have the form $\cF \otimes \bS_{\lambda}(\bV)$, where $\cF \in \Mod_X^{\fgen}$. Moreover, if $X$ is connected then $\bS_{\lambda}(\bV)$ admits no non-trivial $\cO_X$-flat quotients.
\end{enumerate}
\end{corollary}

\subsection{The generic category} \label{ss:gen}

We now study the category $\Mod_A^{\gen}$. The key result is:

\begin{proposition}
There exists an $\cO_X$-linear equivalence of categories $\Phi \colon \Mod_A^{\gen} \to \Rep^{\pol}(\bG)$ that is compatible with tensor products and carries $T(A \otimes \bV)$ to the standard representation $\cE^* \oplus \bV$.
\end{proposition}

Here is the idea of the proof: an $A$-module is a quasi-coherent equivariant sheaf on the space $\uHom(\cE^*, \bV)$. The category $\Mod_A^{\gen}$ can be identified with quasi-coherent equivariant sheaves on the open subscheme $\uHom(\cE^*, \bV)^{\circ}$ where the map is injective. The group $\GL(\bV)_X$ acts transitively on this space with stabilizer $\bG$ (almost), and so $\Mod_A^{\gen}$ is equivalent to $\Rep(\bG)$. We now carry out the details rigorously. This is unfortunately lengthy.

\begin{proof} 
First suppose that $\cE$ is trivial. Let $d$ be the rank of $\cE$, and choose a decomposition $\bV=V_0 \oplus \bV'$ where $V_0$ has dimension $d$. Also choose an isomorphism $i \colon \cE \to V_0^* \otimes \cO_X$. This isomorphism induces a pairing $\cE \otimes \bV \to \cO_X$, which in turn induces a homomorphism $A \to \cO_X$. For an $A$-module $M$, let $\Psi_i(M)=M \otimes_A \cO_X$.

Suppose now that $i'$ is a second choice of isomorphism, and write $i'=gi$ where $g$ is a section of $\GL(V_0)_X$. Regard $\GL(V_0)$ as a subgroup of $\GL(\bV)$ in the obvious manner. Since $\GL(\bV)$ acts on $M$, there is an induced map $g \colon M \to M$. One readily verifies that this induces an isomorphism $\Psi_i(M) \to \Psi_{i'}(M)$. We thus see that $\Psi_i(M)$ is canonically independent of $i$, and so denote it by $\Psi(M)$. (To be more canonical, one could define $\Psi(M)$ as the limit of the $\Psi_i(M)$ over the category of isomorphisms $i$.)

Let $\bG'$ be defined like $\bG$, but using $\bV'$, that is, $\bG'=\GL(\bV') \ltimes (\cE \otimes \bV')$. Let $\bG''$ be the subgroup of $\GL(\bV)_X$ consisting of elements $g$ such that $g(\bV')=\bV'$ and the map $g \colon V_0 \to \bV/\bV'=V_0$ is the identity. Note that $\bG''=\GL(\bV')_X \ltimes (V_0^* \otimes \bV'_X)$, and so $i$ induces an isomorphism $\bG' \cong \bG''$. The group $\bG''$ stabilizes the pairing $\cE \otimes \bV \to \cO_X$, and thus acts on $\Psi_i(M)$. The group $\GL(V_0)$ acts on $\bG''$, via its action on $V_0$. If $i'=gi$ then the induced isomorphism $\varphi \colon \Psi_i(M) \to \Psi_{i'}(M)$ is compatible with the $\bG''$ actions in the sense that $\varphi(hx)={}^gh \varphi(x)$ for $h \in \bG''$. It follows that if we let $\bG'$ act on $\Psi_i(M)$ via the isomorphism $\bG' \cong \bG''$ induced by $i$, then $\varphi(hx)=h \varphi(x)$ for $h \in \bG'$. We thus see that $\bG'$ canonically acts on $\Psi(M)$.

Now suppose that $\cE$ is arbitrary. Then we can define $\Psi(M)$ with its $\bG'$ action over a cover trivializing $\cE$. Since everything is canonical, the pieces patch to define $\Psi(M)$ over all of $X$. We have thus defined a functor $\Psi \colon \Mod_A \to \Rep(\bG')$. We will eventually deduce the desired equivalence $\Phi$ from this functor.

It is clear that $\Psi$ is a tensor functor: indeed, working locally,
\begin{displaymath}
\Psi_i(M \otimes_A N)=(M \otimes_A N) \otimes_A \cO_X=(M \otimes_A \cO_X) \otimes_{\cO_X} (N \otimes_A \cO_X) = \Psi_i(M) \otimes_{\cO_X} \Psi_i(N).
\end{displaymath}
Working locally, we have $\Psi_i(A \otimes \bV)=\bV_X=(V_0)_X \oplus \bV'_X$. This globalizes to $\Psi(A \otimes \bV)=\cE^* \oplus \bV'$, the standard representation of $\bG'$. We thus see that $\Psi(A \otimes \bV^{\otimes n})=(\cE^* \otimes \bV')^{\otimes n}$ is a polynomial representation of $\bG'$. Since every $A$-module is a quotient of a sum of modules of the form $A \otimes \bV^{\otimes n}$, it follows that $\Psi(M)$ is a polynomial representation of $\bG'$ for any $A$-module $M$.

Suppose now that $\cE$ is trivial and $M$ is finitely generated. Let $V$ be a sufficiently large finite dimensional vector space, and choose a decomposition $V=V_0 \oplus V'$. Then $\Psi(M)(V)$ is identified with $M(V) \otimes_{A(V)} \cO_X$. Now, the spectrum of $A(V)$ is identified with the space $\uHom(\cE, V^*)$ over $X$, and so $\Psi(M)(V)$ is identified with the pullback of the coherent sheaf $M(V)$ along the section $X \to \uHom(\cE, V^*)$ induced by $i$. This section lands in the open subscheme $\uHom(\cE, V^*)^{\circ}$ consisting of injective maps.

We now claim that $\Psi$ is exact. This can be checked locally. Furthermore, since $\Psi$ commutes with direct limits, it suffices to check on finitely generated modules. We can therefore place ourselves in the situation of the previous paragraph. We can compute $\Psi(M)(V)$ in two steps: first restrict from $\uHom(\cE, V^*)$ to $\uHom(\cE, V^*)^{\circ}$, and then restrict again to $X$. The first step is exact since restriction to an open subscheme is always exact. Now the key point: $\uHom(\cE, V^*)^{\circ}$ is a $\GL(V)_X$-torsor over $X$, and so any equivariant sheaf or equivariant map of such sheaves is pulled back from $X$. Thus pullback of such modules to $X$ is again exact. This completes the proof of the claim.

We next claim that $\Psi$ kills torsion modules. Again, we can work locally and assume $M$ is finitely generated. If $M$ is torsion then the support of $M(V)$ in $\uHom(\cE, V^*)$ does not meet $\uHom(\cE, V^*)^{\circ}$, and so the pullback to $X$ vanishes. This proves the claim.

We thus see that $\Psi$ induces an exact tensor functor $\ol{\Psi} \colon \Mod_A^{\gen} \to \Rep^{\pol}(\bG')$. We claim that $\ol{\Psi}$ is fully faithful. This can again be checked locally for finitely generated modules after evaluating on $V$ of dimension $n \gg 0$. Since $\GL(V)_X$ acts transitively on $\uHom(\cE, V^*)^{\circ}$ with stabilizer $\bG'$, giving a map of $\GL(V)_X$-equivariant quasi-coherent sheaves on $\uHom(\cE, V^*)^{\circ}$ is the same as giving maps at the fibers at $i$, as $\bG'$-representations.

Suppose that $M \to N$ is a map of $A$-modules such that the induced map $\Psi(M) \to \Psi(N)$ vanishes. Then the map $M(V) \to N(V)$ vanishes over $\uHom(\cE, V^*)^{\circ}$. This implies that the image of $M(V) \to N(V)$ is torsion, and so the image of $M \to N$ is torsion, and so the map $M \to N$ is 0 in $\Mod_A^{\gen}$. This proves faithfulness of $\ol{\Psi}$.

Now suppose that a map $\Psi(M)(V) \to \Psi(N)(V)$ of $\bG'$-representations is given. This is induced from a map $M(V) \to N(V)$ over $\uHom(\cE, V^*)^{\circ}$. This induces a map of quasi-coherent sheaves $M(V) \to j_*(N(V))$, where $j$ is the open immersion. Now, $j_*(N(V))$ is a $\GL(V)$-equivariant $A(V)$-module, but may not be polynomial. Let $N(V)'$ be the maximal polynomial subrepresentation, which is an $A(V)$-submodule containing $N(V)$. Let $N'$ be the canonical $A$-module with $\ell(N') \le n$ satisfying $N'(V)=N(V)'$. The map $M(V) \to N(V)'$ is induced from a map of $A$-modules $M \to N'$. Now, $N'(V)/N(V)=j_*(j^*(N(V)))/N(V)$ pulls back to 0 under $j^*$, and is thus torsion. It follows that $N'/N$ is torsion, and so $N=N'$ in $\Mod_A^{\gen}$. Thus the constructed map $M \to N'$ of $A$-modules gives the required map $M \to N$ in $\Mod_A^{\gen}$. This proves fullness of $\ol{\Psi}$.

We now claim that $\ol{\Psi}$ is essentially surjective. Since all categories are cocomplete and $\Psi$ is cocontinuous and fully faithful, it suffices to show that all finitely generated objects are in the essential image. By Proposition~\ref{cor:GAtors}(b), a finitely generated object $M$ of $\Rep^{\pol}(\bG)$ can be realized as the kernel of a map $f \colon P \to Q$, where $P$ and $Q$ are each sums of representations of the form $\cF \otimes (\cE^* \oplus \bV)^{\otimes n}$ with $\cF$ an $\cO_X$-module. We have already shown that such modules are in the essential image of $\Psi$. Thus $P=\ol{\Psi}(P')$ and $Q=\ol{\Psi}(Q')$ for $P'$ and $Q'$ in $\Mod_A^{\gen}$. Since $\ol{\Psi}$ is full, $f=\ol{\Psi}(f')$ for some $f' \colon P' \to Q'$ in $\Mod_A^{\gen}$. Finally, since $\ol{\Psi}$ is exact, $M=\ker(f)=\ol{\Psi}(\ker(f'))$, which shows that $M$ is in the essential image of $\ol{\Psi}$.

We have thus shown that $\ol{\Psi}$ is an equivalence of categories $\Mod_A^{\gen} \to \Rep^{\pol}(\bG')$. Combining this with the obvious equivalence $\Rep^{\pol}(\bG')=\Rep^{\pol}(\bG)$ coming from a choice of isomorphism $\bV' \cong \bV$, we obtain the desired equivalence $\Phi$.
\end{proof}

There is a canonical map $\bV \otimes A \to \cE^* \otimes A$. We let $\cK$ be the kernel of the corresponding map in $\Mod_A^{\gen}$. Under the equivalence $\Phi$ in the proposition, we have $\Phi(\cK)=\bV$. Combining the proposition with Corollary~\ref{cor:GAtors}, we obtain:

\begin{corollary} \label{cor:generic-inj}
We have the following:
\begin{enumerate}[\indent \rm (a)]
\item If $\cF$ is an injective object in $\cV_X$ then $T(\cF \otimes A)$ is injective in $\Mod_A^{\gen}$.
\item Every finitely generated object $M$ of $\Mod_A^{\gen}$ admits a resolution $M \to N^{\bullet}$ where each $N^i$ has the form $T(\cF \otimes A)$ with $\cF \in \cV_X^{\fgen}$ and $N^i=0$ for $i \gg 0$.
\item Every finitely generated object of $\Mod_A^{\gen}$ admits a finite length filtration where the graded pieces have the form $\cF \otimes \bS_{\lambda}(\cK)$ where $\cF \in \Mod_X^{\fgen}$. Moreover, if $X$ is connected then $\bS_{\lambda}(\cK)$ has no non-trivial $\cO_X$-flat quotients.
\end{enumerate}
\end{corollary}

The category $\Mod_A^{\gen}$ is naturally a module for the tensor category $\cV$. Thus $\rK(\Mod_A^{\gen})$ is a module for $\Lambda=\rK(\cV)$. The following result describes $\rK(\Mod_A^{\gen})$ as a $\Lambda$-module.

\begin{corollary} \label{cor:Ktheory-generic}
The map $\rK(X) \to \rK(\Mod_A^{\gen})$ taking $[V]$ to $[T(A \otimes V)]$ induces an isomorphism $\Lambda \otimes \rK(X) \to \rK(\Mod_A^{\gen})$.
\end{corollary}

\begin{proof}
This follows from Corollary~\ref{cor:generic-inj}(c).
\end{proof}

\begin{remark}
Combining the results of the previous several sections, we obtain an equivalence $\Psi \colon \Mod_A^{\gen} \to \Mod_A^0$. Note that this equivalence is \emph{not} $\cV$-linear! Indeed, $\Psi(\bS_{\lambda}(\bV) \otimes A)=J_{\lambda}$, and this is not isomorphic to $\bS_{\lambda}(\bV) \otimes \Psi(A)=\bS_{\lambda}(\bV)$. This computation also shows that the isomorphism $\rK(\Mod_A^{\gen,\fgen}) \to \rK(\Mod_A^{0,\fgen})$ induced by $\Psi$ is not $\Lambda$-linear.
\end{remark}

\begin{remark}
Suppose that $R=\Sym(V)$ is a general polynomial tca, where $V$ is a finite length object of $\cV$. One can then show (by direct calculation) that the subcategory of $\Mod_R^{\gen}$ spanned by the images of the projective objects is equivalent to the subcategory of $\Mod_R^0$ spanned by the injective objects. From this, it follows that there is a unique left exact functor $\Mod_R^0 \to \Mod_R^{\gen}$ taking each injective to the corresponding localized projective. We expect that this functor is an equivalence in general. However, we have only been able to prove this in essentially two cases: the one above, and the case where $V$ is $\Sym^2(\bC^{\infty})$ or $\lw^2(\bC^{\infty})$, which is treated in \cite{sym2noeth}. In each case, it has been essential to use the description of $\Mod_R^{\gen}$ as the representation category of a generic stabilizer; without this, we have not found a way to show that objects in $\Mod_R^{\gen}$ behave as we expect.
\end{remark}

\subsection{The section functor} \label{ss:section}

We now study the section functor $S \colon \Mod_A^{\gen} \to \Mod_A$ using a geometric approach. Let $n \ge \rank(\cE)$ be an integer. Let $\fH_n$ be the space of linear maps $\bC^n \to \cE^*$, thought of as a scheme over $X$; in fact, $\fH_n$ is just $\Spec(A(\bC^n))$. Let $\fU_n$ be the open subscheme of $\fH_n$ where the map is surjective, and write $j \colon \fU_n \to \fH_n$ for the inclusion.

By a {\bf polynomially (resp.\ algebraically) equivariant sheaf} on $\fH_n$ we mean a $\GL_n$-equivariant quasi-coherent sheaf that is a subquotient of a direct sum of sheaves of the form $\cF \otimes V \otimes \cO_{\fH_n}$, where $\cF$ is a quasi-coherent sheaf on $X$ and $V$ is a polynomial (resp.\ algebraic) representation of $\GL_n$. We write $\Mod_{\fH_n}$ (resp.\ $\Mod_{\fH_n}^{\alg}$) for the category of polynomially (resp.\ algebraically) equivariant sheaves. We make similar definitions for $\fU_n$, though we will only use polynomially equivariant sheaves on $\fU_n$.

We can identify $\Mod_{\fH_n}$ (resp.\ $\Mod^{\alg}_{\fH_n}$) with the category of $\GL_n$-equivariant $A(\bC^n)$-modules that decompose as a polynomial (resp.\ algebraic) representation of $\GL_n$. If $V$ is an algebraic representation of $\GL_n$ over $\cO_X$ then it has a maximal polynomial subrepresentation $V^{\pol}$, and the construction $V \mapsto V^{\pol}$ is exact. This construction induces an exact functor $\Mod_{\fH_n}^{\alg} \to \Mod_{\fH_n}$, denoted $M \mapsto M^{\pol}$.

Let $Y=\Gr_d(\bC^n)_X$, let $\pi \colon Y \to X$ be the structure map, and let $\cQ$ be the tautological bundle on $Y$. A point in $\fU_n$ is a surjection $f \colon \bC^n \to \cE^*$. We thus obtain a map $\rho \colon \fU_n \to Y$ by associating to $f$ the quotient $\coker(f)$ of $\bC^n$. In fact, specifying $f$ is the same as specifying an isomorphism $\coker(f) \to \cE^*$, and so we see that $\fU_n$ is identified with the space $\uIsom(\cQ, \cE^*)$ over $\Gr_d(\bC^n)$. In particular, the map $\rho$ is affine: $\uIsom(\cQ, \cE^*)$ is the relative spectrum of the algebra
\begin{displaymath}
\bigoplus_{\lambda} \bS_{\lambda}(\cQ) \otimes \bS_{\lambda}(\cE),
\end{displaymath}
the sum taken over all dominant weights $\lambda$. We thus see that $\rR^n j_*$ can be identified with $\rR^n \pi_* \circ \rho_*$, where we identify $\fH_n$-modules with $A(\bC^n)$-modules on $X$.

\begin{lemma} \label{lem:algcoh}
Let $M$ be a $\GL_n$-equivariant quasi-coherent sheaf on $Y$ that is a subquotient of a direct sum of sheaves of the form $\pi^*\cF \otimes \bS_{\lambda}(\bC^n)$ where $\cF$ is an $\cO_X$-module. Then $\rR^i \pi_*(M)$ is an algebraic representation of $\GL_n$ over $\cO_X$, for any $i$.
\end{lemma}

\begin{proof}
This can be checked locally on $X$, so we can assume $X$ is affine. Write $Y=G/H$ where $G=(\GL_n)_X$ and $H$ is an appropriate parabolic subgroup of $G$. Then $M$ corresponds to an algebraic representation $N$ of $H$ over $\cO_X$. The push-forward $\rR^i \pi_*(M)$ is then identified with the derived induction from $H$ to $G$ of $N$ by \cite[Proposition I.5.12(a)]{jantzen}, which is an algebraic representation of $G$ by definition.
\end{proof}

\begin{lemma}
Let $M \in \Mod_{\fU_n}$. Then $\rR^n j_*(M) \in \Mod_{\fH_n}^{\alg}$.
\end{lemma}

\begin{proof}
By definition, $M$ is a subquotient of a sheaf of the form $V \otimes \cO_{\fU_n}$, where $V$ is a polynomial representation of $\GL_n$ over $\cO_X$. We thus see that $\rho_*(M)$ is a subquotient of
\begin{displaymath}
\rho_*(V \otimes \cO_{\fU_n}) = V \otimes \bigoplus_{\lambda} \bS_{\lambda}(\cQ) \otimes \bS_{\lambda}(\cE).
\end{displaymath}
Thus $\rho_*(M)$ is a $\GL_n$-equivariant quasi-coherent sheaf on $Y$ that is a subquotient of a direct sum of sheaves of the form $\pi^* \cF \otimes \bS_{\nu}(\bC^n)$ where $\cF$ is an $\cO_X$-module. The result now follows from Lemma~\ref{lem:algcoh}.
\end{proof}

\begin{lemma} \label{lem:jsupp}
\begin{enumerate}[\indent \rm (a)]
\item If $M \in \Mod_{\fH_n}$ then the natural map $M \to j_*(j^*(M))^{\pol}$ has kernel and cokernel supported on the complement of $\fU_n$, as does the cokernel of the inclusion $j_*(j^*(M))^{\pol} \to j_*(j^*(M))$.
\item If $M \in \Mod_{\fU_n}$ then the inclusion $j_*(M)^{\pol} \to j_*(M)$ has cokernel supported on the complement of $\fU_n$.
\end{enumerate}
\end{lemma}

\begin{proof}
(a) The map $M \to j_*(j^*(M))$ has kernel and cokernel supported on the complement of $\fU_n$, and factors through the inclusion $j_*(j^*(M))^{\pol} \to j_*(j^*(M))$. The result follows.

(b) If $M$ has the form $j^*(N)$ for $N \in \Mod_{\fH_n}$ then the result follows from (a). Since every object of $\Mod_{\fU_n}$ is, by definition, a subquotient of one this form, it suffices to show that if (b) holds for $M$ then it holds for subs and quotients of $M$. Thus let $M$ be given and let $N$ be a submodule of $M$. Then $j_*(N)$ is a submodule of $j_*(M)$, and $j_*(N)^{\pol} = j_*(N) \cap j_*(M)^{\pol}$. Thus the map $j_*(N)/j_*(N)^{\pol} \to j_*(M)/j_*(M)^{\pol}$ is injective. Since the target is supported on the complement of $\fU_n$, it follows that the source is as well. Now let $N$ be a quotient of $M$. The cokernel of the map $j_*(M) \to j_*(N)$ is then supported on the complement of $\fU_n$, by general theory. Thus the same is true for the cokernel of the map $j_*(M)/j_*(M)^{\pol} \to j_*(N)/j_*(N)^{\pol}$. Since the source is supported on the complement of $\fU_n$, the same is thus true for the target.
\end{proof}

\begin{lemma}
The restriction functor $\Mod_{\fH_n} \to \Mod_{\fU_n}$ identifies $\Mod_{\fU_n}$ with the Serre quotient of $\Mod_{\fH_n}$ by the subcategory of sheaves supported on the complement of $\fU_n$.
\end{lemma}

\begin{proof}
Let $\cC$ be the subcategory of sheaves supported on $\fH_n \setminus \fU_n$. Restriction to the open subscheme $\fU_n$ is an exact functor and annihilates $\cC$, so we get a functor $\Mod_{\fH_n} / \cC \to \Mod_{\fU_n}$. To see that it is faithful, consider a morphism $f \colon M \to N$ of sheaves on $\fH_n$ whose restriction to $\fU_n$ is $0$. This means that the image of $f$ is supported on $\fH_n \setminus \fU_n$, so $f=0$ in the Serre quotient $\Mod_{\fH_n} / \cC$. To get fullness, let $g \colon j^*(M) \to j^*(N)$ be a morphism of sheaves. Then we get $j_* g \colon j_*(j^*(M)) \to j_*(j^*(N))$, which induces a map $g' \colon j_*(j^*(M))^{\pol} \to j_*(j^*(N))^{\pol}$. By Lemma~\ref{lem:jsupp}(a), we have $g=j^*(g')$. Also by Lemma~\ref{lem:jsupp}(a), the map $M \to j_*(j^*(M))^{\pol}$ is an isomorphism in $\Mod_{\fH_n}/\cC$, and similarly for $N$, and so $g'$ actually defines a map $g' \colon M \to N$ in the quotient category. Finally, for essential surjectivity, let $M \in \Mod_{\fU_n}$ be given. By Lemma~\ref{lem:jsupp}(b), the natural map $j^*(j_*(M)^{\pol}) \to j^*(j_*(M))$ is an isomorphism. By general theory, there is a natural isomorphism $j^*(j_*(M)) \to M$. We thus see that $M \cong j^*(N)$ where $N=j_*(M)^{\pol}$ is an object of $\Mod_{\fH_n}$.
\end{proof}

Suppose $M$ is an $A$-module. Then $M(\bC^n)$ is a polynomially $\GL_n$-equivariant $A(\bC^n)$-module, and thus defines an $\fH_n$-module. This gives an exact functor $\Mod_A \to \Mod_{\fH_n}$.

\begin{lemma} \label{lem:geogen}
We have a commutative $($up to isomorphism$)$ diagram
\begin{displaymath}
\xymatrix{
\Mod_A \ar[r]^-T \ar[d] & \Mod_A^{\gen} \ar@{..>}[d] \\
\Mod_{\fH_n} \ar[r]^{j^*} & \Mod_{\fU_n} }
\end{displaymath}
where the left map is $M \mapsto M(\bC^n)$.
\end{lemma}

\begin{proof}
By the universal property of Serre quotients, it suffices to show that if $M$ is an $A$-module with $T(M)=0$ then the $\fH_n$-module $M(\bC^n)$ is supported on the complement of $\fU_n$. If $T(M)=0$ then the annihilator $\fa$ of $M$ is non-zero, and so $\fa(\bC^n)$ is also non-zero (since $n \ge \rank(\cE) = \ell(A)$). We thus see that the support of $M(\bC^n)$ is a proper closed subset of $\fH_n$. It is therefore contained in the complement of $\fU_n$, as this is the maximal proper closed $\GL_n$-stable subset.
\end{proof}

The following theorem is the key to our understanding of the saturation functor and its derived functors.

\begin{theorem} \label{thm:geoS}
We have a diagram
\begin{displaymath}
\xymatrix{
\Mod_A^{\gen} \ar[rr]^{\rR^n S} \ar[d] && \Mod_A \ar[d] \\
\Mod_{\fU_n} \ar[rr]^{(\rR^n j_*)^{\pol}} && \Mod_{\fH_n} }
\end{displaymath}
that commutes up to a canonical isomorphism. Here the vertical maps are as in Lemma~\ref{lem:geogen}.
\end{theorem}

\addtocounter{equation}{-1}
\begin{subequations}
\begin{proof}
In this proof, a torsion $A$-module is one localizing to~0 in $\Mod_A^{\gen}$, and a torsion $\fH_n$-module is one restricting to~0 on $\fU_n$.

We first construct a canonical injection $S(M)(\bC^n) \to j_*(M(\bC^n))$ for $M \in \Mod_A^{\gen}$. First suppose that $N$ is an $A$-module. Then $\Sigma(N)$ is torsion-free and so $\Sigma(N)(\bC^n)$ is as well. Thus $N(\bC^n) \to \Sigma(N)(\bC^n)$ is a map from $N(\bC^n)$ to a torsion-free object with torsion kernel and cokernel. However, $N(\bC^n) \to j_*(N(\bC^n) \vert_{\fU_n})$ is the universal such map, and so we obtain a canonical map $\Sigma(N)(\bC^n) \to j_*(N(\bC^n) \vert_{\fU_n})$, which is necessarily injective. Now, let $M=T(N)$. Then $\Sigma(N)=S(M)$ and $N(\bC^n) \vert_{\fU_n}=M(\bC^n)$, by definition. We thus obtain the desired map.

We now claim that the map just constructed is an isomorphism if $M=T(V \otimes A)$, with $V \in \cV_X$. We have maps
\begin{displaymath}
(V \otimes A)(\bC^n) \to S(M)(\bC^n) \to j_*(M(\bC^n)).
\end{displaymath}
Since the second map is injective, it suffices to show that the composite map is an isomorphism. For this, we will compute the rightmost object. The $\fH_n$-module $(V \otimes A)(\bC^n)$ is $V(\bC^n) \otimes \cO_{\fH_n}$, and so we see that the $\fU_n$-module $M(\bC^n)$ is $V(\bC^n) \otimes \cO_{\fU_n}$. We thus have
\begin{equation} \label{eq:geoS}
\rho_*(M(\bC^n)) = V(\bC^n) \otimes \bigoplus_{\lambda} \bS_{\lambda}(\cQ) \otimes \bS_{\lambda}(\cE)
\end{equation}
where the sum is over all dominant weights $\lambda$. We now apply $\pi_*$, and use the fact that $\pi_*(\bS_{\lambda}(\cQ))=\bS_{\lambda}(\bC^n)$ if $\lambda$ is a partition and $\pi_*(\bS_{\lambda}(\cQ))=0$ otherwise. We obtain
\begin{displaymath}
j_*(M(\bC^n)) = V(\bC^n) \otimes \bigoplus_{\lambda} \bS_{\lambda}(\bC^n) \otimes \bS_{\lambda}(\cE) = V(\bC^n) \otimes A(\bC^n)
\end{displaymath}
where now the sum is over all partitions $\lambda$. We leave to the reader the verification that the natural map $(V \otimes A)(\bC^n) \to j_*(M(\bC^n))$ is the identity with the above identification.

We now claim that the map $S(M)(\bC^n) \to j_*(M(\bC^n))$ is an isomorphism for all $M \in \Mod_A^{\gen}$. To see this, let $M$ be given and choose an exact sequence
\begin{displaymath}
0 \to M \to I^0 \to I^1
\end{displaymath}
where $I^0$ and $I^1$ have the form $T(V \otimes A)$ for $V \in \cV_X$. We then obtain a commutative diagram
\begin{displaymath}
\xymatrix{
0 \to S(M)(\bC^n) \ar[r] \ar[d] & S(I^0)(\bC^n) \ar[r] \ar[d] & S(I^1)(\bC^n) \ar[d] \\
0 \to j_*(M(\bC^n)) \ar[r] & j_*(I^0(\bC^n)) \ar[r] & j_*(I^1(\bC^n)) }
\end{displaymath}
with exact rows. Since the right two vertical maps are isomorphisms, so is the left vertical map.

We have thus proved the result for $n=0$. (In fact, we showed that one does not even need to take the polynomial piece in this case.) We now prove the result for arbitrary $n$. The functors $\rR^n j_* \colon \Mod_{\fU_n} \to \Mod_{\fH_n}^{\alg}$ form a cohomological $\delta$-functor. Since formation of the polynomial subrepresentation is exact on $\Mod_{\fH_n}^{\alg}$, it follows that the functors $(\rR^n j_*)^{\pol} \colon \Mod_{\fU_n} \to \Mod_{\fH_n}$ also form a cohomological $\delta$-functor. Since evaluation on $\bC^n$ is exact, the functors $(\rR^n S(-))(\bC^n)$ and $\rR^n j_*((-)(\bC^n))^{\pol}$ are both cohomological $\delta$-functors $\Mod_A^{\gen} \to \Mod_{\fH_n}$. The first is clearly universal, since the higher derived functors kill injective objects of $\Mod_A^{\gen}$. Thus to prove the result, it suffices to show that the second one is universal, and for this it suffices to show that it is coeffaceable. Since every object of $\Mod_A^{\gen}$ injects into an object of the form $M=T(V \otimes A)$ with $V \in \cV_X$, it suffices to show that $\rR^n j_*(M(\bC^n))^{\pol}=0$ for $n>0$. Applying $\rR^n \pi_*$ to \eqref{eq:geoS}, and using the projection formula, we find
\begin{displaymath}
\rR^n j_*(M(\bC^n)) = V(\bC^n) \otimes \bigoplus_{\lambda} \bS_{\lambda}(\cE) \otimes \rR^n \pi_*(\bS_{\lambda}(\cQ)).
\end{displaymath}
We now come to the point: $\rR^n \pi_*(\bS_{\lambda}(\cQ))^{\pol}=0$ for all $n>0$. Indeed, if $\lambda$ is a partition then $\rR^n \pi_*(\bS_{\lambda}(\cQ))=0$ for $n>0$. Now suppose $\lambda$ is not a partition. By Borel--Weil--Bott (Theorem~\ref{thm:BWB}), either $\rR^n \pi_*(\bS_{\lambda}(\cQ))$ vanishes for all $n$, or vanishes for all $n \ne n_0$ and for $n=n_0$ has the form $\bS_{\nu}(\bC^n)$. In the latter case, $\nu = \sigma \bullet \lambda'$ where $\lambda' = (\lambda_1, \ldots, \lambda_d, 0, \ldots, 0) \in \bZ^n$ and $\sigma \in S_n$. If $\lambda$ is not a partition then $\lambda_d<0$, and it is clear from the formulation of Theorem~\ref{thm:BWB} that $\nu$ has a negative entry. Thus $\nu$ is not a partition, and so $\bS_{\nu}(\bC^n)^{\pol}=0$.
\end{proof}
\end{subequations}

\begin{corollary}
The functor $\rR^n S$ is $\cV_X$-linear. Precisely, if $V \in \cV_X$ is $\cO_X$-flat then there is a canonical isomorphism $\rR^n S(V \otimes M) = V \otimes \rR^n S(M)$ for all $M \in \Mod_A^{\gen}$. More generally, we have a canonical isomorphism $\rR S(V \stackrel{\rL}{\otimes}_{\cO_X} M)=V \stackrel{\rL}{\otimes}_{\cO_X} \rR S(M)$.
\end{corollary}

\begin{proof}
Let $V_\bullet$ be an $\cO_X$-flat complex quasi-isomorphic to $V$ and let $M \to I(M)$ be an injective resolution of $M$ in $\Mod_A^\gen$. For each $i$, choose an injective resolution $V_i \otimes M \to I(V_i \otimes M)$, the iterated mapping cone of these complexes is denoted $I(V_\bullet \otimes M)$. Then both $V_\bullet \otimes I(M)$ and $I(V_\bullet \otimes M)$ are quasi-isomorphic to $V \otimes M$ and we can lift the identity map on $V \otimes M$ to get a morphism $V_\bullet \otimes I(M) \to I(V_\bullet \otimes M)$. Apply $S$ to both sides to get
\[
V \stackrel{\rL}{\otimes}_{\cO_X} \rR S(M) \to \rR S(V \stackrel{\rL}{\otimes}_{\cO_X} M).
\]
Now evaluate this map on $\bC^n$. By Theorem~\ref{thm:geoS}, this replaces $\Mod_A^\gen$ by $\Mod_{\fU_n}$ and $\Mod_A$ by $\Mod_{\fH_n}$, in which case $\rR S$ is identified with $\rR j_*$. Then the map above is a quasi-isomorphism by the usual projection formula.
\end{proof}

\begin{corollary} \label{cor:dersat}
If $V \in \cV_X$ then $A \otimes V$ is derived saturated, that is, the natural map $A \otimes V \to \rR \Sigma(A \otimes V)$ is an isomorphism.
\end{corollary}

\begin{proof}
It suffices to show this after evaluating on $\bC^n$ for all $n$. This was shown in the course of the proof of Theorem~\ref{thm:geoS}.
\end{proof}

\begin{corollary}
If $V \in \cV_X$ is injective then $V \otimes A$ is an injective $A$-module. In particular, if $X=\Spec(\bC)$ then all projective $A$-modules are also injective.
\end{corollary}

\begin{proof}
The functor $S$ takes injectives to injectives. By Corollary~\ref{cor:generic-inj}, $T(V \otimes A)$ is injective in $\Mod_A^{\gen}$, and we have just shown that $S(T(V \otimes A))=\Sigma(V \otimes A)$ is $V \otimes A$.
\end{proof}

\begin{corollary} \label{cor:Sfin}
If $M \in \Mod_A^{\gen}$ is finitely generated then $\rR S(M)$ is represented by a finite length complex of modules of the form $V \otimes A$ with $V \in \cV_X^{\fgen}$. In particular, $\rR^n S(M)$ is finitely generated for all $n \ge 0$ and vanishes for $n \gg 0$.
\end{corollary}

\begin{proof}
Using Corollary~\ref{cor:generic-inj}, pick a resolution $M \to T(V^{\bullet} \otimes A)$ where $V^i \in \cV_X^{\fgen}$ and $V^i=0$ for $i \gg 0$. Since $V^i \otimes A$ is $\Sigma$-acyclic, it follows that $T(V^i \otimes A)$ is $S$-acyclic. We can thus use this resolution to compute $\rR S$. Since $\Sigma(V^i \otimes A)=V^i \otimes A$, we obtain a quasi-isomorphism $\rR S(M) \to V^{\bullet} \otimes A$. The result follows.
\end{proof}

Finally, we compute the derived saturation of the objects $\bS_{\mu}(\cK)$. Recall that $\cK$ is the kernel of the canonical map $\bV \otimes A \to \cE^* \otimes A$ in $\Mod_A^{\gen}$. For a weight $\lambda$ and partition $\nu$, write $\lambda \xrightarrow{n} \nu$ if Bott's algorithm applied to $\lambda$ terminates after $n$ steps on $\nu$, see Remark~\ref{rmk:bott-algorithm}.

\begin{corollary}
For a partition $\mu$, we have
\begin{displaymath}
\rR^i S (\bS_{\mu}(\cK)) = \bigoplus_{[\lambda,\mu] \xrightarrow{i} \nu} \bS_{\nu}(\bV) \otimes \bS_{\lambda}(\cE),
\end{displaymath}
where the sum is over all partitions $\lambda$ and $\nu$ with $\ell(\lambda) \le d$ and thus related, and $[\lambda,\mu]$ is the weight $(\lambda_1, \ldots, \lambda_d, \mu_1, \mu_2, \ldots)$. In particular,
\begin{displaymath}
S(\bS_{\mu}(\cK)) = \bigoplus_{\lambda_d \ge \mu_1} \bS_{[\lambda,\mu]}(\bV) \otimes \bS_{\lambda}(\cE).
\end{displaymath}
\end{corollary}

\begin{proof}
Take $n \ge \rank \cE + \ell(\mu)$. Using Theorem~\ref{thm:geoS}, we have
\[
\rR^i S (\bS_\mu(\cK))(\bC^n) = (\rR^i j_* \bS_\mu(\cK(\bC^n)))^{\pol}
\]
where $j \colon \fU_n \to \fH_n$ is the inclusion. As discussed above, we have a factorization $\rR^i j_* = \rR^i \pi_* \circ \rho_*$ where $\pi \colon \Gr_d(\bC^n)_X \to X$ is the structure map and $\rho \colon \fU_n \to \Gr_d(\bC^n)_X$ sends a map in $\fU_n$ to its cokernel. Note that $\cK(\bC^n) = \rho^* \cR$, and since pullback commutes with tensor operations, we get
\[
\rho_* ((\bS_\mu \cK)(\bC^n)) = \bS_\mu \cR \otimes \rho_* \cO_{\fU_n} = \bS_\mu \cR \otimes \bigoplus_\lambda \bS_\lambda \cQ \otimes \bS_\lambda \cE
\]
where the sum is over all dominant weights $\lambda$. Hence, the desired result follows from Borel--Weil--Bott (Theorem~\ref{thm:BWB}), noting that any $\lambda$ with negative entries are deleted from the final computation since we need to take the polynomial piece.
\end{proof}

\begin{remark}
When $d=1$ and $X = \Spec(\bC)$, this essentially recovers \cite[Proposition~7.4.3]{symc1}. To be precise, Corollary~\ref{cor:generic-inj}(d) says that the $\bS_\mu \cK$ are the simple objects of $\Mod_A^\gen$, so they are the simple objects $L_\mu$ defined in \cite{symc1}. The local cohomology calculation there for $i \ge 2$ agrees with $\rR^i S$ by \cite[Corollary 4.4.3]{symc1}, and the discussion in \cite[\S 7.4]{symc1} connects the border strip combinatorics mentioned there with Borel--Weil--Bott.
\end{remark}

\section{Rank subquotient categories} \label{s:rank}

\subsection{Set-up}

We fix, for all of \S \ref{s:rank}, a scheme $X$ over $\bC$ (noetherian, separated, and of finite Krull dimension, as always) and a vector bundle $\cE$ of rank $d$ on $X$. We let $A=\bA(\cE)$.

We introduce some notation mirroring that from \S \ref{ss:formalism2}. We write $\Mod_{A,\le r}$ for the category of $A$-modules supported on $V(\fa_r)$, i.e., that are locally annihilated by powers of $\fa_r$. This gives an ascending chain of Serre subcategories
\begin{displaymath}
\Mod_{A,\le 0} \subset \Mod_{A,\le 1} \subset \cdots \subset \Mod_{A,\le d}=\Mod_A
\end{displaymath}
that we refer to as the {\bf rank stratification}. We define quotient categories 
\begin{align*}
\Mod_{A,>r} &= \Mod_A / \Mod_{A,\le r}\\
\Mod_{A,r} &= \Mod_{A,\le r}/\Mod_{A,\le r-1}.
\end{align*}
We let $T_{>r} \colon \Mod_A \to \Mod_{A,>r}$ be the localization functor and $S_{>r}$ its right adjoint. We put $\Sigma_{>r}=S_{>r} \circ T_{>r}$, as usual, and let $\Gamma_{\le r} \colon \Mod_A \to \Mod_{A,\le r}$ be the functor that assigns to a module the maximal submodule supported on $V(\fa_r)$.

We let $\rD(A)_{\le r}$ be the full subcategory of $\rD(A)$ on objects $M$ such that $\rR \Sigma_{>r}(M)=0$, and we let $\rD(A)_{>r}$ be the full subcategory on objects $M$ such that $\rR \Gamma_{\le r}(M)=0$. We also put $\rD(A)_r=\rD(A)_{\le r} \cap \rD(A)_{\ge r}$. These are all triangulated subcategories of $\rD(A)$. By \S \ref{s:formalism}, we have a semi-orthogonal decomposition $\rD^+(A)=\langle \rD^+(A)_0, \ldots, \rD^+(A)_d \rangle$.

\subsection{The category $\Mod_{A,r}[\fa_r]$} \label{ss:modar}

We let $\Mod_A[\fa_r]$ be the category of $A$-modules annihilated by $\fa_r$. This is a subcategory of $\Mod_{A,\le r}$. We let $\Mod_{A,r}[\fa_r]$ be the subcategory of $\Mod_{A,r}$ on objects of the form $T_{\ge r}(M)$, where $M$ is an $A$-module such that $\fa_r M$ is supported on $V(\fa_{r-1})$. Obviously, $T_{\ge r}$ carries $\Mod_A[\fa_r]$ into $\Mod_{A,r}[\fa_r]$. In fact:

\begin{proposition}
The functor $T_{\ge r} \colon \Mod_A[\fa_r] \to \Mod_{A,r}[\fa_r]$ identifies $\Mod_{A,r}[\fa_r]$ with the Serre quotient of $\Mod_A[\fa_r]$ by $\Mod_{A,<r}[\fa_r]$.
\end{proposition}

\begin{proof}
The functor $T_{\ge r}$ is exact and kills $\Mod_{A,<r}[\fa_r]$, and thus induces a functor
\begin{displaymath}
\Phi \colon \frac{\Mod_A[\fa_r]}{\Mod_{A,<r}[\fa_r]} \to \Mod_{A,r}[\fa_r].
\end{displaymath}
We must show that the functor $\Phi$ is an equivalence. We write $\cC$ for the domain of $\Phi$.

We first claim that the functor $\Phi$ is essentially surjective. It suffices to show that the functor $T_{\ge r} \colon \Mod_A[\fa_r] \to \Mod_{A,r}[\fa_r]$ is essentially surjective. Thus let $T_{\ge r}(M) \in \Mod_{A,r}[\fa_r]$ be a typical object, so that $M$ is an $A$-module such that $\fa_r M$ is supported on $V(\fa_{r-1})$. Then $\ol{M}=M/\fa_r M$ belongs to $\Mod_A[\fa_r]$. Since the map $M \to \ol{M}$ is surjective and has kernel supported on $V(\fa_{r-1})$, it follows that $T_{\ge r}(M) \to T_{\ge r}(\ol{M})$ is an isomorphism. Since $\ol{M} \in \Mod_A[\fa_r]$, this establishes the claim.

We now show that $\Phi$ is fully faithful. Let $M, N \in \Mod_A[\fa_r]$. Let $T \colon \Mod_A[\fa_r] \to \Mod_{A,r}[\fa_r]$ be the localization functor. Then
\begin{displaymath}
\Hom_{\Mod_{A,r}[\fa_r]}(T(M), T(N)) = \varinjlim \Hom_{\Mod_A[\fa_r]}(M', N'),
\end{displaymath}
where the colimit is over $M' \subset M$ such that $M/M' \in \Mod_{A,<r}[\fa_r]$ and quotients $N \to N'$ with kernel in $\Mod_{A,<r}[\fa_r]$. On the other hand,
\begin{displaymath}
\Hom_{\Mod_{A,r}[\fa_r]}(T_{\ge r}(M), T_{\ge r}(N)) = \varinjlim \Hom_A(M'', N''),
\end{displaymath}
where the colimit is over $M'' \subset M$ such that $M/M'' \in \Mod_{A,<r}$ and quotients $N \to N''$ with kernel in $\Mod_{A,<r}$. Since $M$ and $N$ are killed by $\fa_r$, it follows that $M/M''$ and $\ker(N \to N'')$ are as well, and so this colimit is exactly the same as the previous one.
\end{proof}

We write $\ol{S}_{\ge r} \colon \Mod_{A,r}[\fa_r] \to \Mod_A[\fa_r]$ for the right adjoint of the localization functor $T_{\ge r}$ appearing in the proposition. We write $\rR \ol{S}_{\ge r}$ for the derived functor of $\ol{S}_{\ge r}$. The notation $\rR S_{\ge r}$ always means the derived functor of $S_{\ge r} \colon \Mod_{A,\ge r} \to \Mod_A$. Thus for $M \in \Mod_{A,r}[\fa_r]$ one computes $\rR \ol{S}_{\ge r}(M)$ by using an injective resolution of $M$ in the category $\Mod_{A,r}[\fa_r]$, while one computes $\rR S_{\ge r}(M)$ by using an injective resolution in $\Mod_{A,\ge r}$ (or simply $\Mod_{A,r}$). Injective objects in these two categories are quite different; nonetheless, we have:

\begin{proposition} \label{prop:equivS}
The functor $\rR \ol{S}_{\ge r}$ is isomorphic to the restriction of $\rR S_{\ge r}$ to the derived category of $\Mod_{A,r}[\fa_r]$.
\end{proposition}

\begin{proof}
Let $I_0$ be an injective of $\Mod_{A,r}[\fa_r]$, and put $I=\ol{S}_{\ge r}(I_0)$, an injective of $\Mod_A[\fa_r]$. It suffices to show that the map $I \to \rR \Sigma_{\ge r}(I)$ is an isomorphism. Indeed, suppose this is the case. Then $I_0=T_{\ge r}(I)$ is $S_{\ge r}$-acyclic and satisfies $S_{\ge r}(I_0)=I=\ol{S}_{\ge r}(I_0)$. Thus if $M \to I^{\bullet}$ is an injective resolution in $\Mod_{A,r}[\fa_r]$ then $S_{\ge r}(I^{\bullet})$ computes $\rR S_{\ge r}(M)$, since the objects $I^k$ are $S_{\ge r}$-acyclic, and equals $\ol{S}_{\ge r}(I^{\bullet})$, which computes $\rR \ol{S}_{\ge r}$.

To prove that $I \to \rR \Sigma_{\ge r}(I)$ is an isomorphism, it suffices (by Proposition~\ref{prop:satcrit}) to show $\Ext^j_A(N,I)=0$ for all $N \in \Mod_{A,<r}$ and $j \ge 0$. We first treat the $j=0$ case, i.e., we show that any map $N \to I$ with $N \in \Mod_{A,<r}$ is zero. It suffices to treat the case where $N$ is finitely generated and thus annihilated by a power of $\fa_{r-1}$. By d\'evissage, we can assume $\fa_{r-1} N=0$. But then $\fa_r N=0$ as well, and so $N \in \Mod_{A,<r}[\fa_r]$. Since $I$ is saturated with respect to this category, the result follows.

We now consider the case $j>0$. Since $I$ is an $A/\fa_r$-module, derived adjunction gives
\begin{displaymath}
\rR \Hom_A(N,I)=\rR \Hom_{A/\fa_r}(N \stackrel{\rL}{\otimes}_A A/\fa_r, I).
\end{displaymath}
As $I$ is injective as an $A/\fa_r$-module, this $\rR \Hom$ can be changed to $\Hom$. We find
\begin{displaymath}
\Ext^j_A(N,I)=\Hom_{A/\fa_r}(\Tor^A_j(N, A/\fa_r), I)
\end{displaymath}
Since $N$ is supported on $V(\fa_{r-1})$, so are the $\Tor$'s. Thus, by the $j=0$ case, the above $\Hom$ vanishes. This completes the proof.
\end{proof}

\begin{corollary}
Let $M$ be an $A$-module annihilated by $\fa_r$. Then $\rR^n \Sigma_{\ge r}(M)$ is annihilated by $\fa_r$ for all $n \ge 0$.
\end{corollary}

\begin{proof}
Indeed, $\rR^n \Sigma_{\ge r}(M)$ is by definition $\rR^n S_{\ge r}(T_{\ge r}(M))$, which by Proposition~\ref{prop:equivS} is identified with $\rR^n \ol{S}_{\ge r}(T_{\ge r}(M))$, and $\ol{S}_{\ge r}$ (and its derived functors) take values in $\Mod_A[\fa_r]$.
\end{proof}

\begin{remark}
We note that, a priori, $\rR^n \Sigma_{\ge r}(M)$ is supported on $V(\fa_{r-1})$ for $n>0$, and thus locally annihilated by a power of $\fa_{r-1}$. However, this does not directly imply that $\rR^n \Sigma_{\ge r}(M)$ is annihilated by $\fa_r$.
\end{remark}

We now give a complete description of the category $\Mod_{A,r}[\fa_r]$. Let $Y=\Gr_r(\cE)$ be the Grassmannian of rank $r$ quotients of $\cE$. Let $\pi \colon Y \to X$ be the natural map, and let $\cQ$ be the tautological rank $r$ quotient bundle of $\pi^*(\cE)$. Let $B=\bA(\cQ)$. We let $S'$ and $T'$ be the usual functors between $\Mod_B$ and $\Mod_B^{\gen}$. We have a natural map $\pi^*(A) \to B$, which induces a functor $\Phi \colon \Mod_A \to \Mod_B$ via $M \mapsto \pi^*(M) \otimes_{\pi^*(A)} B$.

\begin{theorem} \label{thm:ABequiv}
The functor $T' \circ \Phi \colon \Mod_A[\fa_r] \to \Mod_B^{\gen}$ is exact and kills $\Mod_{A,<r}[\fa_r]$. The induced functor
\begin{displaymath}
\Psi \colon \Mod_{A,r}[\fa_r] \to \Mod_B^{\rm gen}
\end{displaymath}
is an equivalence and compatible with tensor products.
\end{theorem}

\begin{proof}
Let $\fH_n$ be defined as in \S \ref{ss:section}, let $\fH_n^{\le r}$ be the closed subscheme defined by $\fa_r(\bC^n)$, and let $\fH_n^{=r}$ be the complement of $\fH_n^{\le r-1}$ in $\fH_n^{\le r}$. Specialization defines a functor $\Mod_A[\fa_r] \to \Mod_{\fH_n^{\le r}}$, which induces a functor $\Mod_{A,r}[\fa_r] \to \Mod_{\fH_n^{=r}}$, just as in Lemma~\ref{lem:geogen}. Let $\fU'_n$ be defined like $\fU_n$ as in \S \ref{ss:section} but with respect to $B$; thus $\fU'_n$ is the scheme of surjections $\bC^n \to \cQ^*$. There is an isomorphism of schemes $\fU'_n \to \fH_n^{=r}$, since a map $\bC^n \to \cE^*$ of rank $r$ determines a rank $r$ quotient of $\cE^*$. Consider the diagram
\begin{displaymath}
\xymatrix{
\Mod_A[\fa_r] \ar[r] \ar[d] \ar@/^2em/[rr]^{T' \circ \Phi} & \Mod_{A,r}[\fa_r] \ar@{..>}[r]^-{\Psi} \ar[d] & \Mod_B^{\gen} \ar[d] \\
\Mod_{\fH_n^{\le r}} \ar[r] & \Mod_{\fH_n^{=r}} \ar[r] & \Mod_{\fU'_n} }
\end{displaymath}
Both functors in the bottom row are exact. It follows that $T' \circ \Phi$ is exact. Indeed, it is right exact, so it suffices to verify that it preserves injections. If $M \to N$ were an injection such that $T'(\Phi(N)) \to T'(\Phi(M))$ were not injective, then for $n \gg 0$ the specialization of the kernel to $\bC^n$ would be a non-zero object of $\Mod_{\fU_n'}$, contradicting exactness of the bottom row. Thus $T' \circ \Phi$ is exact. It follows from the above diagram that $T' \circ \Phi$ kills $\Mod_{A,<r}[\fa_r]$: indeed, if $M$ were in this category then its specialization to $\bC^n$ would restrict to~0 on $\fH_n^{=r}$ for all $n$, and so $T'(\Phi(M))=0$. We thus get the induced functor $\Psi$ as in the diagram.

We first show that $\Psi$ is fully faithful. Let $M,N \in \Mod_{A,r}[\fa_r]$ be finitely generated, and thus bounded. To verify that $\hom_{\Mod_{A,r}[\fa_r]}(M,N) \to \hom_{\Mod_B^{\gen}}(\Psi (M), \Psi(N))$ is an isomorphism, we can do so after specializing to $\bC^n$ for $n$ sufficiently large. But this is clear, since the bottom right map in the above diagram is an equivalence.

We now claim that every object of $\Mod_B^{\gen}$ is a quotient of one of the form $T'(\pi^*(V) \otimes B)$ with $V \in \cV_X$. By definition, a $B$-module is a quotient of $W \otimes B$ for some $W \in \cV_Y$. It thus suffices to show that if $\cF$ is an $\cO_Y$-module then $T'(\cF \otimes B)$ is a quotient of $T'(\pi^*(V) \otimes B)$ for some $V \in \cV_X$. We note that the natural map $T'(\bV \otimes B) \to T'(\cQ^* \otimes B)$ is surjective. Indeed, under the equivalence $\Mod_B^{\gen}=\Mod_B^0$, this corresponds to the natural surjection $\cQ^* \oplus \bV \to \cQ^*$. We thus have a surjection $T'(\bV^{\otimes r} \otimes B) \to T'(\cL^* \otimes B)$, where $\cL=\lw^r(\cQ)$. Since $\cL$ is an ample line bundle relative to $X$, any $\cO_Y$-module $\cF$ can be written as a quotient of a sum of $\cO_Y$-modules of the form $\pi^*(\cG) \otimes (\cL^*)^{\otimes n}$ where $\cG$ is an $\cO_X$-module and $n>0$ is an integer. In this way, we obtain a surjection $T'(\pi^*(V) \otimes B) \to T'(\cF \otimes B)$ where $V \in \cV_X$.

We now verify that $\Psi$ is essentially surjective. Let $M \in \Mod_B^{\gen}$ be given. Choose a presentation
\begin{displaymath}
T'(\pi^*(W) \otimes B) \stackrel{f}{\to} T'(\pi^*(V) \otimes B) \to M \to 0
\end{displaymath}
with $V,W \in \cV_X$, which is possible by the previous paragraph. Since $\Psi$ is fully faithful, we can write $f=\Psi(g)$ for some morphism $g \colon W \otimes A/\fa_r \to V \otimes A/\fa_r$ in $\Mod_{A,r}[\fa_r]$. Since $\Psi$ is exact, we have $M=\Psi(\coker(g))$.

It is clear from the construction that $\Psi$ is a compatible with tensor products.
\end{proof}

\begin{proposition} \label{prop:pi-adjoint}
The functor $\pi_* \colon \Mod_B \to \Mod_A[\fa_r]$ is the right adjoint to the functor $\Phi \colon \Mod_A[\fa_r] \to \Mod_B$. Moreover, $\rR \pi_*$ is the derived functor of $\pi_*$ on $\Mod_B$.
\end{proposition}

\begin{proof}
We can identify $\Mod_A$ and $\Mod_B$ with categories of quasi-coherent sheaves on the schemes $\Spec(A(\bV))$ and $\Spec(B(\bV))$. The map $\pi^*(A) \to B$ induces a map $f \colon \Spec(B(\bV)) \to \Spec(A(\bV))$. Under the previous identifications, $\pi_*$ corresponds to $f_*$ and $\Phi$ to $f^*$. The adjointness statement follows from the usual adjointness of $f_*$ and $f^*$.

We now show that the $\rR\pi_*$ is the derived functor of $\pi_*$ on $\Mod_B$. It suffices to show that injective $B$-modules are $\pi_*$-acyclic. Thus let $I$ be an injective $B$-module. Then each multiplicity space $I_{\lambda}$ is injective as an $\cO_X$-module by Proposition~\ref{prop:injOX}, and therefore acyclic for $\pi_*$  (see \cite[Tag 0BDY]{stacks}). Since $\pi_*$ is computed on $\cV_X$ simply by applying $\pi_*$ to each multiplicity space, it follows that $I$ is $\pi_*$-acyclic.
\end{proof}

The following diagram summarizes the picture:
\begin{displaymath}
\xymatrix@C=90pt@R=35pt{
\Mod_A[\fa_r] \ar@<3pt>[d]^{T_{\ge r}} \ar@<3pt>[r]^{\Phi} &
\Mod_B \ar@<3pt>[l]^{\pi_*} \ar@<3pt>[d]^{T'} \\
\Mod_{A,r}[\fa_r] \ar[r]^{\Psi} \ar@<3pt>[u]^{S_{\ge r}} &
\Mod_B^{\gen} \ar@<3pt>[u]^{S'} }
\end{displaymath}

\begin{lemma} \label{lem:Sr-formula}
$S_{\ge r} = \pi_* \circ S' \circ \Psi$.
\end{lemma}

\begin{proof}
The two paths from $\Mod_A[\fa_r]$ to $\Mod_B^{\gen}$ commute by definition of $\Psi$. Since $\Psi$ is an equivalence, it follows that $S_{\ge r} \circ \Psi^{-1}$ is the right adjoint to $\Psi \circ T_{\ge r}$. On the other hand, since $S'$ is right adjoint to $T'$ and $\pi_*$ is right adjoint to $\Phi$, it follows that $\pi_* \circ S'$ is right adjoint to $T' \circ \Phi$. Thus the two paths from $\Mod_B^{\gen}$ to $\Mod_A[\fa_r]$ (one of which uses the undrawn $\Psi^{-1}$) also agree.
\end{proof}

\begin{proposition} \label{prop:Srformula}
Let $M \in \Mod_{A,r}[\fa_r]$, and let $N=\Psi(M)$ be the corresponding object of $\Mod_B^{\gen}$. Then $\rR S_{\ge r}(M)$ is canonically isomorphic to $\rR \pi_*(\rR S'(N))$.
\end{proposition}

\begin{proof}
By Lemma~\ref{lem:Sr-formula}, $S_{\ge r} = \pi_* \circ S' \circ \Psi$. We thus see that $\rR S_{\ge r} = \rR \pi_* \circ \rR S' \circ \Psi$. Here we have used the fact that $\rR \pi_*$ is the derived functor of $\pi_*$ on $\Mod_B$ (Proposition~\ref{prop:pi-adjoint}) and the fact that $\rR S_{\ge r}$ is the derived functor of $S_{\ge r}$ on $\Mod_{A,r}[\fa_r]$ (Proposition~\ref{prop:equivS}).
\end{proof}

\begin{corollary} \label{cor:Srfin}
Let $M \in \Mod_{A,r}^{\fgen}$. Then $\rR^n S_{\ge r}(M)$ is a finitely generated $A$-module for all $n \ge 0$, and vanishes for $n \gg 0$.
\end{corollary}

\begin{proof}
By d\'evissage, we can reduce to the case $M \in \Mod_{A,r}^{\fgen}[\fa_r]$. By Proposition~\ref{prop:Srformula}, we have $\rR S_{\ge r}(M)=\rR \pi_*(\rR S'(N))$, where $N=\Psi(M)$ is a finitely generated object of $\Mod_B^{\gen}$. Since $\rR S'$ carries $\rD^b_{\fgen}(\Mod_B^{\gen})$ into $\rD^b_{\fgen}(\Mod_B)$ (Corollary~\ref{cor:Sfin}) and $\rR \pi_*$ carries $\rD^b_{\fgen}(\Mod_B)$ to $\rD^b_{\fgen}(\Mod_A)$ (Corollary~\ref{cor:finpushfwd}), the result follows.
\end{proof}

\subsection{Finiteness of local cohomology and derived saturation}

The following theorem is one of the fundamental results of this paper.

\begin{theorem} \label{thm:satfin}
Let $M \in \rD^b_{\fgen}(A)$. Then $\rR \Sigma_{>r}(M)$ and $\rR \Gamma_{\le r}(M)$ also belong to $\rD^b_{\fgen}(A)$.
\end{theorem}

\begin{proof}
We proceed by descending induction on $r$. When $r=d$, we have that $\Gamma_{\le r}$ is the identity functor and $\Sigma_{>r}=0$, so the statement is clear. Now let us prove the statement for $r$, assuming it has been proved for $r+1$. Consider the triangle
\begin{displaymath}
\rR \Gamma_{\le r+1}(M) \to M \to \rR \Sigma_{>r+1}(M) \to
\end{displaymath}
Applying $\rR \Sigma_{>r}$, we obtain a triangle
\begin{displaymath}
\rR \Sigma_{>r}(\rR \Gamma_{\le r+1}(M)) \to \rR \Sigma_{>r}(M) \to \rR \Sigma_{>r}(\rR \Sigma_{>r+1}(M)) \to
\end{displaymath}
But $\Sigma_{>r} \Sigma_{>r+1}=\Sigma_{>r+1}$, so the rightmost term is $\rR \Sigma_{>r+1}(M)$, which belongs to $\rD^b_{\fgen}(A)$ by the inductive hypothesis. Since $\rR \Gamma_{\le r+1}(M)$ belongs to $\rD^b_{\fgen}(A)$ and is supported on $V(\fa_{r+1})$, it follows from Corollary~\ref{cor:Srfin} that $\rR \Sigma_{>r}(\rR \Gamma_{\le r+1}(M))$ belongs to $\rD^b_{\fgen}(A)$. It now follows from the above triangle that $\rR \Sigma_{>r}(M)$ belongs to $\rD^b_{\fgen}(A)$. From the canonical triangle relating $\rR \Sigma_{>r}$ and $\rR \Gamma_{\le r}$, we see that $\rR \Gamma_{\le r}(M)$ also belongs to $\rD^b_{\fgen}(A)$.
\end{proof}

The theorem exactly states that the hypothesis (Fin) from \S \ref{ss:Fin} holds, and so all the consequences of (Fin) given there hold as well.

\begin{remark}
We summarize the proof of Theorem~\ref{thm:satfin}. There are two parts. The first is that we can compute $\rR \Sigma_{\ge r}(M)$ if $M$ is an $A/\fa_r$-module since we can relate it to cohomology of sheaves on Grassmannians by Proposition~\ref{prop:Srformula}. (Note that in the formula in that proposition, $\rR \pi_*$ is sheaf cohomology on $\Gr_r(\cE)$, while $\rR S'$ is essentially sheaf cohomology on $\Gr_r(\bC^{\infty})$ by Theorem~\ref{thm:geoS}.) The second is that we can formally deduce the full result from this particular case via the inductive procedure in the above proof.
\end{remark}

\begin{remark} \label{rmk:loccohfin}
One can define local cohomology functor with respect to any ideal of $A$. However, the finiteness observed in the theorem for determinantal ideals does not hold in general. In fact, it seems plausible that finiteness essentially holds only for determinantal ideals (essentially because the property only depends on the radical).
\end{remark}

\subsection{Generators for $\rD^b_{\fgen}(A)$} \label{ss:Dgenerators}

Let $\cT$ be a triangulated category and let $S$ be a collection of objects in $\cT$. The triangulated subcategory of $\cT$ {\bf generated} by $S$ is the smallest triangulated subcategory of $\cT$ containing $S$. The following result gives a useful set of generators for $\rD^b_{\fgen}(A)_r$. We use notation as in \S \ref{ss:modar}: $Y=\Gr_r(\cE)$, $\cQ$ is the tautological bundle, $B=\bA(\cQ)$, and $\pi \colon Y \to X$ is the structure map.

\begin{proposition} \label{prop:Dr}
The category $\rD^b_{\fgen}(A)_r$ is the triangulated subcategory of $\rD^b_{\fgen}(A)$ generated by the objects $\rR \pi_*(V \otimes B)$ with $V \in \cV_Y^{\fgen}$.
\end{proposition}

\begin{proof} 
By Proposition~\ref{prop:derived-decomp}, the functor $\rR S_{\ge r} \colon \rD^b_{\fgen}(\Mod_{A,r}) \to \rD^b_{\fgen}(A)_r$ is an equivalence. Now, $\rD^b_{\fgen}(\Mod_{A,r})$ is generated by $\Mod_{A,r}^{\fgen}$ (thought of as complexes in degree~0). Every object of $\Mod_{A,r}^{\fgen}$ has a finite length filtration where the graded pieces belong to $\Mod_{A,r}^{\fgen}[\fa_r]$, and so it follows that $\Mod_{A,r}^{\fgen}[\fa_r]$ generates $\rD^b_{\fgen}(\Mod_{A,r})$. By Theorem~\ref{thm:ABequiv}, $\Mod_{A,r}^{\fgen}[\fa_r]$ is equivalent to $(\Mod_B^{\gen})^{\fgen}$, and under this equivalence, $\rR S_{\ge r}$ corresponds to $\rR \pi_* \circ \rR S'$ (Proposition~\ref{prop:Srformula}). We thus see that the image of $(\Mod_B^{\gen})^{\fgen}$ in $\rD^b_{\fgen}(A)$ under $\rR \pi_* \circ \rR S'$ generates $\rD^b_{\fgen}(A)_r$.

Now, by Corollary~\ref{cor:generic-inj}, every object of $(\Mod_B^{\gen})^{\fgen}$ admits a finite length forward resolution by objects of the form $T'(V \otimes B)$ with $V \in \cV_Y^{\fgen}$. It follows that the objects $\rR\pi_*(\rR \Sigma'(V \otimes B))$ generate $\rD^b_{\fgen}(A)_r$. By Corollary~\ref{cor:generic-inj}, $\rR \Sigma'(V \otimes B)=V \otimes B$, so the proposition follows.
\end{proof}

\begin{remark}
Let us spell out a little more precisely what Proposition~\ref{prop:Dr} means. Given $M \in \rD^b_{\fgen}(A)_r$, Proposition~\ref{prop:Dr} implies that there are objects $0=M_0, \ldots, M_n=M \in \rD^b_{\fgen}(A)_r$, objects $V_1, \ldots, V_n \in \cV^{\fgen}_Y$, integers $k_1, \ldots, k_n$, and exact triangles
\begin{displaymath}
M_i \to M_{i+1} \to \rR f_*(V_{i+1} \otimes B)[k_i] \to.
\end{displaymath}
This gives a way of inductively building arbitrary objects of $\rD^b_{\fgen}(A)_r$ from objects of the form $\rR f_*(V \otimes B)$. One often has tools to study these more simple objects, which is why Proposition~\ref{prop:Dr} is useful.
\end{remark}

\begin{remark} \label{rmk:proof-thm:struc}
If one takes $V = \bS_\lambda(\cQ)$ in Proposition~\ref{prop:Dr}, then there are no higher pushforwards, and $\pi_*(\bS_\lambda \cQ \otimes B)$ is the module $K_{r,\lambda}$ appearing in Theorem~\ref{thm:struc}. Since the objects $\pi^*(\cF) \otimes \bS_{\lambda}(\cQ)$, with $\lambda \subseteq r \times (d-r)$ and $\cF$ a finitely generated $\cO_X$-module, generate $\rD^b_{\fgen}(Y)$ (see Corollary~\ref{cor:grass-gen}), we find that the objects $\pi^*(\cF) \otimes \bS_{\mu}(\bV) \otimes K_{r,\lambda}$ generate $\rD^b_{\fgen}(A)_r$. This proves (a generalization of) Theorem~\ref{thm:struc}.
\end{remark}

\begin{corollary} \label{cor:DbA-gens}
The category $\rD^b_{\fgen}(A)$ is generated by the objects from Proposition~\ref{prop:Dr}, allowing $r$ to vary.
\end{corollary}

\begin{proof}
This is immediate since, by Theorem~\ref{thm:satfin}, $\rD^b_{\fgen}(A)$ is generated by the $\rD^b_{\fgen}(A)_r$.
\end{proof}

\subsection{An axiomatic approach to $\bA$-modules} \label{ss:axiomatic}

Using Proposition~\ref{prop:Dr}, we now formulate an axiomatic approach to proving results about $\bA$-modules. By a {\bf property of $\bA$-modules} we mean a rule that assigns to every triple $(X, \cE, M)$ consisting of a scheme $X$, a locally free coherent sheaf $\cE$ on $X$, and an object $M$ of $\rD^b_{\fgen}(\bA(\cE))$ a boolean value $\cP_{X,\cE}(M)$.

\begin{proposition} \label{prop:axiomatic}
Let $\cP$ be a property of $\bA$-modules. Suppose the following:
\begin{enumerate}[\indent \rm (a)]
\item If $\cP_{X,\cE}$ is true for two terms in an exact triangle then it is true for the third.
\item If $\cP_{X,\cE}(M)$ is true then so is $\cP_{X,\cE}(M[n])$ for all $n \in \bZ$.
\item If $\cE \to \cQ$ is a surjection then $\cP_{X,\cQ}(M) \implies \cP_{X,\cE}(M)$ for $M \in \rD^b_{\fgen}(\bA(\cQ))$.
\item Suppose $f \colon Y \to X$ is a proper map of schemes and $\cE$ is a locally free coherent sheaf on $X$. Then $\cP_{Y,f^*(\cE)}(M) \implies \cP_{X,\cE}(\rR f_*(M))$ for $M \in \rD^b_{\fgen}(\bA(f^*(\cE)))$.
\item $\cP_{X,\cE}$ is true for modules of the form $\bA(\cE) \otimes V$ with $V \in \cV_X^{\fgen}$.
\end{enumerate}
Then $\cP_{X,\cE}(M)$ is true for all $X$, $\cE$, and $M$.
\end{proposition}

\begin{proof}
Let $X$ and $\cE$ be given, and let us prove $\cP_{X,\cE}$ holds on all of $\rD^b_{\fgen}(\bA(\cE))$. We note that by~(a) and~(b), the full subcategory on objects for which $\cP_{X,\cE}$ holds is a triangulated subcategory of $\rD^b_{\fgen}(\bA(\cE))$. Let $Y=\Gr_r(\cE)$, let $\cQ$ be the tautological bundle on $Y$, and let $f \colon Y \to X$ be the structure map. By~(e), $\cP_{Y,\cQ}$ holds for all modules of the form $\bA(\cQ) \otimes V$ with $V \in \cV_Y^{\fgen}$. Thus by~(c), $\cP_{Y,f^*(\cE)}$ holds for all such modules as well. By~(d), we see that $\cP_{X,\cE}$ holds for all modules of the form $\rR f_*(\bA(\cQ) \otimes V)$ with $V \in \cV_Y^{\fgen}$. By Proposition~\ref{prop:Dr}, it follows that $\cP_{X,\cE}$ holds for all objects in $\rD^b_{\fgen}(\bA(\cE))_r$, for all $r$. Finally, $\rD^b_{\fgen}(\bA(\cE))$ is generated by the categories $\rD^b_{\fgen}(\bA(\cE))_r$ as $r$ varies (Corollary~\ref{cor:DbA-gens}), so $\cP_{X,\cE}$ holds on all of $\rD^b_{\fgen}(\bA(\cE))$, completing the proof.
\end{proof}

\begin{remark}
It is clear from the proof that the conditions in Proposition~\ref{prop:axiomatic} are stronger than what is actually needed: for instance, in~(d) it is enough to consider $Y$ that are relative Grassmannians. For our applications, the above proposition is enough though.
\end{remark}

\subsection{Grothendieck groups} \label{ss:groth}

The category $\Mod_A$ is naturally a $\cV$-module, and so $\rK(A)$ is a $\Lambda$-module. We now describe its structure as a $\Lambda$-module. Let $\pi_r \colon \Gr_r(\cE) \to X$ be the structure map, and let $\cQ_r$ be the tautological quotient bundle on $\Gr_r(\cE)$. Define
\begin{displaymath}
i_r \colon \rK(\Gr_r(\cE)) \to \rK(A), \qquad
i_r([V]) = [\rR \pi_r{}_*(V \otimes \bA(\cQ_r))].
\end{displaymath}
The main result is then:

\begin{theorem} \label{thm:groth}
The maps $i_r$ induce an isomorphism of $\Lambda$-modules
\begin{displaymath}
\bigoplus_{r=0}^d \Lambda \otimes \rK(\Gr_r(\cE)) \to \rK(A).
\end{displaymath}
\end{theorem}

\addtocounter{equation}{-1}
\begin{subequations}
\begin{proof} 
We have $\rK(A) = \bigoplus_{r=0}^d \rK(\rD^b_{\fgen}(A)_r)$ by Proposition~\ref{prop:Kdecomp}. We now have identifications
\begin{equation}
\label{eq:groth}
\rK(\rD^b_{\fgen}(A)_r)=\rK(\Mod_{A,r})=\rK(\Mod_{A,r}[\fa_r])=\rK(\Mod_{\bA(\cQ_r)}^{\gen})=\Lambda \otimes \rK(\Gr_r(\cE)).
\end{equation}
The first follows from Proposition~\ref{cor:Kisom}; the second from the fact that everything in $\Mod_{A,r}$ has a filtration with graded pieces in $\Mod_{A,r}[\fa_r]$; the third from the equivalence of $\Mod_{A,r}[\fa_r]$ with $\Mod_{\bA(\cQ_r)}^{\gen}$ (Theorem~\ref{thm:ABequiv}); and the fourth from Proposition~\ref{cor:Ktheory-generic}. We thus have an isomorphism $\rK(A)=\bigoplus_{r=0}^d \Lambda \otimes \rK(\Gr_r(\cE))$. It only remains to verify that this isomorphism is given by the claimed formula.

Both isomorphisms are $\Lambda$-linear, so it suffices to check that they agree on $[V] \in \rK(\Gr_r(\cE))$. We now trace $[V]$ backwards through the identifications in \eqref{eq:groth}, using notation as in \S \ref{ss:modar}. It gives $[M]$ in $\rK(\Mod_{\bA(\cQ_r)}^{\gen})$, with $M=T'(\bA(\cQ_r) \otimes V)$; which gives $[\Psi^{-1}(M)]$ in $\rK(\Mod_{A,r})$; which gives $[\rR S_{r-1}(\Psi^{-1}(M))]$ in $\rK(\rD^b_{\fgen}(A)_r)$. From Proposition~\ref{prop:Srformula} we have an isomorphism
\begin{displaymath}
\rR S_{\ge r} (\Psi^{-1}(M)) = \rR \pi_r{}_* \rR S'(M).
\end{displaymath}
By Corollary~\ref{cor:dersat} we have $\rR S'(M)=\bA(\cQ_r) \otimes V$. We thus see that $[V]$ gives $[\rR \pi_r{}_*(\bA(\cQ_r) \otimes V)]$ in $\rK(A)$, which is exactly $i_r([V])$.
\end{proof}
\end{subequations}

\begin{corollary}
$\rK(A)$ is isomorphic to a direct sum of $2^d$ copies of $\Lambda \otimes \rK(X)$. In particular, if $X=\Spec(\bC)$, then $\rK(A)$ is free of rank $2^d$ as a $\Lambda$-module.
\end{corollary}

\begin{proof}
By Corollary~\ref{cor:grass-Ktheory}, $\rK(\Gr_r(\cE)) \cong \rK(X)^{\oplus \binom{d}{r}}$, and $\sum_{r=0}^d \binom{d}{r} = 2^d$.
\end{proof}

\begin{remark}
Suppose $X=\Spec(\bC)$. Since $\bA(\bC)^{\otimes d}=\bA(\bC^d)$, there is a natural map
\begin{displaymath}
\rK(\bA(\bC))^{\otimes d} \to \rK(\bA(\bC^d)),
\end{displaymath}
given by taking the external tensor product of modules. One can take the tensor product on the left as $\Lambda$-modules, and so both sides are free $\Lambda$-modules of rank $2^d$. However, this map is not an isomorphism. We explain for $d=2$. Write $L_1$ and $L_2$ for two copies of $\bC$. The $\Lambda$-module $\rK(\bA(L_i))$ is free of rank~2, and the classes of $\bC$ and $\bA(L_i)$ form a basis. Thus the image of the above map is the $\Lambda$-module spanned by the external tensor product of these modules. These products are $\bC$, $\bA(L_1)$, $\bA(L_2)$, and $\bA(L_1 \oplus L_2)=\bA(\bC^2)$. However, the classes of $\bA(L_1)$ and $\bA(L_2)$ coincide: indeed, under the description of $\rK(\bA(\bC^2))$ in terms of Grassmannians the class of $\bA(L_1)$ corresponds to the class of the point $0 \in \bP^1$ (or rather, its structure sheaf), while the class of $\bA(L_2)$ corresponds to the class of the point $\infty$. Since all points in $\bP^1$ represent the same class in $\rK$-theory, we see that $[\bA(L_1)]=[\bA(L_2)]$ in the $\rK$-groups of $\bA(\bC^2)$. Thus the image of the external tensor product map on $\rK$-theory has rank at most~3 over $\Lambda$.
\end{remark}

\section{Koszul duality} \label{s:koszul}

\subsection{Three formulations of Koszul duality}

\subsubsection{Formulation 1: Exterior coalgebra comodules}

We let $X$ and $\cE$ and $A=\bA(\cE)$ be as in previous sections. Let $B=\lw(\cE\langle 1 \rangle)$. We note that $B$ is naturally a coalgebra; this structure will be relevant. Let $\bK=\bK(\cE)$ be the Koszul complex resolving $A/A_+=\cO_X$: this has $\bK^i=\lw^{-i}(\cE\langle 1 \rangle) \otimes A$ for $i \le 0$ and $\bK^i=0$ for $i>0$, and has the usual Koszul differential. Given a complex $M^{\bullet}$ of $A$-modules, the tensor product complex $\bK \otimes_A M$ is naturally a dg-comodule over $\Sym(\cE\langle 1 \rangle[-1])$. We now modify this construction to get a complex of $B$-comodules.

For a complex $M$ of objects in $\cV_X$, define the {\bf right shear} by
\begin{displaymath}
(M^R)^n = \bigoplus_{i \in \bZ} M^{n-i}_i.
\end{displaymath}
Here $M^{n-i}_i$ denotes the degree $i$ piece of $M^{n-i}$. The right shear shifts the degree $i$ piece of the complex $i$ units to the right. We also define the {\bf left shear} by
\begin{displaymath}
(M^L)^n = \bigoplus_{i \in \bZ} M^{n+i}_i.
\end{displaymath}
This is inverse to the right shear.

For a complex $M$ of $A$-modules we now define
\begin{displaymath}
\sK_{\cE}(M) = (\bK \otimes_A M)^R.
\end{displaymath}
Since the right shear of $\Sym(\cE\langle 1 \rangle[-1])$ is $B$, this is a complex of $B$-comodules. To be completely explicit, we have
\begin{displaymath}
\sK_{\cE}(M)^n = B \otimes \bigoplus_{i \ge 0} M^{n-i}_i = B \otimes (M^R)^n.
\end{displaymath}
The $B$-comodule structure on $\sK_{\cE}(M)^n$ is the obvious one (it is cofree). The differential (in the case where $X$ is affine) is given by 
\begin{align*}
d(x_1 \wedge \cdots \wedge x_n \otimes m) &=
x_1 \wedge \cdots \wedge x_n \otimes dm \\
&\qquad + \sum_{i=1}^n (-1)^{i+k+1} x_1 \wedge \cdots \wedge \hat{x}_i \wedge \cdots \wedge x_n \otimes x_i m,
\end{align*}
where $x_1, \ldots, x_n \in \cE\langle 1 \rangle$ and $m \in M^k$. We note that if $M$ is bounded below then so is $\sK_{\cE}(M)$. Furthermore, $\sK_{\cE}$ induces a functor $\rD(\Mod_A) \to \rD(\CoMod_B)$.

We now define a functor in the reverse direction. The degree~0 copy of $\cO_X$ in $B$ is a subcomodule. Let $\bL = \bL(\cE)$ be the Koszul complex resolving it: this has $\bL^i=\Sym^i(\cE\langle 1 \rangle) \otimes B$ for $i \ge 0$ and $\bL^i=0$ for $i<0$. Suppose that $N$ is a complex of $B$-comodules. We put
\begin{displaymath}
  \sL_{\cE}(N)=(\bL \otimes^B N)^L.
\end{displaymath}
Explicitly,
\begin{displaymath}
\sL_{\cE}(N)^n=A \otimes \bigoplus_{i \ge 0} N^{n+i}_i=A \otimes (N^L)^n.
\end{displaymath}
The $A$-module structure on $\sL_{\cE}(N)^n$ is the obvious one (it is free). Consider the comultiplication map $N \to B \otimes N$ and, in particular, the graded component $\nabla \colon N^k_n \to B_1 \otimes N^k_{n-1}$. When $X$ is affine, the differential on $\sL_{\cE}(N)$ is the sum of the following two maps:
\begin{align*}
  A_m \otimes N^k_i &\xrightarrow{1 \otimes \nabla} A_m \otimes B_1 \otimes N^k_{i-1} \xrightarrow{\mu \otimes 1} A_{m+1} \otimes N^k_{i-1}\\
  A_m \otimes N^k_i &\xrightarrow{(-1)^k \otimes d_N} A_m \otimes N^{k+1}_i
\end{align*}
where $\mu \colon A_m \otimes B_1 \to A_{m+1}$ is the multiplication map and $d_N$ is the differential on $N$. Note that this is dual to the differential that we have defined on $\sK_\cE(M)$ above. To see that this is a complex, consider the following diagram:
\[
  \xymatrix{ A_m \otimes N_i^k \ar[r]^-{1 \otimes \nabla} \ar[d]_-{1 \otimes d_N} & A_m \otimes B_1 \otimes N^k_{i-1} \ar[r]^-{\mu \otimes 1} \ar[d]_-{1 \otimes 1 \otimes d_N} & A_{m+1} \otimes N^k_{i-1} \ar[d]^-{1 \otimes d_N} \\
    A_m \otimes N^{k+1}_i \ar[r]^-{1 \otimes \nabla} & A_m \otimes B_1 \otimes N^{k+1}_{i-1} \ar[r]^-{\mu \otimes 1} & A_{m+1} \otimes N^{k+1}_{i-1} }
\]
The left hand square commutes since $N$ is a complex of $A$-modules, while the right hand square commutes since the horizontal and vertical maps act on different tensor factors. Hence the main rectangle commutes and the signed sum of the boundaries of this rectangle compute the components of the square of the differential of $\sL_\cE(N)$.
If $N$ is bounded above then so is $\sL_{\cE}(N)$. Furthermore, $\sL_{\cE}$ induces a functor $\rD(\CoMod_B) \to \rD(\Mod_A)$.

We can identify $\sL_{\cE}(\sK_{\cE}(M))$ with the complex $A \otimes B \otimes M$ with the cohomological grading
\[
(A \otimes B \otimes M)^n = A \otimes \bigoplus_{i \ge 0} B_i \otimes M^{n-i}
\]
and differential ($f \in A$, $x_1,\dots,x_i \in \cE\langle 1 \rangle$, $m \in M^{n-i}$) 
\begin{align*}
  d(f \otimes x_1 \wedge \cdots \wedge x_i \otimes m) &= \sum_{k=1}^i (-1)^{k-1} x_k f \otimes x_1 \wedge \cdots \hat{x}_k \cdots \wedge x_i \otimes m\\
                                                      &\qquad +(-1)^n f \otimes x_1 \wedge \cdots \wedge x_i \otimes d(m) \\
  & \qquad + \sum_{k=1}^i (-1)^{i+k+1} f \otimes x_1 \wedge \cdots \hat{x}_k \cdots \wedge x_i \otimes x_k m.
\end{align*}
We define maps $(A \otimes B \otimes M)^n \to M^n$ as follows. We have $(A \otimes B \otimes M)^n = A \otimes \bigoplus_{i \ge 0} B_i \otimes M^{n-i}$, and for $i=0$, we take the multiplication map $A \otimes M^n \to M^n$, and define it to be $0$ on all other components. This defines a morphism of chain complexes. Consider the corresponding cone $\sL_\cE(\sK_\cE(M)) \to M$. We can filter this complex by the cohomological grading on $M$. The associated graded complex is a direct sum of complexes of the form
\[
\cdots \to A \otimes B_2 \otimes M^i \to  A \otimes B_1 \otimes M^i \to A \otimes B_0 \otimes M^i \to M^i \to 0
\]
which are everywhere exact. Hence the cone is acyclic and the map $\sL_\cE(\sK_\cE(M)) \to M$ is a quasi-isomorphism.

% Since $\sL_{\cE}(\sK_{\cE}(\cO_X))$ is the Koszul complex, there is a canonical quasi-isomorphism $\sL_{\cE}(\sK_{\cE}(\cO_X)) \to \cO_X$, and hence we have a canonical quasi-isomorphism $\sL_{\cE}(\sK_{\cE}(M)) \to M$ for any complex of $A$-modules \Acom{this should be correct, but I don't know about the reasoning. I believe there should be a canonical map $A \otimes B \otimes M \to M$; this should be defined roughly by taking the canonical map $A \otimes M \to M$ on the degree~0 part of $B$, and~0 on the higher degree parts of $B$, I think. we may need to say more, or give a reference, for why it's a q.i.}.
Similarly, there is a canonical map $N \to \sK_{\cE}(\sL_{\cE}(N))$ of complexes of $B$-comodules that is always a quasi-isomorphism. Thus $\sK_{\cE}$ and $\sL_{\cE}$ induce mutually quasi-inverse equivalences of $\rD(\Mod_A)$ and $\rD(\CoMod_B)$.

For the benefit of later use, we record the following simple result here.

\begin{proposition} \label{prop:kozhom}
Let $M$ be a complex of $A$-modules. Then 
\begin{displaymath}
\rH^n(\sK_{\cE}(M)) = \bigoplus_{i \in \bZ} \Tor^A_{i-n}(M, \cO_X)_i.
\end{displaymath}
\end{proposition}

\subsubsection{Formulation 2: Exterior algebra modules}

We now want to modify the constructions of the previous section to replace the comodules that appear with modules. We do this by applying a duality to $\cV_X$ that interchanges $B$-modules and $B$-comodules.

In this section, we assume that $X$ has a dualizing complex $\omega$ (see \cite[Chapter V]{H-RD} for definitions and basic properties). Then $\omega$ induces a duality $\bD$ of $\rD_{\fgen}(X)$ via $\bD(M)=\rR \Hom_X(M, \omega)$. We say that $M \in \cV_X$ is {\bf degree-wise finitely generated} (dfg) if $M_{\lambda}$ is a coherent sheaf on $X$ for all $\lambda$. Similarly, we say that a complex $M$ in $\cV_X$ is dfg if each $\rH^i(M)$ is. We let $\rD_{\dfg}(\cV_X)$ be the full subcategory of $\rD(\cV_X)$ on the dfg objects. We extend $\bD$ to $\rD_{\dfg}(\cV_X)$ by simply applying $\bD$ to the multiplicity spaces. That is, for $M \in \rD_{\dfg}(\cV_X)$ we write $M=\bigoplus_{\lambda} M_{\lambda} \otimes \bS_{\lambda}(\bV)$, where $M_{\lambda} \in \rD_{\fgen}(X)$, and then define $\bD(M)=\bigoplus_{\lambda} \bD(M_{\lambda}) \otimes \bS_{\lambda}(\bV)$. 

\begin{lemma} \label{lem:duality-schur}
$\bD(M \otimes \bS_{\lambda}(\bV))$ is canonically isomorphic to $\bD(M) \otimes \bS_{\lambda}(\bV)$.
\end{lemma}

\begin{proof}
The $\bS_\nu(\bV)$ multiplicity space of $M \otimes \bS_\lambda(\bV)$ is $\bigoplus_\mu M_\mu \otimes \hom_{\cV}(\bS_\nu, \bS_\lambda \otimes \bS_\mu)$, so it suffices to construct a canonical isomorphism $\hom_\cV(\bS_\nu, \bS_\lambda \otimes \bS_\mu)^* \cong \hom_\cV(\bS_\nu, \bS_\lambda \otimes \bS_\mu)$. The former can be identified with $\hom_\cV(\bS_\lambda \otimes \bS_\mu, \bS_\nu)$. Note that there is a duality on $\cV$ given (in the polynomial functor perspective) by $F^\vee(V) := F(V^*)^*$ for finite-dimensional $V$. When $F$ is the identity, we canonically have $F^\vee = F$, and hence, we get a canonical identification for any Schur functor and their tensor products.
\end{proof}

Suppose now that $M$ is a complex of $B$-comodules. We thus have a comultiplication map $M \to M \otimes \cE\langle 1 \rangle$. Applying $\bD$ to this map yields a map $\bD(M \otimes \cE\langle 1 \rangle) \to \bD(M)$. Recall that $\cE\langle 1 \rangle$ is just $\cE \otimes \bV$. Then $\bV$ pulls out of $\bD$ by Lemma~\ref{lem:duality-schur}. Since $\cE$ is a locally free coherent sheaf, we have $\bD(\cE \otimes -)=\cE^* \otimes \bD(-)$. We thus have a map $\cE^*\langle 1 \rangle \otimes \bD(M) \to \bD(M)$. In fact, one can show that $\bD(M)$ naturally has the structure of a $B^*$-module, where $B^*=\lw(\cE^*\langle 1 \rangle)$. This construction gives an equivalence between $\rD_{\dfg}(\CoMod_B)$ and $\rD_{\dfg}(\Mod_{B^*})$. It interchanges the bounded below and bounded above subcategories, and preserves the bounded subcategory.

We now define
\begin{displaymath}
\sK^*_{\cE} \colon \rD_{\dfg}(\Mod_A)^{\op} \to \rD_{\dfg}(\Mod_{B^*}), \qquad \sK^*_{\cE} = \bD \circ \sK_{\cE}
\end{displaymath}
and
\begin{displaymath}
\sL^*_{\cE} \colon \rD_{\dfg}(\Mod_{B^*})^{\op} \to \rD_{\dfg}(\Mod_A), \qquad \sL^*_{\cE} = \sL_{\cE} \circ \bD.
\end{displaymath}
(We note that the functors $\sK_{\cE}$ and $\sL_{\cE}$ preserve the dfg condition.) It is clear that $\sK^*_{\cE}$ and $\sL^*_{\cE}$ are mutually quasi-inverse equivalences. We note that both $\sK^*_{\cE}$ and $\sL^*_{\cE}$ take the bounded below subcategory to the bounded above subcategory. We also note that $\sK^*_{\cE}$ and $\sL^*_{\cE}$ depend on the choice of dualizing complex $\omega$, though this dependence is absent from the notation.

\subsubsection{Formulation 3: Symmetric algebra modules}

We now want to modify the constructions of the previous section to replace modules over the exterior algebra with modules over the symmetric algebra. We do this by applying the transpose functor to $\cV_X$. Recall that this is a covariant functor $(-)^{\dag}$ that is $\Mod_X$-linear and satisfies $\bS_{\lambda}(\bV)^{\dag}=\bS_{\lambda^{\dag}}(\bV)$, where $\lambda^{\dag}$ is the transpose of the partition $\lambda$. Furthermore, while it is a tensor functor, it is not a {\it symmetric} tensor functor: it interchanges the usual symmetry and the graded symmetry of the tensor product on $\cV_X$ (see \cite[\S 7.4]{expos}). Let $A^*=\Sym(\cE^*\langle 1 \rangle)=(B^*)^{\dag}$. Then we have an equivalence of categories $\Mod_{B^*} \to \Mod_{A^*}$ via $M \mapsto M^{\dag}$. We now define
\begin{displaymath}
\sK^{*,\dag}_\cE \colon \rD_{\dfg}(\Mod_A)^{\op} \to \rD_{\dfg}(\Mod_{A^*}), \qquad \sK^{*,\dag}_\cE = (-)^{\dag} \circ \sK_{\cE}^*
\end{displaymath}
and
\begin{displaymath}
\sL^{*,\dag}_\cE \colon \rD_{\dfg}(\Mod_{A^*})^{\op} \to \rD_{\dfg}(\Mod_A), \qquad \sL^{*,\dag}_\cE = \sL_{\cE}^* \circ (-)^{\dag}.
\end{displaymath}
Once again, it is clear that $\sK^{*,\dag}_A$ and $\sL^{*,\dag}_A$ are mutually quasi-inverse equivalences. 

\begin{proposition} \label{prop:kosz-sym}
We have $\sK^{*,\dag}_\cE=\sL^{*,\dag}_{\cE^*}$ and $\sL^{*,\dag}_\cE = \sK^{*,\dag}_{\cE^*}$.
\end{proposition}

\begin{proof}
We have 
\begin{align*}
\sK^{*, \dag}_\cE(M) &= \bD(( \bK(\cE) \otimes_A M)^R)^\dagger = (\bD( \bK(\cE)^\dagger \otimes_B M^\dagger))^L = (\bK(\cE)^{*, \dagger} \otimes^{B^*} \bD(M^\dagger))^L.
\end{align*}
The second equality uses that $(-)^\dagger$ commutes with $\bD$ and is a tensor functor. The third equality uses that $\bK(\cE)$ is a complex of locally free sheaves. Finally, $\bK(\cE)^{*,\dagger} = \bL(\cE^*)$, so we see that $\sK^{*, \dag}_{\cE} = \sL^{*,\dag}_{\cE^*}$. The other identity is similar.
\end{proof}

\subsection{The Fourier transform} \label{ss:fourier}

We now define the {\bf Fourier transform}
\begin{displaymath}
\sF_{\cE} \colon \rD_{\dfg}(\bA(\cE))^{\op} \to \rD_{\dfg}(\bA(\cE^*))
\end{displaymath}
to be the functor $\sK^{*,\dag}_{\cE}$. It is an equivalence of categories. We gather some of its basic properties here.

\begin{proposition} \label{prop:FTprop}
We have the following:
\begin{enumerate}[\indent \rm (a)]
\item $\sF_{\cE}$ and $\sF_{\cE^*}$ are canonically quasi-inverse to each other.
\item $\sF_{\cE}$ carries $\rD^+_{\dfg}(\bA(\cE))$ into $\rD^-_{\dfg}(\bA(\cE^*))$.
\item $\sF_{\cE}(\bS_{\lambda}(\bV) \otimes -)=\bS_{\lambda^{\dag}}(\bV)[-n] \otimes \sF_{\cE}(-)$, where $n=\vert \lambda \vert$.
\item If $\cM$ is a locally free coherent sheaf on $X$ then $\sF_{\cE}(\cM \otimes -)=\cM^* \otimes \sF_{\cE}(-)$.
\item If $\cE \to \cQ$ is a surjection of vector bundles and $M \in \rD_{\dfg}(\bA(\cQ))$ then $\sF_{\cE}(M)$ is canonically isomorphic to $\sF_{\cQ}(M) \otimes_{\bA(\cQ^*)} \bA(\cE^*)$.
\item If $\cM$ is a coherent sheaf on $X$ then $\sF_{\cE}(\bA(\cE) \otimes \cM)=\bD(\cM)$, regarded as a trivial $\bA(\cE^*)$-module.
\item If $\sF_{\cE}'$ is defined with respect to a different dualizing complex then there is an integer $d$ and a line bundle $\cL$ on $X$ such that $\sF_\cE'(M) \cong \sF_\cE(M)[d] \otimes \cL$.
\end{enumerate}
\end{proposition}

\begin{proof}
(a) This follows from Proposition~\ref{prop:kosz-sym}.

(b) We have already noted this for $\sK^{*,\dag}_{\cE}$.

(c) We have $\sK(\bS_{\lambda}(\bV) \otimes -)=\bS_{\lambda}(\bV)[n] \otimes \sK(-)$. Thus $\sK^*(\bS_{\lambda}(\bV) \otimes -)=\bS_{\lambda}(\bV)[-n] \otimes \sK^*(-)$. Finally, taking transposes yields the stated formula.

(d) This is clear.

(e) We have 
\begin{align*}
\sF_\cE(M) &= (\bD(\bK(\cE) \otimes_{\bA(\cE)} M)^R)^\dagger\\
&= (\bD(\bK(\cE) \otimes_{\bK(\cQ)} (\bK(\cQ) \otimes_{\bA(\cQ)} M))^R)^\dagger\\
&= \bA(\cE^*) \otimes_{\bA(\cQ^*)} (\bD(\bK(\cQ) \otimes_{\bA(\cQ)} M)^R)^\dagger.
\end{align*}

(f) $\bK(\cE) \otimes_{\bA(\cE)} (\bA(\cE) \otimes \cM)$ is quasi-isomorphic to $\cM$ concentrated in degree $0$, so $\sK(\cM)^\dagger \simeq \cM$, which gives $\sF_\cE(\bA(\cE) \otimes \cM) = \bD(\cM)$.

(g) Follows from \cite[\S V.3]{H-RD}.
\end{proof}

We now examine how the Fourier transform interacts with pushforwards. We first set some notation. Let $f \colon Y \to X$ be a proper map of schemes, let $\cE_X$ be a vector bundle $X$, and let $\cE_Y=f^*(\cE_X)$ be its pullback to $Y$. Put $A_Y=\bA(\cE_Y)$ and $A_X=\bA(\cE_X)$. We let $B_X$, $B^*_X$, and $A^*_X$ be defined as in previous sections. Let $\omega_X$ be a dualizing sheaf on $X$ and let $\omega_Y=f^!(\omega_X)$ be the corresponding one on $Y$. Write $\bD_Y$ and $\bD_X$ for the duality functors they give.

\begin{proposition} \label{koszul-pushfwd}
We have canonical functorial isomorphisms of functors on $\rD^+_{\dfg}(\Mod_{A_Y})$:
\begin{enumerate}[\indent \rm (a)]
\item $\rR f_* \circ \sK_{\cE_Y} = \sK_{\cE_X} \circ \rR f_*$.
\item $\rR f_* \circ \sK^*_{\cE_Y} = \sK^*_{\cE_X} \circ \rR f_*$.
\item $\rR f_* \circ \sF_{\cE_Y} = \sF_{\cE_X} \circ \rR f_*$.
\end{enumerate}
\end{proposition}

\begin{proof}
(a) Let $M \in \rD^+_{\dfg}(\Mod_{A_Y})$ and pick a quasi-isomorphism $M \to I$ with $I$ a bounded-below complex of injective $A_Y$-modules. Note that each multiplicity space of an injective $A_Y$-module is injective as an $\cO_Y$-module (Proposition~\ref{prop:injOX}). Thus $\rR f_*(M) \cong f_*(I)$. Recall that
\begin{displaymath}
\sK_{\cE_Y}(I)^n = B_Y \otimes \bigoplus_{i \ge 0} I_i^{n-i}.
\end{displaymath}
This is a bounded-below complex. As $B_Y=f^*(B_X)$, the projection formula implies that the sheaf $B_Y \otimes I_i^{n-i}$ is $f_*$-acyclic. We can thus compute $\rR f_*$ of the above complex by simply applying $f_*$. Doing this, and using the projection formula again, gives
\begin{displaymath}
f_*(\sK_{\cE_Y}(I))^n = B_X \otimes \bigoplus_{i \ge 0} f_*(I)_i^{n-i}.
\end{displaymath}
However, this exactly coincides with $\sK_{\cE_X}(f_*(I))$.

(b) Precompose the identity in part~(a) with $\bD_X$ and use the duality theorem $\bD_X \circ \rR f_*=\rR f_* \circ \bD_Y$.

(c) Simply apply transpose to the identity in part~(b).
\end{proof}

We now carry out a fundamental computation. Let $Y=\Gr_r(\cE)$, let $\cQ$ and $\cR$ be the usual bundles on $Y$, and let $\pi \colon Y \to X$ be the structure map. Let $Y'=\Gr_{d-r}(\cE^*)$, and let $\cQ'$, $\cR'$, and $\pi'$ be defined analogously. We note that $Y$ and $Y'$ are canonically isomorphic.

\begin{proposition} \label{prop:FTgrass}
Let $\cM$ be a finitely generated $\cO_Y$-module, and let $\cM'$ be the corresponding $\cO_{Y'}$-module. Let $\lambda$ be a partition of size $n$. Then there is a canonical isomorphism
\begin{displaymath}
\sF_{\cE}(\rR \pi_*(\bS_{\lambda}(\bV) \otimes \cM \otimes \bA(\cQ))) =
\rR \pi'_*(\bS_{\lambda^{\dag}}(\bV)[-n] \otimes \bD(\cM') \otimes \bA(\cQ')).
\end{displaymath}
\end{proposition}

\begin{proof}
We compute the left side. We first note that we can switch $\sF_{\cE}$ and $\rR \pi_*$ by Proposition~\ref{koszul-pushfwd}. Next, $\bS_{\lambda}(\bV)$ pulls out of $\sF_{\cE}$ and becomes $\bS_{\lambda^{\dag}}(\bV)[-n]$ by Proposition~\ref{prop:FTprop}(c). We have $\sF_{\cQ}(\cM \otimes \bA(\cQ))=\bD(\cM)$, a trivial $\bA(\cQ^*)$-module, by Proposition~\ref{prop:FTprop}(f), and so $\sF_{\cE}(\cM \otimes \bA(\cQ))=\bD(\cM) \otimes \bA(\cR^*)$ by Proposition~\ref{prop:FTprop}(e). We have thus shown
\begin{displaymath}
\sF_{\cE}(\rR \pi_*(\bS_{\lambda}(\bV) \otimes \cM \otimes \bA(\cQ)))
=\rR \pi_*(\bS_{\lambda^{\dag}}(\bV)[-n] \otimes \bD(\cM) \otimes \bA(\cR^*)).
\end{displaymath}
We now move everything to $Y'$ via the isomorphism between $Y$ and $Y'$. This changes $\pi$ to $\pi'$ and $\cM$ to $\cM'$ and $\cR^*$ to $\cQ'$. This yields the stated result.
\end{proof}

\subsection{The finiteness theorem} \label{ss:fourierfin}

The following is the fundamental finiteness result about the Fourier transform:

\begin{theorem} \label{thm:fourierfinite}
The Fourier transform $\sF_{\cE}$ carries $\rD^b_{\fgen}(\bA(\cE))$ into $\rD^b_{\fgen}(\bA(\cE^*))$.
\end{theorem}

\begin{proof}
Let $\cP_{X,\cE}(M)$ be the truth-value of the statement ``$\sF_{\cE}(M)$ is bounded with finitely generated cohomology.'' (Note that while $\sF_{\cE}$ depends on the choice of a dualizing sheaf on $X$, the value of $\cP_{X,\cE}$ does not by Proposition~\ref{prop:FTprop}(g).) Then $\cP$ is a property of $\bA$-modules. We show that $\cP$ holds for all modules by verifying the five conditions in Proposition~\ref{prop:axiomatic}. The first two conditions are clear. We now consider the other three.

(c) Let $\cE \to \cQ$ be a surjection of locally free coherent sheaves on $X$ and let $M \in \rD^b_{\fgen}(\bA(\cQ))$. Then $\sF_{\cE}(M)$ is isomorphic to $\sF_{\cQ}(M) \otimes_{\bA(\cQ^*)} \bA(\cE^*)$ by Proposition~\ref{prop:FTprop}(e). Thus if $\cP_{X,\cQ}(M)$ holds then so does $\cP_{X,\cE}(M)$.

(d) Suppose $f \colon Y \to X$ is a proper morphism of schemes, $\cE$ is a locally free coherent sheaf on $X$, and $M \in \rD^b_{\fgen}(\bA(f^*(\cE)))$. Proposition~\ref{koszul-pushfwd}(c) gives an isomorphism $\rR f_*(\sF_{f^*(\cE)}(M))=\sF_{\cE}(\rR f_*(M))$. (We assume here that $\omega_Y$ is chosen to be $f^!(\omega_X)$.) So $\cP_{Y,f^*(\cE)}(M) \Rightarrow \cP_{X,\cE}(\rR f_*(M))$ by Corollary~\ref{cor:finpushfwd}.

(e) This follows from Proposition~\ref{prop:FTprop}(c,f).
\end{proof}

\begin{corollary}
A finitely generated $A$-module has finite regularity.
\end{corollary}

\subsection{The duality theorem} \label{ss:fourierdual}

The following is a sort of duality theorem involving the Fourier transform and the rank stratification.

\begin{theorem} \label{thm:ftduality}
Set $d=\rank(\cE)$. We have natural identifications of functors $\rD^b_{\fgen}(\bA(\cE)) \to \rD^b_{\fgen}(\bA(\cE^*))${\rm :}
\begin{enumerate}[\indent \rm (a)]
\item $\sF_{\cE} \circ \rR \Gamma_{\le r} = \rR \Sigma_{\ge d-r} \circ \sF_{\cE}$.
\item $\sF_{\cE} \circ \rR \Sigma_{\ge r} = \rR \Gamma_{\le d-r} \circ \sF_{\cE}$.
\item $\sF_{\cE} \circ \rR \Pi_r = \rR \Pi_{d-r} \circ \sF_{\cE}$.
\end{enumerate}
\end{theorem}

\begin{proof}
It follows immediately from Propositions~\ref{prop:Dr} and~\ref{prop:FTgrass} that $\sF_{\cE}$ carries $\rD^b_{\fgen}(A)_r$ into $\rD^b_{\fgen}(A^*)_{d-r}$. Now, let $M \in \rD^b_{\fgen}(A)$. We then have an exact triangle
\begin{displaymath}
\rR \Gamma_{\le r}(M) \to M \to \rR \Sigma_{>r}(M) \to.
\end{displaymath}
Applying $\sF_\cE$ yields an exact triangle
\begin{displaymath}
\sF_\cE(\rR \Sigma_{>r}(M)) \to \sF_\cE(M) \to \sF_\cE(\rR \Gamma_{\le r}(M)) \to.
\end{displaymath}
Since $\rR \Gamma_{\le r}(M)$ belongs to $\rD^b_{\fgen}(A)_{\le r}$, it follows that $\sF_\cE(\rR \Gamma_{\le r}(M))$ belongs to $\rD^b_{\fgen}(A^*)_{\ge d-r}$. Similarly, $\sF_\cE(\rR \Sigma_{>r}(M))$ belongs to $\rD^b_{\fgen}(A^*)_{<d-r}$. We also have an exact triangle
\begin{displaymath}
\Gamma_{<d-r}(\sF_{\cE}(M)) \to \sF_{\cE}(M) \to \Sigma_{\ge d-r}(\sF_{\cE}(M))
\end{displaymath}
Since $\rD^b_{\fgen}(A^*)$ admits a semi-orthogonal decomposition $\langle \rD^b_{\fgen}(A^*)_{<d-r}, \rD^b_{\fgen}(A^*)_{\ge d-r} \rangle$, it follows that there are canonical isomorphisms $\sF_\cE(\rR \Sigma_{>r}(M))=\rR \Gamma_{<d-r}(\sF_\cE(M))$ and $\sF_\cE(\rR \Gamma_{\le r}(M))=\rR \Sigma_{\ge d-r}(\sF_\cE(M))$. This proves (a) and (b). As for (c), we have
\begin{displaymath}
\begin{split}
\sF_\cE \circ \rR \Pi_r
&= \sF_\cE \circ \rR \Gamma_{\le r} \circ \rR \Sigma_{\ge r} \\
&= \rR \Sigma_{\ge d-r} \circ \sF_\cE \circ \rR \Sigma_{\ge r} \\
&= \rR \Sigma_{\ge d-r} \circ \rR \Gamma_{\le d-r} \circ \sF_\cE \\
&= \rR \Pi_{d-r} \circ \sF_\cE.
\end{split}
\end{displaymath}
In the first and fourth lines we used the definition of $\rR \Pi_r$, in the second line we used part~(a), and in the third line we used part~(b).
\end{proof}

\subsection{The induced map on Grothendieck groups} \label{ss:fouriergroth}

Let $(-)^{\ast} \colon \Lambda \to \Lambda$ be the map taking $s_{\lambda}$ to $(-1)^{\vert \lambda \vert} s_{\lambda^{\dag}}$. This is a ring homomorphism. Since $\sF_{\cE}$ is an equivalence $\rD^b_{\fgen}(\bA(\cE)) \to \rD^b_{\fgen}(\bA(\cE^*))$, it induces an isomorphism $\varphi \colon \rK(\bA(\cE)) \to \rK(\bA(\cE^*))$. This map is $\ast$-linear, meaning $\varphi(ax)=a^* \varphi(x)$ for $a \in \Lambda$ and $x \in \rK(\bA(\cE))$, by Proposition~\ref{prop:FTprop}(c). The following result gives a complete description of $\varphi$.

\begin{proposition}
We have a commutative diagram
\begin{displaymath}
\xymatrix{
\bigoplus_{r=0}^d \Lambda \otimes \rK(\Gr_r(\cE)) \ar[r] \ar[d] & \rK(\bA(\cE)) \ar[d]^{\varphi} \\
\bigoplus_{r=0}^d \Lambda \otimes \rK(\Gr_r(\cE^*)) \ar[r] & \rK(\bA(\cE^*)) }
\end{displaymath}
where the horizontal maps are the ones from Theorem~\ref{thm:groth}, and the left vertical map is $(-)^{\ast}$ on the $\Lambda$ factors, and takes $[\cM] \in \rK(\Gr_r(\cE))$ to $[\bD(\cM')] \in \rK(\Gr_{d-r}(\cE^*))$, where $\cM'$ corresponds to $\cM$ under the isomorphism $\Gr_r(\cE)=\Gr_{d-r}(\cE^*)$.
\end{proposition}

\begin{proof}
This follows immediately from the description of the maps in Theorem~\ref{thm:groth} and the calculation in Proposition~\ref{prop:FTgrass}.
\end{proof}

\appendix

\section{Basic facts about Grassmannians} \label{s:grass}

Let $X$ be a noetherian separated scheme of finite Krull dimension over a field of characteristic $0$ and let $\cE$ be a vector bundle of rank $d$. let $Y=\Gr_r(\cE)$, let $\pi \colon Y \to X$ be the structure map, and let $\cQ$ and $\cR$ be the tautological bundles.

\subsection{Borel--Weil--Bott} \label{sec:BWB}

Let $S_d$ denote the symmetric group on $d$ letters, more precisely the group of bijections of $[d] = \{1,\dots,d\}$. Given $\sigma \in S_d$, define its length to be
\[
\ell(\sigma) = \#\{(i,j) \mid 1 \le i < j \le d, \quad \sigma(i) > \sigma(j)\}.
\]
Also define 
\[
\rho = (d-1, d-2, \dots, 1 ,0) \in \bZ^d.
\]
Given $v \in \bZ^d$, define $\sigma(v) = (v_{\sigma^{-1}(1)}, \dots, v_{\sigma^{-1}(d)})$ and $\sigma \bullet v = \sigma(v + \rho) - \rho$. Note that given any $v \in \bZ^d$, either there exists $\sigma \ne 1$ such that $\sigma \bullet v = v$, or there exists a unique $\sigma$ such that $\sigma \bullet v$ is weakly decreasing.

Let $\alpha = (\alpha_1, \dots, \alpha_r) \in \bZ^{r}$ and $\beta = (\beta_1, \dots, \beta_{d-r}) \in \bZ^{d-r}$ be weakly decreasing and set $v = (\alpha_1, \dots, \alpha_{r}, \beta_1, \dots, \beta_{d-r})$. For the following, see \cite[Corollary 4.1.9]{weyman}. 

\begin{theorem}[Borel--Weil--Bott] \label{thm:BWB}
Exactly one of the following two cases happens:
\begin{enumerate}[\rm (a)]
\item If there exists $\sigma \ne 1$ such that $\sigma \bullet v = v$, then $\rR^j \pi_*(\bS_\alpha(\cQ) \otimes \bS_\beta(\cR)) = 0$ for all $j$.
\item Otherwise, there exists unique $\sigma$ such that $\gamma = \sigma \bullet v$ is weakly decreasing, and 
\[
\rR^j \pi_*(\bS_\alpha(\cQ) \otimes \bS_\beta(\cR)) \cong 
\begin{cases} 
\bS_\gamma(\cE) & \text{if $j = \ell(\sigma)$}\\
0 & \text{if $j \ne \ell(\sigma)$} \end{cases}.
\]
\end{enumerate}
\end{theorem}

Note that $\bS_\lambda (\cQ^*) \cong \bS_\mu \cQ$ where $\mu = (-\lambda_r, \dots, -\lambda_1)$, and similarly for any vector bundle.

\begin{remark} \label{rmk:bott-algorithm}
The length $\ell(\sigma)$ of a permutation is also equal to the minimal number of adjacent transpositions $s_i = (i,i+1)$ needed to write $\sigma$, i.e., the minimal $\ell$ such that we can write $\sigma = s_{i_1} s_{i_2} \cdots s_{i_\ell}$. The operation $s_i \bullet v$ has the effect of replacing $v_i,v_{i+1}$ with $v_{i+1} - 1, v_i + 1$. 

So in (b) above, the process of getting $\gamma$ from $v$ can be thought of in terms of a bubble sorting procedure: if $v_i < v_{i+1}$, apply $s_i \bullet$ to get a new sequence with $v'_i > v'_{i+1}$; the number of times needed to do this is $\ell(\sigma)$. We will refer to this procedure as ``Bott's algorithm'', and keeping the notation of (b), we write $v \xrightarrow{n} \gamma$ where $n = \ell(\sigma)$.
\end{remark}

\begin{corollary} \label{cor:dual-basis}
Suppose $\alpha, \beta \subseteq r \times (d-r)$ and $\cF$ is a coherent sheaf on $X$. Then
\[
\rR^j \pi_*( \bS_\alpha(\cQ^*) \otimes \bS_{\beta^\dagger}(\cR) \otimes \pi^* \cF) \cong
\begin{cases}
\cF & \text{if $\alpha=\beta$ and $j = |\alpha|$}\\
0 & \text{otherwise}
\end{cases}.
\]
\end{corollary}

\begin{proof}
Using the projection formula, we may assume that $\cF = \cO_X$. In that case, this is \cite[Lemma 3.2]{kapranov} when $X$ is a point, but the combinatorics is exactly the same in the general setting. Here is a sketch of how this can be proven. Pick $\sigma \in S_d$ and consider $c = \sigma \bullet 0$. Then $c_1 \ge \cdots \ge c_r$ and $c_{r+1} \ge \cdots \ge c_d$ if and only if $\sigma^{-1}(1) \le \cdots \le \sigma^{-1}(r)$ and $\sigma^{-1}(r+1) \le \cdots \le \sigma^{-1}(d)$; furthermore, $c = (-\lambda_r, \dots, -\lambda_1, \lambda_1^\dagger, \dots, \lambda_{d-r}^\dagger)$ where $\lambda \subseteq r \times (d-r)$, so we write $\sigma = w_\lambda$; also $\ell(w_\lambda) = |\lambda|$. Then what remains to show is: if $\lambda \ne \mu$, then $((w_\lambda \bullet 0)_{1, \dots, r}, (w_\mu \bullet 0)_{r+1, \dots, d})$ has a repeated element, and this follows since we have $w_\lambda^{-1}(i) = w_\mu^{-1}(j)$ for some $1 \le i \le r$ and $r+1 \le j \le d$.
\end{proof}

\subsection{Derived category and K-theory}

This is adapted from \cite{kapranov}.

Let $\pi_i \colon Y \times_X Y \to Y$ denote the projection maps for $i=1,2$. Given sheaves $\cF, \cG$ on $Y$, define $\cF \boxtimes \cG = \pi_1^*\cF \otimes \pi_2^* \cG$. We have the following maps:
\[
\cR \boxtimes \cO_Y \to V \boxtimes \cO_Y = \cO_Y \boxtimes V \to \cO_Y \boxtimes \cQ.
\]
The composition corresponds to a section of $\cR^* \boxtimes \cQ$, whose zero locus is the diagonal $\Delta_Y$ of $Y \times_X Y$, and has codimension equal to the rank of $\cR^* \boxtimes \cQ$. Hence the following Koszul complex is exact:
\[
0 \to \bigwedge^{r(d-r)} (\cR \boxtimes \cQ^*) \to \cdots \to \bigwedge^2(\cR \boxtimes \cQ^*) \to \cR \boxtimes \cQ^* \to \cO_{Y \times_X Y} \to \cO_{\Delta_Y} \to 0.
\]
Using the Cauchy identity, we can write 
\[
\bigwedge^i(\cR \boxtimes \cQ^*) = \bigoplus_{\substack{\lambda \subseteq r \times (d-r)\\ |\lambda| = i}} \bS_{\lambda^\dagger}(\cR) \boxtimes \bS_\lambda (\cQ^*).
\]
Given $M \in \rD^b(Y)$, we have a quasi-isomorphism 
\begin{align} \label{eqn:diagonal-id}
M \simeq \rR (\pi_2)_* (\rL \pi_1^* M \otimes^\rL_Y \cO_{\Delta_Y}).
\end{align}
This is a formal verification: let $\iota \colon Y \cong \Delta_Y \to Y \times_X Y$ be the inclusion. Then
\begin{align*}
\rR (\pi_2)_* (\rL \pi_1^* M \otimes^\rL_Y \cO_{\Delta_Y}) &= 
\rR (\pi_2)_* (\rL \pi_1^* M \otimes^\rL_Y \rR \iota_* \cO_{Y})\\
&= \rR (\pi_2)_* (\rR \iota_*(\rL\iota^* \rL \pi_1^* M \otimes^\rL_Y \cO_{Y})) = M.
\end{align*}
In the second equality, we used the projection formula; in the final equality, we used that $\pi_1 \iota = \pi_2 \iota = {\rm id}_Y$. The right side of \eqref{eqn:diagonal-id} can be computed using the Koszul complex, which gives a spectral sequence
\[
\rE^1_{p,q} = \bigoplus_{\substack{\lambda \subseteq r \times (d-r)\\ |\lambda|=q}} \rR^{-p} \pi_* (M \otimes \bS_{\lambda^\dagger}(\cR)) \otimes \bS_\lambda(\cQ^*)
\]
which converges to $M$ concentrated in degree $(0,0)$. So we conclude the following:

\begin{proposition} \label{prop:grass-gen}
$\rD^b_\fgen(Y)$ is generated by objects of the form $\pi^*(\cF) \otimes \bS_\lambda(\cQ^*)$ where $\cF \in \rD^b_\fgen(X)$ and $\lambda \subseteq r \times (d-r)$.
\end{proposition}

\begin{corollary} \label{cor:grass-gen}
$\rD^b_{\fgen}(Y)$ is generated by objects of the form $\pi^*(\cF) \otimes \bS_{\lambda}(\cQ)$ where $\cF \in \rD^b_{\fgen}(X)$ and $\lambda \subseteq r \times (d-r)$.
\end{corollary}

\begin{proof}
Let $\lambda^c$ be the complement of $\lambda$ in the $r \times (d-r)$ rectangle, thought of as a partition. Then $\bS_{\lambda}(\cQ^*)$ is isomorphic to $\bS_{\lambda^c}(\cQ) \otimes \det(\cQ^*)^{\otimes r}$, and tensoring with $\det(\cQ^*)$ is an automorphism of the derived category.
\end{proof}

For each $\lambda \subseteq r \times (d-r)$, define $u_\lambda \colon \rK(X) \to \rK(Y)$ by $u_\lambda(M) = \rL \pi^*M \otimes \bS_\lambda(\cQ^*)$. Define 
\[
u \colon \bigoplus_{\lambda \subseteq r \times (d-r)} \rK(X) \to \rK(Y)
\]
as the sum $u = \sum_\lambda u_\lambda$.

\begin{corollary} \label{cor:grass-Ktheory}
$u$ is an isomorphism, so $\rK(Y) \cong \rK(X)^{\oplus \binom{d}{r}}$. In particular, if $X$ is a point, then $\rK(Y) \cong \bZ^{\oplus \binom{d}{r}}$.
\end{corollary}

\begin{proof}
For $\lambda \subseteq r \times (d-r)$, define $v_\lambda \colon \rK(Y) \to \rK(X)$ by $v_\lambda(M) = \rR\pi_*(\bS_{\lambda^\dagger}(\cR) \otimes M)$ and define $v \colon \rK(Y) \to \bigoplus_{\lambda \subseteq r \times (d-r)} \rK(X)$ using $v_\lambda$ as the components. It follows from Corollary~\ref{cor:dual-basis} that $vu$ is a diagonal matrix whose diagonals are $\pm 1$, so $u$ is injective. It follows from Proposition~\ref{prop:grass-gen} that $u$ is also surjective, so we are done.
\end{proof}

\section{Finiteness properties of resolutions} \label{ss:oldkoszul}

In this appendix, we outline an alternative, direct approach to proving finiteness properties of resolutions of finitely generated $\bA(E)$-modules in the case that $E$ is a $\bC$-vector space.

Let $C$ be a graded coalgebra with finite-dimensional components and let $N$ be a graded $C$-comodule with finite-dimensional components. We say that $N$ is {\bf finitely cogenerated} if there is a finite length quotient $N \to N'$ such that the composition $N \to N \otimes C \to N' \otimes C$ is injective. This is equivalent to saying that the graded dual of $N$ is a finitely generated module over the graded dual of $C$.

Given a module $M$ over $A$, let $M^{\le n}$ be the quotient of $M$ by the sum of all Schur functors with more than $n$ rows. 

\begin{proposition} \label{prop:truncationregularity}
Fix a  partition $\lambda$ and $n \ge 1$.
  \begin{enumerate}[\indent \rm (a)]
  \item The module $(\bS_\lambda \bC^\infty \otimes A)^{\le n}$ has finite regularity. If $\lambda_n \ge \dim E$, then the regularity is $0$, and otherwise, the regularity is at most $n(\dim E - \lambda_n - 1)$.
  \item $\ext^\bullet_A((\bS_\lambda \bC^\infty \otimes A)^{\le n}, \bC)$ is finitely generated over $\ext^\bullet_A(\bC, \bC)$.
  \end{enumerate}
\end{proposition}

\begin{proof}
Let $X$ be the Grassmannian of rank $n$ quotients of the space $\bC^\infty$. Then we have the tautological exact sequence
\[
0 \to \cR \to \bC^\infty \times X \to \cQ \to 0
\]
where $\cQ$ has rank $n$. By Theorem~\ref{thm:BWB}, for any partition $\mu$, we have $\rH^0(X; \bS_\mu \cQ) = \bS_\mu \bC^\infty$, and all higher cohomology vanishes. In particular, 
\[
\rH^0(X; \bS_\lambda \cQ \otimes \Sym(E \otimes \cQ)) = (\bS_\lambda \bC^\infty \otimes A)^{\le n}
\]
as an $A$-module. Let $\xi = E \otimes \cR$. Using \cite[Theorem 5.1.2]{weyman}, the minimal free resolution $\bF_\bullet$ of $(\bS_\lambda \bC^\infty \otimes A)^{\le n}$ is given by
\[
\bF_i = \bigoplus_{j \ge 0} \rH^j(X; (\bigwedge^{i+j} \xi) \otimes \bS_\lambda \cQ) \otimes A(-i-j).
\]
In particular, the regularity is the supremum over $j$ such that $\rH^j$ is nonzero. By \cite[Corollary 2.3.3]{weyman}, we have
\addtocounter{equation}{-1}
\begin{subequations}
\begin{align} \label{eqn:exteriorcauchy}
\bigwedge^e \xi = \bigoplus_{\substack{\mu\\ \mu_1 \le \dim E\\ |\mu| = e}} \bS_{\mu^\dagger} E \otimes \bS_\mu \cR.
\end{align}
\end{subequations}
To calculate the cohomology of $\bS_\lambda \cQ \otimes \bS_\mu \cR$, consider the sequence
\[
\alpha = (\lambda_1, \dots, \lambda_n, \mu_1, \mu_2, \dots)
\]
and define $\rho = (0, -1, -2, \dots)$. We have an action of $S_\infty$ coming from $w \bullet \alpha = w(\alpha + \rho) - \rho$. By Borel--Weil--Bott (Theorem~\ref{thm:BWB}), if there is a non-identity $w \in S_\infty$ so that $w \bullet \alpha = \alpha$, then all cohomology vanishes, and otherwise, there is a unique such $w$ so that $w \bullet \alpha$ is a partition, and the cohomology is $\bS_{w \bullet \alpha} \bC^\infty$ concentrated in degree $\ell(w)$.

If $\lambda_n \ge \dim E$, then by \eqref{eqn:exteriorcauchy}, any $\alpha$ that comes from a summand of $\bigwedge^\bullet \xi \otimes \bS_\lambda \cQ$ is a partition, so the resolution $\bF_\bullet$ is linear and we are done. Otherwise, let $i = \dim E - \lambda_n$. We will show that the cohomology of $\bS_\lambda \cQ \otimes \bS_\mu \cR$ vanishes above degree $n(i-1)$. Assume that $\alpha + \rho$ has no repeated entries, otherwise the cohomology vanishes. Then 
\[
(\alpha + \rho)_n = \lambda_n - n + 1 = \dim E - n - i + 1 \ge (\alpha + \rho)_{n+i},
\]
and the permutation $w$ that sorts $\alpha + \rho$ is in $S_{n+i-1}$. Since $w$ satisfies $w(1) < \cdots < w(n)$ and $w(n+1) < \cdots < w(n+i-1)$, its length is at most $n(i-1)$. This proves (a).

For (b), we will instead prove that $\Tor_\bullet^A((\bS_\lambda \bC^\infty \otimes A)^{\le n}, \bC)$ is a finitely cogenerated comodule over $\Tor_\bullet^A(\bC,\bC)$. From (a), we know that there are finitely many linear strands. We will focus on the $j$th linear strand. First, consider the comultiplication map
\[
\Tor_{i+k}^A((\bS_\lambda \bC^\infty \otimes A)^{\le n},\bC)_{i+k+j} \to \Tor_i^A((\bS_\lambda \bC^\infty \otimes A)^{\le n},\bC)_{i+j} \otimes \bigwedge^k (E \otimes \bC^\infty).
\]
We can rewrite this as 
\begin{align*}
\rH^j(X; \bigwedge^{i+k+j} \xi \otimes \bS_\lambda \cQ) \to \rH^j(X; \bigwedge^{i+j} \xi \otimes \bS_\lambda \cQ \otimes \bigwedge^k (E \otimes \bC^\infty)).
\end{align*}

\begin{lemma}
The above map is obtained by applying $\rH^j$ to the composition
\addtocounter{equation}{-1}
\begin{subequations}
\begin{align} \label{eqn:composition}
\bigwedge^{i+k+j} \xi \otimes \bS_\lambda \cQ \to \bigwedge^{i+j} \xi \otimes \bigwedge^k \xi \otimes \bS_\lambda \cQ \to \bigwedge^{i+j} \xi \otimes \bS_\lambda \cQ \otimes \bigwedge^k (E \otimes \bC^\infty),
\end{align}
\end{subequations}
where the first map is comultiplication, and the second map comes from the inclusion $\xi \subset E \otimes \bC^\infty$.
\end{lemma}

\addtocounter{equation}{-1}
\begin{subequations}
\addtocounter{equation}{1}
\begin{proof}
Recall that over a local ring $R$ with residue field $k$, and an $R$-module $M$, we construct the comodule structure on $\Tor^R_\bullet(M,k)$ as follows (this is a modification of Assmus' description \cite{assmus} of the coalgebra structure on $\Tor_\bullet^R(k,k)$). Let $\bF_\bullet \to M$ be an $R$-free resolution of $M$ and let $\bK_\bullet \to k$ be an $R$-free resolution of $k$. Tensoring both $\bF_\bullet$ and $\bK_\bullet$ with the residue field, we get a map
\begin{align} \label{eqn:assmus}
\bF_\bullet \otimes_R \bK_\bullet \to (\bF_\bullet \otimes_R k) \otimes_k (k \otimes_R \bK_\bullet),
\end{align}
and taking homology, and using K\"unneth's formula, this becomes
\begin{align} \label{eqn:tor-comult}
\Tor^R_\bullet(M,k) \to \Tor^R_\bullet(M,k) \otimes_k \Tor^R_\bullet(k,k).
\end{align}
Let $\cE$ be the total space of the trivial bundle $(E \otimes  \bC^\infty)^*$ over $X$.
We have a twisted Koszul complex $\cF_\bullet = \bigwedge^\bullet(\xi) \otimes \cO_\cE \otimes \bS_\lambda \cQ$ on $\cE$. Let $\cK_\bullet = \bigwedge^\bullet(E \otimes  \bC^\infty) \otimes \cO_{\cE}$ be the Koszul resolution of $\cO_X$ over $\cO_{\cE}$ (here $X$ is the zero section in $\cE$). Then we have the relative version of \eqref{eqn:assmus}
\[
\cF_\bullet \otimes_{\cO_{\cE}} \cK_\bullet \to (\cF_\bullet \otimes_{\cO_{\cE}} \cO_X) \otimes_{\cO_X} (\cO_X \otimes_{\cO_{\cE}} \cK_\bullet).
\]
Now we take the hypercohomology of both sides. Since $\cK_\bullet$ is a complex of free $\cO_\cE$-modules, this is a map of the form \eqref{eqn:tor-comult} with $M = (\bS_\lambda  \bC^\infty \otimes A)^{\le n}$. We can calculate hypercohomology of a complex of sheaves in two different ways: either first calculate cohomology (in the complex sense) and then calculate sheaf cohomology, or else calculate sheaf cohomology first and then cohomology (in the complex sense). The two different approaches form the $\rE_2$ page of a spectral sequence which converges to the hypercohomology.

If we first calculate cohomology in the sense of complexes, then we get a relative tor comultiplication map
\[
\Tor^{\cO_\cE}_\bullet(\bS_\lambda \cQ \otimes \cO_\cE, \cO_X) \to \Tor^{\cO_\cE}_\bullet(\bS_\lambda \cQ \otimes \cO_\cE, \cO_X) \otimes_{\cO_X} \Tor^{\cO_\cE}_\bullet(\cO_X, \cO_X).
\]
The maps \eqref{eqn:composition} are graded pieces of this map. Our goal is to understand the map we get by taking sheaf cohomology of both sides. Note that taking sheaf cohomology commutes with the tensor product on the right hand side since $\Tor^{\cO_\cE}_\bullet(\cO_X, \cO_X) = \bigwedge^\bullet(E \otimes \bC^\infty)$ consists of free $\cO_\cE$-modules. Taking sheaf cohomology gives us a map of the form \eqref{eqn:tor-comult}, so the spectral sequence degenerates on the $\rE_2$ page.

If we instead calculate sheaf cohomology first, then we get \eqref{eqn:assmus}. A few remarks are in order: $\cK_\bullet$ is free over $\cO_\cE$, so tensoring with it commutes with taking cohomology; the sheaves $\cF_i$ are pullbacks of sheaves $\cF'_i$ from $\cO_X$, so the sheaf cohomology of $\cF_i \otimes \cO_X$ is the same as the sheaf cohomology of $\cF'_i$ (similarly for $\cK_i$). Taking homology gives us a map of the form \eqref{eqn:tor-comult}, so again this spectral sequence degenerates on the $\rE_2$ page.

Hence both spectral sequences degenerate on the $\rE_2$ page, so we get an identification of the desired maps given that the spectral sequences are isomorphic.
\end{proof}
\end{subequations}

Recall that above we have seen that the shifted Weyl group action that we must perform to calculate the cohomology of $\bS_\lambda \cQ \otimes \bS_\mu \cR$ only depends on the first $i = \max(0,\dim E - \lambda_n - 1)$ parts of $\mu$. Since we have $\mu_1 \le \dim E$, there are only finitely many possibilities for this subpartition. Write $\mu = \mu^\circ|\nu$ where $\mu^\circ$ is the first $i$ parts of $\mu$, and $\nu$ is the rest (the symbol $|$ denotes concatenation). So from now on, we will focus only on $\bS_\lambda \cQ \otimes \bS_\mu \cR$ where $\mu^\circ$ is a fixed partition. Let $|\mu^\circ| = m$. Then in the composition
\begin{align*}
 \bS_\lambda \cQ \otimes \bigwedge^{N+m} \xi &\to \bS_\lambda \cQ \otimes \bigwedge^m \xi \otimes \bigwedge^N \xi
\to \bS_\lambda \cQ \otimes \bigwedge^m \xi \otimes \bigwedge^N(E \otimes \bC^\infty),
\end{align*}
we see the subsheaves
\begin{align*}
  \bS_\lambda \cQ \otimes \bS_{(\mu^\circ)^\dagger + \nu^\dagger} E \otimes \bS_{\mu^\circ | \nu} \cR &\to \bS_\lambda \cQ \otimes (\bS_{(\mu^\circ)^\dagger} E \otimes \bS_{\mu^\circ} \cR) \otimes (\bS_{\nu^\dagger} E \otimes \bS_\nu \cR)\\
&\to \bS_\lambda \cQ \otimes (\bS_{(\mu^\circ)^\dagger} E \otimes \bS_{\mu^\circ} \cR) \otimes (\bS_{\nu^\dagger} E \otimes \bS_\nu \bC^\infty),
\end{align*}
and this restriction is an inclusion. In fact, since we are in characteristic 0, the first map is a direct summand, so applying $\rH^j$, we still get an inclusion. The cokernel of the second map is 
\[
\sum_{\theta \subsetneqq \nu} \bS_\lambda \cQ \otimes (\bS_{(\mu^\circ)^\dagger} E \otimes \bS_{\mu^\circ} \cR) \otimes (\bS_{\nu^\dagger} E \otimes \bS_{\nu/\theta} \cQ \otimes \bS_\theta \cR),
\]
and we can see from Theorem~\ref{thm:BWB} that it will not contain the sections of the first module (namely because any partition $\zeta$ that appears in its cohomology will have $\sum_{k > n+i} \zeta_{k} \le |\theta| < |\nu|$). So applying $\rH^j$ to the composition also gives an inclusion.

So for a cogenerating set of $\bS_\lambda \cQ \otimes \bigwedge^\bullet \xi$, we take $\bS_\lambda \cQ \otimes \bigwedge^m \xi$ where $m \le (\dim E)(\dim E - \lambda_n - 1)$ (this is the largest possible size of $\mu^\circ$ as above). In particular, $\Tor^A_\bullet((\bS_\lambda \bC^\infty \otimes A)^{\le n},\bC)$ is cogenerated in (homological) degrees $\le (\dim E)\max(0,\dim E - \lambda_n - 1)$, which finishes the proof of (b).
\end{proof}

\begin{corollary} \label{cor:fourier-fg}
Every object $M \in {\rm Mod}_A$ has finite regularity and $\ext^\bullet_A(M,\bC)$ is a finitely generated $\ext^\bullet_A(\bC,\bC)$-module.
\end{corollary}

\begin{proof}
Let $M$ be an $A$-module. Let $n$ be bigger than the number of rows in the partitions that appear in the presentation of $M$. Let $\bF_\bullet$ be a finite free resolution of $M(\bC^n)$ over $A(\bC^n)$. Considered as modules over $A$, the $\bF_i$ are direct sums of modules of the form $(\bS_\lambda \otimes A)^{\le n}$. We can construct an $A$-free resolution of $M$ using a mapping cone on $\bF_\bullet$ and $A$-free resolutions on these modules. Since $\bF_\bullet$ is finite and each $\bF_i$ has finite regularity over $A$ by Proposition~\ref{prop:truncationregularity}, we conclude that $M$ has finite regularity.

For finite generation, note that the mapping cone gives us a finitely generated $\ext^\bullet_A(\bC,\bC)$-module. Removing redundancies to get a minimal resolution amounts to throwing away a direct summand.
\end{proof}

\begin{remark}
The argument above can be used to show that truncated modules over degree~2 tca's like $\Sym(\Sym^2(\bC^\infty))$ have infinite regularity in general (even after renormalizing the degrees of the generators to 1).
\end{remark}

\end{document}